\documentclass[12pt,leqno]{amsart}

\usepackage{amsmath,amscd,amsthm,amsxtra,amssymb}
\usepackage{epsfig,graphics,color,colortbl}
\usepackage{amssymb,latexsym}
\usepackage{mathrsfs}
\usepackage[poly,all]{xy}
\usepackage{marginnote}
\usepackage{xspace}
\usepackage[colorlinks=true, pdfstartview=FitV, linkcolor=blue,citecolor=blue,urlcolor=blue]{hyperref}

\usepackage{yfonts}
\usepackage{enumerate}

\usepackage[usenames,dvipsnames,svgnames,table]{xcolor}

\usepackage[normalem]{ulem}  

\usepackage[colorinlistoftodos]{todonotes}

\newdir{ >}{{}*!/-10pt/@{>}}

\allowdisplaybreaks[3]

\newlength{\mylength}
\renewcommand{\le}{\leqslant}
\renewcommand{\ge}{\geqslant}

\theoremstyle{plain}
\newtheorem{thm}{\bf Theorem}[section]
\newtheorem{df}[thm]{\bf Definition}
\newtheorem{prop}[thm]{\bf Proposition}
\newtheorem{coro}[thm]{\bf Corollary}
\newtheorem{lem}[thm]{\bf Lemma}
\newtheorem{conj}[thm]{\bf Conjecture}

\theoremstyle{definition}
\newtheorem{ex}[thm]{\bf Example}
\newtheorem{remark}[thm]{\bf Remark}
\newtheorem{definition}[thm]{\bf Definition}
\newtheorem{question}[thm]{\bf Qusetion}
\newtheorem*{convention}{\bf Convention}

\newcommand{\nc}{\newcommand}

\newenvironment{answer}
{\noindent{\bf Answer}\hs{1ex}}
{\hfill \qedsymbol}
\nc{\Prop}{\begin{prop}}
\nc{\enprop}{\end{prop}}
\nc{\Lemma}{\begin{lem}}
\nc{\enlemma}{\end{lem}}
\nc{\Exam}{\begin{ex}}
\nc{\enexam}{\end{ex}}
\nc{\Th}{\begin{thm}}
\nc{\enth}{\end{thm}}
\nc{\Def}{\begin{definition}}
\nc{\edf}{\end{definition}}
\nc{\Conj}{\begin{conj}}
\nc{\enconj}{\end{conj}}
\nc{\Quest}{\begin{question}}
\nc{\enquest}{\end{question}}
\nc{\Rem}{\begin{remark}}
\nc{\enrem}{\end{remark}}
\nc{\Ans}{\begin{answer}}
\nc{\enans}{\end{answer}}

\newenvironment{red}
{\relax\color{red}}
{\hspace*{.5ex}\relax}

\nc{\berm}{\ber{}\marginnote{\fbox{\scshape\lowercase{M}}}}
\nc{\bermn}{\ber{}\marginnote{\fbox{\scshape\lowercase{MNew}}}}
\nc{\berMH}{\ber{}\marginnote{\fbox{\scshape\lowercase{MH}}}}
\nc{\berE}{\ber{}\marginnote{\fbox{\scshape\lowercase{E}}}}

\newcommand{\cmtM}[1]{\begin{red}{}\marginnote{\fbox{\scshape\lowercase{M}}} (#1)\end{red}}  
\newcommand{\cmtm}[1]{\cmtM{#1}}
\newcommand{\cmtMH}[1]{\begin{red}{}\marginnote{\fbox{\scshape\lowercase{MH}}} (#1)\end{red}}  %

\nc{\on}{\operatorname}

\newcommand{\Q}{\mathbb {Q}}
\newcommand{\Z}{\ms{2mu}{\mathbb Z}}

\newcommand{\D}{\mathscr{D}\ms{1mu}}

\newcommand{\Pbb}{\mathrm{Pol}}

\newcommand{\one}{{\bf{1}}}
\newcommand{\seteq}{\mathbin{:=}}

\newcommand{\hd}{{\mathrm{hd}}}      					 
\newcommand{\To}[1][{\hs{0.8ex}}]{\xrightarrow{\ms{7mu}{#1}\ms{7mu}}}

\newcommand{\g}{\ms{1mu}\mathfrak{g}\ms{1mu}}
\newcommand{\n}{\mathfrak{n}}

\newcommand{\Uqm}[1][{\mathfrak{g}}]{{U_q^-(#1)}}

\newcommand{\Tor}{\operatorname{Tor}}
\newcommand{\Hom}{\operatorname{Hom}}
\newcommand{\HOM}{\on{\mathrm{H{\scriptstyle OM}}}}
\newcommand{\END}{\on{\mathrm{E\scriptstyle ND}}\ms{.1mu}}
\newcommand{\End}{\operatorname{End}}

\newcommand{\isoto}[1][]{\mathop{\xrightarrow%
[{\raisebox{.3ex}[0ex][.3ex]{$\scriptstyle{#1}$}}]%
{{\raisebox{-.6ex}[0ex][-.6ex]{$\mspace{2mu}\sim\mspace{2mu}$}}}}}

\newcommand{\tor}{{\on{tor}}}  
\newcommand{\fl}{{\on{flat}}} 

\newcommand{\Mod}{\on{Mod}}
\newcommand{\gmod}{\text{-}\mathrm{gmod}}

\newcommand{\F}{\mathscr{F}}

\newcommand{\conv}[1][]{
\underset{\raisebox{.5ex}{$\scriptstyle{#1}$}}{\mathbin{\scalebox{1.1}{$\mspace{1.5mu}\circ\mspace{1.5mu}$}}}}
\newcommand{\hconv}{\mathbin{\scalebox{.9}{$\nabla$}}}

\newcommand{\sconv}{\mathbin{\scalebox{.9}{$\Delta$}}}

\renewcommand{\Im}{\on{Im}}
\newcommand{\de}{\on{\textfrak{d}}}
\newcommand{\tEnd}{\on{E\textsc{nd}}}

\newcommand{\cmA}{\cartan}  
\newcommand{\wlP}{\mathsf{P}}   
\newcommand{\rlQ}{\mathsf{Q}}   
\newcommand{\weyl}{\mathsf{W}}  
\newcommand{\sg}{\mathfrak{S}}   
\newcommand{\Po}{\wlP}

\nc{\qQ}{Q}
\newcommand{\bQ}{\overline{\qQ}}

\newcommand{\wt}{\mathrm{wt}} 		
\newcommand{\bR}{\mathbf{k}} 		
\nc{\corp}{\bR}
\newcommand{\catC}{ \mathscr{C}}  	
\newcommand{\tcatC}{ \widetilde{\mathscr{C}}}  	

\newcommand{\lRg}[1][w]{ \tcatC_{#1} }  	
\newcommand{\dM}{ \mathsf{M }}              
\newcommand{\dC}{ \mathsf{C }}              
\newcommand{\gW}{\mathsf{W}}
\newcommand{\sgW}{\mathsf{W}^*}
\newcommand{\tf}{{\widetilde{f}}}  		
\newcommand{\te}{{\widetilde{e}}}  		
\newcommand{\tF}{\widetilde{F}}  		
\newcommand{\tE}{\widetilde{E}}  		
\newcommand{\ep}{\varepsilon}  		
\newcommand{\ph}{\varphi}  		

\newcommand{\Ht}{\mathrm{ht}} 		
\newcommand{\Rr}{\mathbf{r}} 			
\newcommand{\coR}{R} 				
\newcommand{\La}{\Lambda} 			
\newcommand{\tLa}{\widetilde{\Lambda}} 			

\nc{\Ma}{{\ms{1.5mu}\mathsf{M}}}
\nc{\Na}{\mathsf{N}}
\nc{\Xa}{\mathsf{X}}
\nc{\Ya}{\mathsf{Y}}
\nc{\Laa}{\mathsf{L}}

\newcommand{\z}[1][{\Ma}]{{z_{#1}}}

\newcommand{\zM}{{z_\Ma}}
\newcommand{\zN}{{z_\Na}}

\newcommand{\id}{\ms{2mu}{\mathsf{id}}\ms{1mu}}   				
\newcommand{\dphi}{{\phi}}   				
\newcommand{\gH}{\mathrm{H}}   				

\newcommand{\lG}{\Gamma}   					

\newcommand{\opp}{\mathrm{opp}}

\nc{\be}{\begin{enumerate}}
\newcommand{\bnum}{\be[{\rm(i)}]}
\newcommand{\bna}{\be[{\rm(a)}]}

\newcommand{\rtl}{\rlQ}

\newcommand{\etens}{\boxtimes}

\newcommand{\rmat}[1]{\ms{1mu}{\mathbf{r}}_%
{\mspace{-2mu}\raisebox{-.6ex}{${\scriptstyle{#1}}$}}}

\newcommand{\shc}{{\ms{2mu}\mathcal{C}}}

\newcommand{\tC}{\widetilde{\catC}}

\newcommand{\Ob}{\on{Ob}}

\nc{\ms}{\mspace}
\nc{\cl}{\colon}
\nc{\ro}{{\rm (}}
\nc{\rf}{{\rm )}\xspace}
\nc{\noi}{\noindent}
\nc{\bl}{\bigl(}
\nc{\br}{\bigr)}

\newenvironment{myequationn}
{\relax\setlength{\arraycolsep}{1pt}\begin{eqnarray*}}
{\end{eqnarray*}}

\newenvironment{myequation}
{\relax\setlength{\arraycolsep}{1pt}\begin{eqnarray}}
{\end{eqnarray}}

\nc{\eq}{\begin{myequation}}
\nc{\eneq}{\end{myequation}}
\nc{\eqn}{\begin{myequationn}}
\nc{\eneqn}{\end{myequationn}}

\newenvironment{myarray}[1]{\relax\setlength{\arraycolsep}{1pt}
\begin{array}{#1}}{\end{array}\relax}

\newcommand{\ba}{\begin{myarray}}
\newcommand{\ea}{\end{myarray}}

\nc{\hs}{\hspace*}
\nc{\vs}{\vspace*}
\nc{\set}[2]{\left\{{#1}\mid{#2}\right\}}
\nc{\snoi}{\smallskip\noi}
\nc{\mnoi}{\medskip\noi}
\nc{\al}{\alpha}
\nc{\rmz}{\setminus\{0\}}
\nc{\tens}
[1][]{\mathbin{\otimes_{\raise1.5ex\hbox to-.1em{}#1}}}
\nc{\vphi}{\varphi}
\nc{\ee}{\end{enumerate}}
\nc{\la}{\lambda}
\nc{\bc}{\begin{cases}}
\nc{\ec}{\end{cases}}
\nc{\qtq}[1][and]{\quad\text{#1}\quad}
\nc{\qt}[1]{\quad\text{#1}}
\nc{\dual}{{\displaystyle{\ms{1mu}\star}}}
\nc{\wle}{\preceq}
\nc{\epito}{\twoheadrightarrow}
\nc{\epiTo}[1][]{\xymatrix@C=4ex{{}\ar@{->>}[r]^-{#1}&{}}}
\nc{\Proof}{\begin{proof}}
\nc{\lan}{\langle}
\nc{\ran}{\rangle}
\nc{\ang}[1]{\lan{#1}\ran}
\nc{\QED}{\end{proof}}
\nc{\soplus}{\scalebox{.65}{\raisebox{.2ex}{$\displaystyle\bigoplus$}}}
\nc{\eps}{\varepsilon}
\nc{\supp}{\on{supp}}
\nc{\sct}{strongly commute\xspace}
\nc{\scts}{strongly commutes\xspace}
\nc{\bce}{\eta}			
\nc{\height}[1]{\on{ht}(\ms{.5mu}{#1}\ms{.5mu})}
\nc{\braid}{{\ms{1mu}\mathrm{br}}}
\nc{\gp}{\mathfrak{p}}
\nc{\wtl}{\wlP}
\nc{\ra}{real and admits an affinization}
\nc{\ras}{real and admit affinizations}
\nc{\Cor}{\begin{coro}}
\nc{\encor}{\end{coro}}
\nc{\shf}{\mathcal{F}}
\nc{\Cw}[1][{w}]{\catC_{{#1}}}
\nc{\tCw}[1][{w}]{\widetilde{\catC}_{{#1}}} 
\nc{\akew}[1][2ex]{\rule[-1ex]{#1}{0ex}}
\nc{\ake}[1][2ex]{\rule[-1ex]{0ex}{#1}}
\nc{\akete}[1][-1ex]{\rule[{#1}]{0ex}{1ex}}
\nc{\tRm}{(R\gmod)\widetilde{\mbox{$\ake[2.5ex]\akew[.9ex]$}}}
\nc{\monoTo}[1][]{\xymatrix{\ar@{>->}[r]^-{{#1}}&}}
\nc{\monoto}[1][]{\rightarrowtail}
\nc{\tX}{\widetilde{X}}
\nc{\corps}{\corp}
\nc{\tL}{\widetilde{L}}
\nc{\prtl}{\rtl_+}
\nc{\nrtl}{\rtl_-}
\nc{\tK}{\widetilde{K}}
\nc{\tep}{\widetilde\ep}
\nc{\teps}{\widetilde\ep}
\nc{\teta}{\widetilde\eta}
\nc{\ga}{\mathfrak{a}}
\nc{\scbul}{{\,\raise1pt\hbox{$\scriptscriptstyle\bullet$}\,}}
\nc{\bwr}{\mbox{\large$\wr$}}
\nc{\tR}{{\widetilde{\mathrm{R}}}}
\nc{\lS}{\mathsf{S}}
\nc{\lZ}{\mathcal{Z}}
\nc{\prolim}[1][]{\mathop{\varprojlim}\limits_{{#1}}}
\nc{\sym}{\sg}
\newcounter{myc}
\newcounter{mycc}
\nc{\txi}{\tilde{\xi}}
\nc{\rl}{\rlQ}
\nc{\sfC}{\mathsf{C}}
\nc{\cor}{{\ms{1mu}\mathbf{k}\ms{1mu}}}
\nc{\Pro}{\on{Pro}}
\newcommand{\proolim}[1][]{\ms{-1mu}\mathop{\text{``}\ms{-.5mu}\varprojlim\ms{-4mu}\text{''}}\limits_{#1}}
\nc{\hM}{\widehat{\mathsf{M}}}
\nc{\aff}{\mathrm{aff}}
\nc{\rDa}{{\mathscr{D}_\aff}}
\nc{\st}[1]{\{{#1}\}}
\nc{\W}{\mathsf{W}}
\nc{\rt}{\Delta}
\nc{\pwtl}{\wtl_+}
\nc{\rev}{{\mathrm{rev}}}

\nc{\E}[1]{{E}_{{#1}}\ms{1mu}} 
\nc{\Es}[1]{{E}^*_{{#1}}\ms{1mu}}

\nc{\Qt}{\mathscr{Q}}
\nc{\Ctr}{\mathsf{C}}
\nc{\Ctrs}{{\mathsf{C}^*}}
\nc{\Dynkin}{\Delta}
\nc{\cartan}{\mathsf{C}}
\nc{\sfc}{\mathsf{c}}
\nc{\sfa}{\mathsf{a}}
\nc{\SW}{\mathrm{K}}
\nc{\hSW}{\widehat{\mathrm{K}}}
\nc{\refl}{\mathscr{S}}
\nc{\Rre}{\mathrm{R}^{\mathrm{ren}}}
\nc{\Rpre}{\mathrm{R}^{\mathrm{ren}\;'}}
\nc{\bRre}{\ol{\mathrm{R}}^{\ms{2mu}\mathrm{ren}}}
\nc{\sfd}{\ms{1mu}\mathsf{d}\ms{1mu}}
\nc{\shm}{\mathcal{M}}
\nc{\sht}{\mathcal{T}}
\nc{\rank}{\mathrm{rank}}
\nc{\Da}{{\D}\ms{-2.8mu}\raisebox{-.35ex}{$\scriptstyle\mathrm{aff}$}}
\nc{\lDa}{\Da^{-1}}
\nc{\bchi}{{\scalebox{.9}{\mbox{$\mathscr{E}$}}}}
\nc{\bchis}{\bchi{}^{\ms{2mu}*}}
\nc{\Daf}{\mathscr{D}}
\nc{\Laf}{\mathscr{L}}
\nc{\tLaf}{\widetilde{\Laf}}
\nc{\wtaf}{\mathscr{W}\ms{-3mu}{\mathit{t}}}
\nc{\res}[1][]{\mathop\star\limits_{\raisebox{.4ex}{$\scriptstyle #1$}}\ms{2mu}}
\nc{\hchi}{\widehat{\chi}}
\nc{\convaff}{\mathop{\scalebox{1.1}{$\mspace{1.5mu}\circ\mspace{1.5mu}$}}\limits}
\nc{\cvb}{CVB\xspace}
\nc{\svelt}{essentailly samll\xspace}
\nc{\Proc}{\on{Pro}_{\mathrm{coh}}}
\nc{\sha}{\mathcal{A}}
\nc{\Ker}{\on{Ker}}
\nc{\Coker}{\on{Coker}}
\nc{\Aff}[1][z]{\on{Aff}_{\ms{1mu}#1}}
\nc{\scb}{\scalebox}
\nc{\afr}{affreal\xspace}
\nc{\epifrom}{\ms{-5mu}\xymatrix@C=3ex{{}&{}\ar@{->>}[l]}\ms{-5mu}}
\nc{\Mid}{\bigm|}
\nc{\ol}{\overline}
\nc{\bpsi}{\ol{\psi}}
\nc{\Rat}[1][z]{\on{Raff}_{\ms{1mu}#1}}
\nc{\indlim}{\varinjlim\limits}
\nc{\inddlim}{\mathop{\mbox{``{$\varinjlim$}''}}\limits}
\nc{\Rmat}{\mathrm{R}\ms{1mu}}
\nc{\Runi}{\mathrm{R}^{\mathrm{uni}}}
\nc{\Modg}{\mathrm{Modg}}
\nc{\KO}{quasi-rigid\xspace}
\nc{\hF}{\widehat{\F}}
\nc{\Modc}{\Mod_{\mathrm{coh}}}
\nc{\e}{\mathrm{e}}
\nc{\Idx}{\mathsf{\Lambda}}
\nc{\hA}{\widehat{A}}
\nc{\prood}{\mathop{\text{``}\prod\text{''}}\limits}

\nc{\hrefl}{\widehat{\mathscr{S}}}
\nc{\ev}{\mathrm{ev}}
\nc{\coev}{\mathrm{coev}}
\nc{\ihom}{\mathcal{H}om}
\nc{\tY}{\widetilde{Y}}
\nc{\tensz}{\tens[z]\ms{-3.5mu}}

\nc{\tRre}{\widetilde{\mathrm{R}}^{\mathrm{ren}}}
\nc{\dg}{\mathbf{\lambda}}
\nc{\htens}{\hconv}
\nc{\stens}{\sconv}
\nc{\Modgc}{\mathrm{Modg}_{\mathrm{coh}}}
\nc{\Aut}{\mathrm{Aut}}
\nc{\Rd}[1][\dg]{R_{#1}\gmod}
\nc{\nn}{\nonumber}
\nc{\Dual}{\mathrm{D}\ms{1mu}}
\nc{\DA}[1][A]{\ms{1mu}\mathrm{D}_{{#1}}}
\nc{\DmA}[1][A]{\ms{1mu}\Dual^{-1}_{{#1}}}
\nc{\tensa}{\tens[A]\ms{-3mu}}
\nc{\tensc}{\tens[{\ms{3mu}\cor}]\ms{-3mu}}
\nc{\bg}{{\ms{2mu}\mathrm{big}}}
\nc{\afn}{affine object\xspace}
\nc{\afns}{affine objects\xspace}
\nc{\subafn}{affine subobject\xspace}
\nc{\Afns}{Affine objects\xspace}
\nc{\dL}{\mathsf{L}}
\renewcommand{\preceq}{\preccurlyeq}
\nc{\ltens}{\otimes\limits^{\mathbb{L}}}
\nc{\epl}{\epsilon}

\numberwithin{equation}{section}
\title {Affinizations, R-matrices and reflection functors}

\author[M. Kashiwara]{Masaki Kashiwara}
\thanks{The research of M.\ Kashiwara
was supported by Grant-in-Aid for Scientific Research (B) 20H01795,
Japan Society for the Promotion of Science.}
\address[M. Kashiwara]{
Kyoto University Institute for Advanced Study, Research Institute
for Mathematical Sciences, Kyoto University, Kyoto 606-8502, Japan
\& Korea Institute for Advanced Study, Seoul 02455, Korea}
\email[M. Kashiwara]{masaki@kurims.kyoto-u.ac.jp}

\author[M. Kim]{Myungho Kim}
\address[M. Kim]{Department of Mathematics, Kyung Hee University, Seoul 02447, Korea}
\email[M. Kim]{mkim@khu.ac.kr}
\thanks{The research of M.\ Kim was supported by the National Research Foundation of
Korea (NRF) Grant funded by the Korea government(MSIP)
(NRF-2022R1F1A1076214 and NRF-2020R1A5A1016126).}

\author[S.-j. Oh]{Se-jin Oh}
\thanks{ The research of S.-j.\ Oh was supported by the Ministry of Education of the Republic of Korea and the National Research Foundation of Korea (NRF-2022R1A2C1004045).}
\address[S.-j. Oh]{Department of Mathematics,  Sungkyunkwan University,  Suwon 16419 , Korea}
\email[S.-j. Oh]{sejin092@gmail.com}

\author[E. Park]{Euiyong Park}
\thanks{The research of E.\ Park was supported by the  National Research Foundation of Korea(NRF) grant funded by the Korea government (MSIT) (RS-2023-00273425 and NRF-2020R1A5A1016126)}
\address[E. Park]{Department of Mathematics, University of Seoul, Seoul 02504, Korea}
\email[E. Park]{epark@uos.ac.kr}


\makeatletter
\@namedef{subjclassname@2020}{\textup{2020} Mathematics Subject Classification}
\makeatother


\date{February 2, 2024}

\begin{document}

\maketitle

\begin{abstract}
In this paper we establish affinizations and R-matrices in the language of 
pro-objects, and as an application, we construct 
reflection functors over the localizations of quiver Hecke algebras of arbitrary finite types. This reflection functor categorifies
the braid group action on the half of a quantum group 
and the Saito reflection.
\end{abstract}

\tableofcontents

\section{Introduction}

\emph{Affinizations} and \emph{{ R}-matrices} are one of the most powerful tools in the representation theory of quiver Hecke algebras
and affine quantum groups. 
One of the most successful applications of them is the proof of the simplicity of the head of
the tensor product of
simple objects  given in \cite{KKKO15} (cf.\  Proposition~\ref{prop:simplehd}).

R-matrices are distinguished homomorphisms between tensor products of modules, 
which measure the commutativity of tensor products.  
Affinizations of modules  help R-matrices to play their role. 
The study on affinizations and R-matrices 
gives rise to the integer invariants $\La$, $\tLa$ and $\de$, 
which have been used crucially in deriving several remarkable results 
including  monoidal categorification of cluster algebras 
(see \cite{KKKO15, KKKO18, KKOP18, KP18} and references therein). 
In the representation theory of quantum affine algebras,
 R-matrices already occupy an important position.
The generalized Schur-Weyl duality functor relates the
R-matrices in quiver Hecke algebras and the ones in quantum affine algebras in a natural way, 
and they enjoy very similar properties in their own categories. 
See also \cite{CP94, EM03, Kas02, K^3,  KKOP20} and references therein.

In the case of \emph{symmetric} quiver Hecke algebras, there exists 
a functorial construction of an affinization of a module (\cite{K^3}). Let $R$ be a symmetric quiver Hecke algebra over a base field $\bR$ and $M$ an $R$-module. As a $\bR$-vector space, the affinization $\Ma$ of $M$ is isomorphic to $\bR[z_\Ma]\otimes  M$, and the action of the generators $\e(\nu)$ and $\tau_l$ on $ \Ma$ is the same as those on $M$, but the action of $x_k$ is twisted by $z_\Ma$ (see \cite[Section 1.3]{K^3}). The indeterminate $z_\Ma$ can be understood as a monomorphism in $\End(\Ma)$. Once an affinization exists, one can construct the distinguished homomorphisms 
\begin{align*}
 \Rre_{\Ma,\Na} & := (z_\Ma - z_\Na )^{-s}\  \Runi_{\Ma,\Na} \in \Hom ( \Ma \conv \Na,   \Na \conv \Ma ), \\
\Rr_{M,N} & :=  \Rre_{\Ma,\Na}|_{z_\Ma=z_\Na=0} \in  \Hom ( M \conv N,  N \conv M ),
\end{align*}
where $\Runi_{\Ma,\Na}$ is the \emph{universal} R-matrix defined by the intertwiners $\varphi_i$ (see Section \ref{subSec: UR}) and $s$ is the largest integer such that 
$\Runi_{\Ma,\Na} (\Ma \conv \Na) \subset(z_\Ma - z_\Na )^s  \Na \conv \Ma $.  
The \emph{renormalized} R-matrix $ \Rre_{\Ma,\Na}$ commutes with $z_\Ma$ and $z_\Na$,and the R-matrix $\Rr_{M,N}$ never vanishes.

Unlike the case of symmetric quiver Hecke algebras,  the existence of affinizations over a \emph{non-symmetric} quiver Hecke algebra is not guaranteed,
which causes a big difficulty in the study of the representations
over non-symmetric quiver Hecke algebras 
such as \emph{monoidal categorification in the non-symmetric case}.  
 In \cite{KP18}, the notion of affinizations and R-matrices 
for \emph{arbitrary} quiver Hecke algebras 
was studied.
It was shown that, if an affinization exists, the R-matrices can be constructed, and they enjoy very similar properties to those in the symmetric case. 
Affinizations and renormalized R-matrices allow us to consider a \emph{duality datum} (see \cite[Section 4.1]{KP18} for the definition) over the quiver Hecke algebra setting, which provides the generalized Schur-Weyl duality between quiver Hecke algebras (\cite[Section 4]{KP18}). This duality can be understood as an analogue of the generalized Schur-Weyl duality from quiver Hecke algebras to quantum affine algebras introduced in \cite[Section 3]{K^3}.

 In the case of the \emph{localization} of module categories over
quiver Hecke algebras, the theory of affinization becomes more complicated. 
 Let $R\gmod$ be the category of finite-dimensional graded modules over an \emph{arbitrary} quiver Hecke algebra, and let $\weyl$ be the Weyl group.  For any element $w \in  \weyl$, let $\catC_w$ be the full subcategory of $R\gmod$ 
generated by the dual PBW vectors corresponding to $w$ (see \cite{KKKO18, KKOP18}). Note that the subcategory $\catC_w$ categorifies the quantum unipotent coordinate ring $A_q(\n(w))$ and $\catC_w$ gives a monoidal categorification of $A_q(\n(w))$ as a quantum cluster algebra (\cite{KKKO18}). It was shown in \cite{KKOP21} that  $\catC_w$ admits a localization $\tcatC_w$ via a \emph{real commuting family} of central objects. 
In the language of cluster algebras, the central objects correspond to the frozen variables and the localization $\tcatC_w$ categorifies the localization of the quantum cluster algebra $A_q(\n(w))$  at the frozen variables. It turns out that the localization $\tcatC_w$ is a rigid monoidal category (\cite{KKOP21,KKOP22}).

The categorical structure of the localization $\tcatC_w$ is much more complicated than the original category $\catC_w$. 
Let $\Modc(R)$ be the category of finitely generated $R$-modules, and let
$\catC^\bg_w$ be the full subcategory of $\Modc(R)$ consisting of modules whose simple subquotients belong to $\catC_w$. 
Then, we can localize $\catC^\bg_w$
similarly to $\catC_w$ and obtain the localization $\tcatC^\bg_w$.
The problem is that  the dual of an affinization of a simple module in
$\catC_w$
does not belong to $\tcatC^\bg_w$ in general.
Thus it is not valid to translate directly the classical notion of  affinizations and R-matrices  to the localization $\tcatC_w$. 
In order to overcome this difficulty, we have to introduce the category 
$\Pro(\tcatC_w)$ of pro-objects of $\tcatC_w$. Then the dual of
an affinization belongs to $\Pro(\tcatC_w)$.
This categorical approach to affinizations and R-matrices may lead us to new applications in various interesting categories including a new kind of generalized Schur-Weyl duality.

\smallskip

In this paper, we establish the theory of 
affinizations and R-matrices 
in the language of pro-objects, and as  an application, we construct 
\emph{reflection functors} over the localizations of quiver Hecke algebras 
of an \emph{arbitrary finite type}.
The main results of the paper can be summarized as follows:
\bnum
\item The first main result is to study the notion of affinization and R-matrices in the language of pro-objects. We define affinizations and R-matrices in
the categories with certain conditions, and investigate their properties. We then introduce the invariants $\Daf$, $\La$ and $\de$ and prove various properties of them.
Thus, in a purely categorical setting,  
we have recovered  key properties appeared in quiver Hecke algebras and quantum affine algebras. We next introduce a \emph{duality datum} at the category level, and construct a generalized Schur-Weyl duality in the category setting by applying the argument given in \cite{K^3}.

\item 
The second main result is to apply the generalized Schur-Weyl duality 
to the localized category $\tcatC_{s_i w_0}^*$ of a quiver Hecke algebra of \emph{arbitrary finite type}.  As a result, for any $i\in I$, we obtain 
an equivalence of monoidal categories 
$$
\refl_i\cl \tcatC_{s_iw_0} \isoto \tcatC^*_{s_iw_0},
$$
which categorifies
the braid group action (\cite[Chapter 37]{Lu93}) and 
the \emph{Saito reflection} $\sigma_i$ (\cite{Saito}).

In the construction, the \emph{affinizations} (or the \emph{lifts}) of the invariants $\La$, $\de$, $ \tLa$, etc.\ take a crucial role (see Section \ref{Section: RIA}). 

We conjecture that $\refl_i\cl \tcatC_{s_iw_0} \buildrel \sim \over \longrightarrow \tcatC^*_{s_iw_0}$ induces an equivalence of categories $\catC_{s_iw_0}  \buildrel \sim \over \longrightarrow \catC^*_{s_iw_0}$ at the level of module categories  (see Conjecture \ref{Conj: refl}).

In the case of finite ADE type, S.\ Kato (\cite{Kato14}) constructed the Saito reflection functor using geometry, which categorifies the braid group action on the half of a quantum group (see also \cite{Kato20, McNamara17}). 
We conjecture that $\refl_i$ coincides with the functor induced
by the functor constructed by S.\ Kato. 
However, in the case of \emph{non-symmetric type}, there is no relevant result as far as the authors know. 
\ee

\smallskip

Let us explain the main results of the paper in more details.

The first main result is about the notion of affinization and R-matrices in a general category. 
We shall explain it in three parts, i.e., affinization in categories,  affinizations with respect to $\cor[z]$ and R-matrices, 
and generalized Schur-Weyl duality.

\mnoi
{\bf Affinization in categories.}\ 
Let $\shc$ be a $\cor$-linear category satisfying \eqref{cond:fcat} and let 
$A\seteq\soplus_{k\in \Z} A_k$ be a commutative graded $\cor$-algebra satisfying \eqref{cond:gring}. 
We first define the subcategory $\Proc(A,\shc)$ consisting of graded \emph{coherent} $A$-modules in the category $\Proc(\shc)$ of all pro-objects of $\shc$ (see Definition~\ref{Def: Pro_coh} for the precise definition). We investigate the category $\Proc(A,\shc)$ and obtain several homological properties under the condition when objects are $A$-flat (see Section \ref{Sec: aff in abelian}). 
We denote by $\Aff[A](\shc)$ the full subcategory of $\Proc(A,\shc)$ consisting of $A$-flat objects (Definition \ref{def:aff}). 
For a given anti-equivalence $\Dual$ of categories, we show that $\Dual$ can be lifted to the category $\Proc(A,\shc)$, denoted by $\DA$.
Then, the functor $\DA$ gives an equivalence of categories
between $\Aff[A](\shc)$ and its opposite category (see Section \ref{subSec: duality}).
In natural and interesting situations, $\Dual$ can be understood as a duality functor in the category $\shc$.

Affinizations in monoidal categories are also studied in details. We assume further that $\shc$ is an abelian $\cor$-linear graded monoidal category satisfying \eqref{cond:exactmono}.  
For graded $A$-modules $\Ma$ and $\Na$ in $ \Pro(\shc)$, we define a tensor product $\Ma\tensa\Na\in\Modg(A, \Pro(\shc))$ by the universal property \eqref{Eq: ten_A}. Lemma \ref{lem:Affmon} says that the category $\Aff[A](\shc)$ with $\tensa$ forms a monoidal category, and  Proposition \ref{Prop:rigid} tells us that $\Aff[A](\shc)$ is rigid if $\shc$ is a rigid monoidal category.

\mnoi
{\bf Affinizations with respect to $\cor[z]$ and R-matrices.}\ 
In the case where $A = \cor[z]$, we take one more step further. This case exactly appears in the quiver Hecke algebras.
We assume that $\shc$ is a $\cor$-linear category satisfying \eqref{cond:fcat} and $A = \cor[z]$, where $z$ is an indeterminate. We define an \emph{\afn} in $\shc$ as a pair $(\Ma, z)$ belonging to $\Proc(\cor[z],\shc)$ with the conditions (a)--(d) in Definition \ref{Def: affine objects}. 
We then study properties of \afns.
Let $\Aff(\shc)$ denote the category of \afns, and $\Rat(\shc)$ the category of $\Proc(\cor[z],\shc)$ modified by making the morphism $z$ invertible (see Definition \ref{Def: Raff} for the precise definition). We show that there is a canonical functor 
$$
\Aff(\shc) \longrightarrow \Rat(\shc),
$$
which is faithful and essentially surjective, and prove several properties (see Section \ref{suSec: affine obj}). 

We further assume that $\shc$ is an abelian $\cor$-linear graded monoidal category satisfying \eqref{cond:exactmono}.  
For $(\Ma, z_\Ma)$ and $(\Na, z_\Na)$ in $\Aff(\shc)$, 
let $M = \Ma/ z_\Ma \Ma$ and $N = \Na/ z_\Na \Na$.
Under certain assumptions (see Proposition \ref{pro:onedimhom}), we prove that   
$$
\HOM_{\cor[\z,\zN]}(\Ma\tens\Na,\Ma\tens\Na)\simeq\cor[\z,\zN]\id_{\Ma\tens \Na},
$$
and there exists $\Rre_{\Ma,\Na}\in\HOM_{\cor[\z,\zN]}(\Ma\tens\Na,\Na\tens \Ma)$
such that
$$
\HOM_{\cor[\z,\zN]}(\Ma\tens\Na,\Na\tens \Ma)
\simeq\cor[\z,\zN]\Rre_{\Ma,\Na}
$$
and $\Rre_{\Ma,\Na}\vert_{\z=\zN=0}\in\HOM_\shc(M\tens N, N\tens M)$ does not vanish. 
The morphism $\Rre_{\Ma,\Na}$ is called the \emph{renormalized R-matrix}, and 
the morphism $\rmat{M,N} := \Rre_{\Ma,\Na}\vert_{\z=\zN=0}$ is called the \emph{R-matrix} between $M$ and $N$ (see Definition \ref{def: rmat}). The above isomorphisms are typical properties which affinizations and R-matrices should have. We then introduce a \emph{rational center} in $\shc$ using the category $\Rat(\shc)$.
Using the notions of \afns and rational center, 
we finally define an \emph{affinization} $\Ma$ of $ M\in \shc$ (see Definition \ref{def:affinization}). 
Proposition \ref{pro:canoaff} explains that this definition 
is a generalization of the definition for affinizations (see \eqref{eq:oldaff}) given in \cite[Definition 2.2]{KP18}. 
In this categorical setting, 
the condition being a rational center makes R-matrices  satisfy the Yang-Baxter equation (Lemma \ref{lem:YB}).  
We prove several properties of affinizations and $\Rre_{\Ma,\Na}$ and define the invariants in the language of categories: 
\begin{equation} \label{Eq: D, La, de}
\begin{aligned} 
	&\Daf(\Ma,\Na) \in \cor[\z,z_{\Na}]  \qt{such that} \quad  \Rre_{\Na,\Ma} \circ \Rre_{\Ma,\Na}  = \Daf(\Ma,\Na) \id_{\Ma\tens \Na},  \\
	&\La(M,N)\seteq\deg(\rmat{M,N})\in\Z,\\
	&\de(M,N) \seteq \dfrac{1}{2}\bl\La(M,N)+\La(N,M)\br=\deg(\Daf(\Ma,\Na))/2 \in \dfrac{1}{2}\Z_{\ge 0}. 
\end{aligned}
\end{equation}

We next introduce the notion of \emph{\KO} categories. Note that the module categories of quiver Hecke algebras and quantum affine algebras are examples of \KO categories (see \cite{KKKO15}). 
The quasi-rigidity plays an important role in proving the simplicity of the head of a tensor product of simples.  As a consequence of the quasi-rigidity,  
the space of morphisms between various tensor products are one-dimensional (see Proposition~\ref{prop:simplehd} and Corollary~\ref{cor:dimEND}) which guarantees that the renormalized R-matrices  exist.  
It is conjectured that a \KO monoidal category which satisfies \eqref{cond:exactmono}
is embedded into a rigid monoidal category (Conjecture \ref{Conj: KO}).
Under the additional assumption that $\shc$ is \KO, we prove various properties of R-matrices and the invariants $\La$ and $\de$. Thus, in the purely categorical language,  we have recovered the key properties appeared in quiver Hecke algebras and quantum affine algebras.

\mnoi{\bf Generalized Schur-Weyl duality.}
\ 
Let $\shc$ be a graded \KO monoidal category satisfying \eqref{cond:g}, and let $R$ be the quiver Hecke algebra associated with parameter polynomials $Q_{i,j}(u,v)$.
We define a \emph{duality datum} $\bl\st{(\hSW_i,z_i)}_{i\in I},\st{\Rre_{\hSW_i,\hSW_j}}_{i,j\in I}\br$ consisting of 
a family $\st{(\hSW_i,z_i)}_{i\in I}$ of affinizations in $\shc$ satisfying \eqref{eq:SW} and the renormalized R-matrices $\Rre_{\hSW_i,\hSW_j}\cl\hSW_i\tens \hSW_j\to \hSW_j\tens \hSW_i$ satisfying \eqref{eq:normR}.
Applying the argument given in \cite{K^3, KP18} to the duality datum $\bl\st{(\hSW_i,z_i)}_{i\in I},\st{\Rre_{\hSW_i,\hSW_j}}_{i,j\in I}\br$, we obtain a canonical right exact monoidal functor
$$\hF\cl\Modgc(R_\dg)\to\Pro(\shc)$$
such that
\begin{align*}
	&\hF(\tL(i)_{z_i}) \simeq\hSW_i\qt{and}\quad  \hF(L(i))\simeq \SW_i,  \\
	&\hF(e(i,j)\vphi_1)=\Rre_{\hSW_i,\hSW_j}\in \HOM_{\cor[z_i,z_j]}(\hSW_i\tens \hSW_j,\hSW_j\tens \hSW_i), 
\end{align*}
where $\varphi_1$ is the intertwiner given in \eqref{def:intertwiner}, and $L(i)$ be the 1-dimensional simple $R(\al_i)$-module of degree $0$. Then we have the restricted functor
$$
\F\cl R\gmod \to\ \shc,
$$
and prove that this functor enjoys the same good properties as the usual generalized Schur-Weyl duality (see Section \ref{subSection: SW}).

\smallskip

The second main result is the reflection functor  $\refl_i\cl \tcatC_{s_iw_0} \buildrel \sim \over \longrightarrow \tcatC^*_{s_iw_0}$ over the localizations, which categorifies the braid group action on the quantum group. We shall explain it in two parts, i.e., affinizations of the invariants and reflection functors.

\mnoi
{\bf Affinizations of the invariants.}\ 
We lift the integer-valued invariants 
$\La$, $\tLa$, $\de$, $\wt$, $\ep_i$ etc.\ 
which are used crucially in quiver Hecke algebras, 
to functions by using affinizations. 

Let $(\Ma, z_\Ma)$ and $(\Na, z_\Na)$ be affinizations of simple $R$-modules $M$ and $N$ respectively. We assume that $M$ and $N$ are real for convenience. 
Let us recall $\chi_i(\Ma)$, $\bchi_i(\Ma)$ and $\bchis_i(\Ma)$ defined in \cite[Section 3.1] {KKOP21} (see \eqref{Eq: lift of chiE} for the precise definitions). By properties (see Lemma \ref{lem:endpreaff} for example) of affinizations, $\chi_i(\Ma)$, $\bchi_i(\Ma)$ and $\bchis_i(\Ma)$ can be viewed as elements in $\cor[t_i,\z]$. In Definition \ref{def:inv}, we define $\Daf(\Ma,\Na) $ (see \eqref{Eq: D, La, de}) and $\tLaf(\Ma,\Na)$ in $  \cor[\z,\zN]$ using the renormalized R-matrix $\Rre_{\Ma,\Na}$. We then define  $\Laf(\Ma,\Na)$ and $\wtaf(\Ma,\Na)$ in $\cor(x_\Ma, z_\Na)/ \cor^\times$ in terms of $\chi_i(\Ma)$, $Q_{i,j}(t_i, t_j)$, $\tLaf(\Ma,\Na)$. 
It turns out that these new invariants can be understood as lifts of the well-known integer invariants  in the following sense:
\eqn
\deg\bl\chi_i(\Ma)\br&=&n_i(\al_i,\al_i)\qt{where 
	$\wt(M)=-\smash{\sum\nolimits_{i\in I}}n_i\al_i$,}\\
\deg\bl\bchi_i(\Ma)\br&=&\eps_i(M)(\al_i,\al_i)=2\tLa(L(i),M),\\
\deg\bl\bchis_i(\Ma)\br&=&\eps^*_i(M)(\al_i,\al_i)=2\tLa(M,L(i)),\\
\deg\bl\Daf(\Ma,\Na)\br&=&2\de(M,N),\\
\deg\bl\tLaf(\Ma,\Na)\br&=&2\tLa(M,N),\\
\deg\bl\Laf(\Ma,\Na)\br&=&2\La(M,N),\\
\deg\bl\wtaf(\Ma,\Na)\br&=&2\bl\wt(M),\wt(N)\br,
\eneqn
where $\deg$ denotes the homogeneous degree. We investigate relations between the lifted invariants and prove several  interesting identities which are also interpreted as lifts of the identities between the integer invariants (see Section \ref{subsec:Inv}). The identities between lifted invariants are described in terms of the \emph{resultant algebra} (see Section \ref{sec:resultant}). These identities are used crucially in the construction of the reflection functors.

\mnoi
{\bf Reflection functors.}\ 
We first investigate the canonical functor  
$$
\Phi_w \cl \catC_w \to \tcatC_w  \qquad (w\in \weyl)
$$
and prove that the functor $\Phi_w$ is fully faithful (Theorem \ref{th:ff}). 
We thus show that the full subcategory $\catC_w$ of $\tcatC_w$ is stable by taking subquotients. However  $\catC_w$ is not stable by taking extensions in $\tcatC_w$ in general (see Remark \ref{rem:ext}). 

Let $R$ be a quiver Hecke algebra of \emph{arbitrary finite} type, and consider the categories $ \catC_{s_iw_0}$ and $\catC^*_{s_iw_0}$
(see \eqref{def:Cw}). Fix an index $i \in I$. For any $j\in I$, we take the real simple objects in $\tcatC^*_{s_iw_0}$ as follows: 
\eqn
\SW_j&\seteq
\bc
\D\bl\Qt_{s_iw_0}^*(\ang{i})\br &\text{if $j=i$,}\\
 \ang{i^{-\sfc_{i,j}}}\hconv \ang{j}&\text{if $j\not=i$,}
\ec
\eneqn
where $\D$ is the right dual functor in $\tcatC^*_{s_iw_0}$, $\langle i^m \rangle $ is the self-dual simple $R(m\al_i)$-module, and $\Qt_{s_iw_0}^*\cl R\gmod \to \tcatC^*_{s_iw_0}$ is the  localization functor (see Section \ref{subSec: Loc}).
We then take the canonical affinizations $\hSW_j$ of $\SW_j$ (see \eqref{Eq: aff hK}), which yield a duality datum 
$$
\bl\st{(\hSW_j,z_j)}_{j\in I},\st{\Rre_{\hSW_j,\hSW_k}}_{j,k\in I}\br.
$$
In particular,  we take $\hSW_i\seteq\Da\bl\Qt_{s_iw_0}^*(R(\al_i))\br$
as the affinization of $\SW_i$ by using the duality functor $\Da$ developed in 
\S\;\ref{Sec: duality}. 
We then compute the invariants $\La(\SW_j, \SW_k) $ and $\de(\SW_j, \SW_k)$ (Proposition \ref{prop:LaKK}),  and the lifted invariants $\Daf(\hSW_j, \hSW_k)$ (Proposition \ref{Prop: De}). Applying the generalized Schur-Weyl duality to this setting, we obtain the duality functor
$ \F_i\cl	R\gmod \longrightarrow \tcatC^*_{s_iw_0}$. Investigating the image of determinantial modules under the functor $\F_i$, we finally construct a functor $\refl_i:  \tcatC_{s_iw_0} \longrightarrow \tcatC^*_{s_iw_0}$ such that the following diagram quasi-commutes (Proposition \ref{Prop: factor th}):
$$\xymatrix@C=8ex{
	R\gmod\ar[d]_{\Qt_{s_iw_0}}\ar[dr]^{\F_i}\\
	\tcatC_{s_iw_0}\ar[r]_{\refl_i}& \tcatC^*_{s_iw_0}.
}$$
Lemma \ref{lem:FiM} and Proposition \ref{Prop: cat of saito refl} say that the functor $\refl_i$ categorifies the braid group action and the Saito reflection.
Theorem \ref{Thm: relf equi} says that the reflection functor $\refl_i$ gives an equivalence of categories.  By technical reasons,  we separately deal with the case of type $A_2$ in Remark \ref{rem:A2}.

\smallskip 
This paper is organized as follows. 
In Section \ref{Sec: Preliminary}, we recall necessary mathematical backgrounds for categories.  
In Section \ref{Sec: aff in abelian} and Section \ref{Sec: duality}, we study affinizations in abelian categories under various conditions.   
In Section \ref{Sec: coh in mon}, we investigate affinizations in monoidal categories under various conditions, and Section \ref{Sec: R-matrices} explains R-matrices in categories.  
In Section \ref{Sec:QHSW}, we study a generalized Schur-Weyl duality, and  
in Section \ref{Sec: localizations} we prove that the canonical functor $\Phi_w \cl \catC_w \to \tcatC_w $ is fully faithful. 
In Section \ref{Section: RIA}, we develop the lifted invariants using affinizations and Section \ref{Sec: reflection} is devoted to reflection functors over localizations.

\mnoi
{\bf Acknowledgments.}\ 
The authors thank Haruto Murata for valuable questions.
The second, third and fourth authors gratefully acknowledge for the hospitality of RIMS (Kyoto University) during their visit in 2023.
\vskip 2em

\begin{convention}\bnum
\item For a statement $P$, $\delta(P)$ is $1$ or $0$ whether $P$ is true or not.
\item Unless otherwise stated, a module over a ring is a left module.
\item $\cor$ denotes a base field.
\item For a ring $A$, we denote by $\Mod(A)$ the category of $A$-modules.
We denote by $\Modc(A)$ the category of coherent $A$-modules.
\item
For a field $\cor$ and a graded $\cor$-algebra $A$, 
we denote by $A\gmod$ 
the category of 
finite $\cor$-dimensional graded $A$-modules.
\item For a graded module $X=\soplus_{m\in\Z}X_m$, we write
$X_{\ge m}=\soplus_{k \ge m}X_k$, $X_{\le m}=\soplus_{k\le m}X_k$ and $X_{>m}=\soplus_{k>m}X_k$.
\item For a commutative $\cor$-algebra $A$ and $f(z)=\sum_{0\le k\le m}a_kz^k\in A[z]$,
we say that $f(z)$ is a {\em monic} polynomial of degree $m$ if $a_m=1$,
and that $f(z)$ is a {\em quasi-monic} polynomial of degree $m$ if $a_m\in\cor^\times$.
\item For a category $\shc$, $\shc^\opp$ denotes the opposite category of $\shc$.
\item We say that a category $I$ is {\em filtrant} if $I$ satisfies the conditions
\bna
\item $\Ob(I)\not=\emptyset$,
\item for any objects $i,j\in I$, there exist $k\in I$ and morphisms $i\to k$ and $j\to k$,
\item for any morphisms $f\cl i\to j$ and $g\cl i\to j$, there exists a morphism $h\cl j\to k$ such that $h\circ f=h\circ g$.
\ee
We say that $I$ is {\em cofiltrant} if $I^\opp$ is filtrant. 
\ee
\end{convention}

\vskip 1em 

\section{Preliminaries} \label{Sec: Preliminary}

\subsection{Module objects}
For a ring $R$ and an additive category $\shc$, let us denote by $\Mod(R,\shc)$ the category of \emph{$R$-modules} in $\shc$, 
i.e., the category of objects $X\in \shc$ endowed with a ring homomorphism $R \to \End_\shc(X)$.  
The morphisms between an $R$-module $X$ and an $R$-module $Y$ are the morphisms in $\shc$ which commute with the $R$-actions.

Assume further that $\shc$ is an abelian category. 
We can define the tensor functor
\eqn
\bullet \tens_R \bullet \cl \Modc(R^\opp) \times \Mod(R,\shc) \to \shc,
\eneqn
which is characterized by the universal property
\eqn 
\Hom_{\shc}(M \tens_R X, Y) \simeq \Hom_{\Mod( R^{\opp} )}\bl M, \Hom_\shc(X,Y)\br
\eneqn
functorially in $X\in\Mod(R,\shc)$,
$M\in \Modc(R^\opp)$ and $Y \in \shc$ (see \cite[Remark 8.5.7]{KS06}).
The functor $\scbul \tens_R \scbul$ is right exact in each variable.

Similarly, we have a functor
$$\ihom_R(\scbul,\scbul)\cl \Modc(R)^\opp \times \Mod(R,\shc) \to \shc$$
characterized by the universal property
\eqn 
\Hom_{\shc}\bl X,\ihom_R(M, Y)\br \simeq 
\Hom_{\Mod( R )}\bl M, \Hom_\shc(X,Y)\br
\eneqn 
for $M \in \Modc(R)$, $Y\in \Mod(R,\shc) $, $X \in \shc$,
and $\ihom_R(\scbul, \scbul)$ is left exact in each variable.

Assume that $R^\opp$ is noetherian.
We say that an $R$-module $X$ in $\shc$ has \emph{$R$-flat dimension $\le m$} if 
$\Tor^R_i(M,X)=0$ for any $i> m$ and $M\in \Modc(R^\opp)$, 
where $\Tor^R_i(\scbul, X)$ denotes the left derived functor of $\scbul \tens[R]X\cl\Modc(R^\opp)\to\shc$. 
An $R$-module $X$ in $\shc$ is called \emph{$R$-flat} if its $R$-flat dimension is $\le 0$, i.e.\ if $\scbul\tens[R]X\cl \Modc(R^\opp)\to\shc$ is an exact functor.

\subsection{Graded categories}\label{subsec:graded}
Let $\shc$ be an additive category.
It is called a graded category if $\shc$ is endowed with
an auto-equivalence $q$
(called the \emph{grading shift functor}).
For an additive graded category $\shc$,
we set
$$\HOM_\shc(M,N)\seteq\soplus_{k\in \Z}\HOM_\shc(M,N)_k\qt{
with }\HOM_\shc(M,N)_k\seteq \Hom_\shc(q^k M,N)$$
for  $M,N\in \shc$.
Then $\shc$ has a category structure with $\HOM_\shc$ as
a set of morphisms.
We say that $f\in\HOM_\shc(M,N)_k$ is a morphism of degree $k$.

For a graded ring $R=\soplus_{k\in\Z}R_k$, 
a graded $R$-module in $\shc$ is an object $X$ endowed with a graded
ring homomorphism
$R\to\END_\shc(X)\seteq\HOM_\shc(X,X)$.
The category of graded $ R$-modules in $\shc$ is an additive graded category
and  denoted by $\Modg(R,\shc)$.
The morphisms between two graded-$R $-modules in $\shc$ are the morphisms in 
$\shc$ which commute with the $R$-actions.

If $\shc$ is abelian, we can define
\eqn\scbul\tens[R]\scbul&&\cl\Modgc(R^\opp)\times \Modg(R,\shc)\To \shc
\qtq\\\ihom_R(\scbul,\scbul)&&\cl\Modgc(R)^\opp\times \Modg(R,\shc)\To\shc.
\eneqn
Here $\Modgc(R)$ is the category of coherent graded $R$-modules.

\smallskip
In the sequel, we sometimes {\em neglect grading shifts}.
For example, we sometimes write that $f\cl M\to N$ is a morphism when $f\in\HOM(M,N)_k$.

\subsection{Pro-objects}
For a category $\shc$, let $\Pro(\shc)$ be the category of pro-objects of $\shc$ (see \cite[Section 6.1]{KS06} for the definition and its properties).
Then there is a fully faithful functor $\shc \to \Pro(\shc)$ and we regard $\shc$ as a full subcategory of $\Pro(\shc)$.
The category $\Pro(\shc)$ admits small cofiltrant  projective limits,  which we denote by
$\proolim$ (\cite[Theorem 6.1.8]{KS06}).
Every object in $\Pro(\shc)$ is isomorphic to 
the projective limit $\proolim[i\in I] X_i$ of a small cofiltrant projective system $\{X_i\}_{i\in I}$ in $\shc$. 
Let $\cor$ be a field. 
Assume that $\shc$ is an abelian (respectively, $\cor$-linear abelian) category.
Then the  category $\Pro(\shc)$  is also an  abelian (respectively, $\cor$-linear abelian) category,  and the canonical functor $\shc \to \Pro(\shc)$ is exact.
Moreover the functor $\proolim[i\in I] $  is exact
(\cite[Theorem 6.1.19]{KS06}).
Namely, for any exact sequence of  cofiltrant projective systems
$X_i\to Y_i\to Z_i$ in $\Pro(\shc)$ indexed by a small cofiltrant category $I$,
$\proolim[i\in I]X_i\to \proolim[i\in I]Y_i\to \proolim[i\in I]Z_i$
is exact in $\Pro(\shc)$. 

The subcategory $\shc$ of $\Pro(\shc)$ is closed by taking 
kernels, cokernels and extensions (\cite[Proposition 8.6.11]{KS06}).

If $\shc$ is an essentially small abelian (respectively, $\cor$-linear abelian) category, then  the category $\Pro(\shc)$ is equivalent to the opposite category of the
category of left exact functors  from $\shc$ to the category of abelian groups (respectively, $\cor$-vector spaces). 

If $\shc$ is an abelian monoidal category with a bi-exact tensor functor $\tens$, then so is $\Pro(\shc)$ with the natural extension of $\tens$.

\begin{lem} \label{lem:art}
Let $\shc$ be an abelian category, and let $X$ be an artinian object of $\shc$. Then any quotient and subobject of $X$ in $\Pro(\shc)$ are isomorphic to objects in $\shc$.
\end{lem}
\begin{proof}
Let $Y$ be an subobject of $X$ in $\Pro(\shc)$. Then we can write $Y \simeq \proolim[i\in I] Y_i$, where $\{Y_i\}_{i\in I}$ is a projective system of subobjects of $X$ indexed by a small cofiltrant category $I$. 

\smallskip
Since $X$ is artinian, there exists $i_0$ in $I$ such that $Y_i\isoto Y_{i_0} $ for any $i \to i_0$.
Hence $Y\isoto Y_{i_0}$ belongs to $\shc$.
If $Z$ is a quotient of $X$ in $\Pro(\shc)$, then $\Ker(X\to Z) \in \shc$ and hence $Z\in \shc$.
\end{proof}

Let $F\cl \shc\to\shc'$ be a functor.
Then $F$ naturally induces a functor $\Pro(\shc)\to\Pro(\shc')$.
{\em We denote this functor by the same letter $F$.} 

\section{Affinizations in abelian categories} \label{Sec: aff in abelian}

\subsection{Coherent objects in $\Pro(\shc)$} 

Let $\cor$ be a base field. Let $\shc$ be a $\cor$-linear category such that 
\eq&&\hs{2ex}\left\{
\parbox{75ex}{\be[{$\bullet$}]
\item $\shc$ is abelian,
\item $\shc$ is $\cor$-linear graded, namely $\shc$ is endowed with a $\cor$-linear auto-equivalence $q$,

\item any object has a finite length,
\item  any simple object $S$ is absolutely simple, i.e., $\cor\isoto\END_\shc(S)
\seteq\HOM_\shc(S,S)$.\ee
}\right.
\label{cond:fcat}
\eneq
It follows from \eqref{cond:fcat} that $\HOM_{\shc}(M,N)$ is finite-dimensional over $\cor$ for any $M,N \in \shc$ and
that $S\not \simeq q^kS $ for any $k\in \Z\setminus\{0\}$ and any non-zero 
$S \in \shc$.

By Lemma \ref{lem:art}, the full subcategory $\shc$ of $\Pro(\shc)$ is stable by taking subquotients.
The grading shift functor is extended to $\Pro(\shc)$. 

Let $A=\soplus_{k\in \Z} A_k$ be a commutative graded $\cor$-algebra such that 
\eq&&\left\{
\parbox{70ex}{
\be[$\bullet$]
\item $A_0\simeq\cor$,
\item $A_k=0$ for any $k<0$,
\item $A$ is a finitely generated $\cor$-algebra.
\ee}\right.
\label{cond:gring}
\eneq
In particular, $A$ is a noetherian ring and $\dim_\cor A_k < \infty$ for all $k$.
For each $m\in \Z$, set
\eqn 
A_{\ge m} \seteq\soplus _{k\ge m} A_k, 
\eneqn
which is an ideal of $A$.

{\em Hereafter, we assume that
$\shc$ satisfies \eqref{cond:fcat} and $A$ satisfies \eqref{cond:gring}.}

Recall that $\Modg(A, \Pro(\shc))$ denotes
the category of graded $A$-modules in $\Pro(\shc)$ 
(see \S\,\ref{subsec:graded}).

 For $\Ma, \Na \in \Modg(A, \Pro(\shc))$, we set 
$\HOM_{A}(\Ma,\Na)\seteq\soplus_{k\in \Z} \, \HOM_A(\Ma,\Na)_k  \, $ with
$\HOM_{A}(\Ma,\Na)_k\seteq\Hom_{\Modg(A, \Pro(\shc))}(q^k \Ma,\Na)$.   
Then
$\HOM_A(\Ma,\Na)$ has a structure of a graded $A$-module.   

\begin{definition} \label{Def: Pro_coh}
We denote by $\Proc(A,\shc)$ the full subcategory of $\Modg(A, \Pro(\shc))$ consisting of graded $A$-modules $\Ma$ in $\Pro(\shc)$ such that
\be[(a)]
\item $\Ma/A_{>0}\Ma \in \shc$,
\item $\Ma \isoto \proolim[k] \Ma / A_{\ge k} \Ma$, or equivalently $\proolim[k] A_{\ge k}\Ma  \simeq 0$.
\ee
Here $A_{\ge k}\Ma  \seteq \Im(A_{\ge k} \tens_A \Ma \to A \tens_A \Ma  ) \subset \Ma$. 
Note that $ X \tens_A  \Ma $ belongs to $\Modg(A, \Pro(\shc))$ for any 
$\Ma\in\Modg(A, \Pro(\shc))$ and any finitely generated graded $A$-module $X$.
\end{definition}

\begin{lem}\label{lem:AAM}
Let $\Ma \in \Modg(A,\Pro(\shc))$.
\bnum
\item
For any $m \in \Z_{\ge0}$, the canonical morphism
\eq
(A_{\ge m}/A_{>m}) \tens_\cor(\Ma/A_{>0}\Ma)\simeq (A_{\ge m}/A_{>m})\tens_A\Ma\to A_{\ge m}\Ma / A_{>m}\Ma\label{eq:morgr}
\eneq
is an epimorphism in $\shc$.
\item In particular, 
if $\Ma/A_{>0}\Ma \in \shc$, then
the object $A_{\ge m}\Ma / A_{>m}\Ma$ belongs to $\shc$ by {\rm Lemma \ref{lem:art}.}
\item If $\Ma$ is $A$-flat, then the morphism \eqref{eq:morgr} is an isomorphism.
\ee
\end{lem}
\begin{proof}
Note that
\eqn
(A_{\ge m}/A_{>m})\tens_A\Ma
 &&\simeq 
\bl( A_{\ge m}/A_{>m} )\tens_\cor (A/A_{> 0})\br \tens_A \Ma \\
&&\simeq 
( A_{\ge m}/A_{>m} )\tens_\cor \bl(A/A_{> 0}) \tens_A \Ma\br
\simeq ( A_{\ge m}/A_{>m} )\tens_\cor ( \Ma/A_{>0}\Ma).
\eneqn

We have a commutative diagram in $\Modg(A, \Pro(\shc))$
\eq&&
\ba{l}\xymatrix{
A_{>m}\tens_A\Ma\ar[r]\ar@{->>}[d]&A_{\ge m}\tens_A\Ma\ar[r]\ar@{->>}[d]
&(A_{\ge m}/A_{>m})\tens_A\Ma\ar[r]\ar[d]&0\\
A_{>m}\Ma\ar[r]&A_{\ge m}\Ma\ar[r]
&A_{\ge m}\Ma/A_{>m}\Ma\ar[r]&0\\
}\ea\label{eq:am}
\eneq
with exact rows.
Since the the middle vertical arrow is an epimorphism,
so is the  right vertical arrow.

If $\Ma$ is $A$-flat, then the left and the middle vertical arrows 
in \eqref{eq:am}
are isomorphisms, and hence so is the  right vertical arrow. 
\QED
\begin{prop}\label{prop:noether}
The full subcategory $\Proc(A,\shc)$ of $\Modg(A,\Pro(\shc))$ has the following properties.
\bnum
\item The category $\Proc(A,\shc)$ is stable by taking subquotients 
and taking extensions
in $\Modg(A,\Pro(\shc))$. 
\label{item:affstab} 
\item Any object of $\Proc(A,\shc)$ is noetherian.\label{it:Noethe}
\item For any $\Ma\in\Proc(A,\shc)$, we have 
\bna
\item $A_{\ge m}\Ma/A_{\ge n}\Ma\in\shc$ for any $m,n\in\Z$ such that $0\le m\le n$,
\label{item: Mmn}
\item for any $m\in\Z_{\ge0}$ and any $\Na\in\Proc(A,\shc)$ 
such that $\Na\subset \Ma$, there exists $n>0$ such that
$\Na\cap  A_{\ge n}\Ma\subset A_{\ge m}\Na$,\label{it:subK}
\item if $\Ma=A_{>0}\Ma$, then $\Ma=0$. \label{it:Nakayama}
\ee\label{item:mic}
\ee
\end{prop}
\begin{proof}
\eqref{item: Mmn} immediately follows from Lemma \ref{lem:AAM}.

Take a set $\{z_i\}_{1\le  i \le r}$ of homogeneous generators of
the $\cor$-algebra $A$ with positive degrees.  \cmtm{$l$ is used later.}
We shall show the statements by induction on $r$.
When $r=0$,  they are obvious. Assume that $ r>0$.
We set $B\seteq \cor[z_2,\ldots,z_r] \subset A$. Then \eqref{item:affstab}--\eqref{item:mic} with $A$  replaced by $B$ hold by the induction hypothesis.
Set $d_1\seteq \deg(z_1)\in \Z_{>0}$ and $d\seteq\max\set{\deg(z_i)}%
{1\le i\le r }$. 

Then $A_{\ge k} = \sum_{j\in\Z_{\ge 0}} z_1^j B_{\ge k-d_1j}$.
In particular, we have $A_{>0} = z_1A+B_{>0}$. 

Since $A_{\ge dk} \subset (A_{>0})^k \subset A_{\ge k}$  for any $k\ge 1$, \eqref{it:Nakayama} is obvious. 
 It is also obvious that
 \eq
\text{$\Proc(A,\shc)$  is stable by taking quotients.}\label{eq:quotient}
\eneq
Let us first show that 
\eq
\parbox{70ex}{Let $\Ma\in\Mod(A,\Pro(\shc))$, and assume that there exists $n\in\Z_{>0}$ such that $z_1^n\Ma=0$.
Then $\Ma\in\Proc(A,\shc)$ if and only if $\Ma\in\Proc(B,\shc)$.}
\label{claim:AB}
\eneq

Assume that $\Ma \in \Proc(B,\shc)$. 
Since there is an epimorphism $\Ma/B_{>0}\Ma \epito \Ma/A_{>0}\Ma$ in $\Pro(\shc)$,  we have $\Ma/A_{>0}\Ma \in \shc$.
Since $$A_{\ge k } \subset Az_1^n+B_{>k-nd_1}A \qt{for $k\in \Z_{\ge 0}$},$$
we have
$\proolim[k] A_{\ge k}\Ma \subset \proolim[k] (B_{>k-nd_1}\Ma)\simeq 0$.
Hence we obtain $\Ma \in \Proc(A,\shc)$.
Conversely, assume that $\Ma \in \Proc(A,\shc)$. Then we have
$\proolim[k] B_{\ge k}\Ma \subset \proolim[k] A_{\ge k}\Ma\simeq 0$. 
Let us show $\Ma/ B_{> 0}\Ma \in \shc$.
Note that $A_{\ge nd_1}\subset z_1^nA+B_{>0}A$. 
Since $\Ma / A_{\ge nd_1}\Ma \in \shc$ by \eqref{item: Mmn}, its quotient $\Ma/(z_1^nA+B_{>0}A)\Ma \simeq \Ma / B_{>0}\Ma$ belongs to $\shc$. 
It completes the proof of \eqref{claim:AB}.

\medskip
By \eqref{eq:quotient} and \eqref{claim:AB}, we have
\eq&&
\hs{3ex}\parbox{70ex}{for any $\Ma\in\Proc(A,\shc)$ and $j,k\in\Z_{\ge0}$
such that $j\le k$, we have $z_1^j\Ma/z_1^{k}\Ma\in\Proc(B,\shc)$.}\label{claim:ABk}
\eneq

\smallskip
We shall show  that
\eq
\parbox{70ex}{for any $\Ma\in\Proc(A,\shc)$ and 
$\Na\in\Modg(A,\Pro(\shc))$ such that $\Na\subset \Ma$,
there exists $c\in\Z_{>0}$ such that
$\Na\cap z_1^{k+c}\Ma\subset z_1^k\Na$ for any $k\in\Z_{\ge0}$.}\label{claim:z}
\eneq
Since $\Ma/z_1\Ma\in\Proc(B,\shc)$ is noetherian,
the increasing  sequence $\st{(z_1^k)^{-1}\Na+z_1\Ma}_{k\in\Z_{\ge0}}$ 
of subobjects of $\Ma$ is stationary.
Hence there exists $k_0\in\Z_{\ge0}$ such that
$(z_1^k)^{-1}\Na+z_1\Ma\subset (z_1^{k_0})^{-1}\Na+z_1\Ma$ for any $k\ge k_0$.
Hence we have
$\Na\cap z_1^k\Ma\subset  z_1^{k-k_0}\Na+z_1^{k+1}\Ma$.
Let $l\in\Z_{\ge0}$. Then,  we have
$\Na\cap z_1^{k}\Ma\subset z_1^l\Na+\Na\cap z_1^{k+1}\Ma$ for any $k\ge k_0+l$. Hence, by induction on $k$, we have
$\Na\cap z_1^{k_0+l}\Ma\subset z_1^l\Na+\Na\cap z_1^{k}\Ma$ for any $k\ge k_0+l$.
Since $\proolim[k]z_1^{k}\Ma\simeq0$, we obtain
$\Na\cap z_1^{k_0+l}\Ma\subset z_1^l\Na$.
It shows \eqref{claim:z}.

\medskip
Let us show 
\eq
\text{$\Proc(A,\shc)$ is stable by extensions.}\label{claim:ext}
\eneq
Let $0\to \Ma'\to \Ma\to \Ma''\to0$ be an exact sequence in $\Mod(A, \Pro(\shc))$
such that $\Ma', \Ma''\in\Proc(A,\shc)$.
Let us show that $\Ma\in \Proc(A,\shc)$.

Since we have an exact sequence 
$$\Ma'/A_{>0}\Ma'\to \Ma/A_{>0}\Ma\to \Ma''/A_{>0}\Ma''\to0$$
in which the left and right terms belong to $\shc$,
we have $\Ma/A_{>0}\Ma\in\shc$.

Let us show $\proolim[k]A_{>k}\Ma\simeq0$.
Since the composition $\Na\seteq\proolim[k]A_{>k}\Ma\monoto\Ma\to\Ma''$
factors through $\proolim[k]A_{>k}\Ma''\simeq0$,
it is a zero morphism.
Hence
$\Na\subset \Ma'$.
Hence we have
$$
0\simeq\proolim[k]A_{>k}\Ma'\supset
\proolim[k]A_{>k}\Na
\simeq \proolim[k,j]A_{>k}A_{>j}\Ma\simeq
 \proolim[k]A_{>k}\Ma,$$
which implies that $\proolim[k]A_{>k}\Ma\simeq0$.
Thus we obtain \eqref{claim:ext}.

\medskip
Now, we are ready to prove
\eq
\text{$\Proc(A,\shc)$ is stable by taking subobjects.}\label{claim:sub}
\eneq
Let $\Ma\in\Proc(A,\shc)$ and $\Na\in \Modg(A,\Pro(\shc))$ such that $\Na\subset \Ma$.
By \eqref{claim:z}, there exists $m\in\Z_{>0}$ such that
$\Na\cap z_1^m\Ma\subset z_1\Na$.
Since $\Ma/z_1^m\Ma\in\Proc(B,\shc)$ by \eqref{claim:ABk}, 
its subobject
$\Na/(\Na\cap z_1^m\Ma)$ belongs to $\Proc(B,\shc)$ by the induction hypothesis.
Hence its quotient $\Na'\seteq\Na/z_1\Na$ also belongs to $\Proc(B,\shc)$.
Since $\Na'/B_{>0}\Na'\in \shc$
and we have $(\Na/z_1\Na)/B_{>0}(\Na/z_1\Na)\simeq
\Na/(z_1\Na+B_{>0}\Na)\epito \Na/A_{>0}\Na$, we conclude that
$\Na/A_{>0}\Na\in \shc$. 
Moreover, $\proolim[k]A_{\ge k}\Na\subset\proolim[k]A_{\ge k}\Ma$
vanishes. Hence $\Na\in\Proc(A,\shc)$ and
we obtain \eqref{claim:sub}.

Thus we have completed the proof of \eqref{item:affstab}.

\medskip
Let us show \eqref{it:subK}. 
Since $\proolim[k](\Na\cap A_{\ge k}\Ma)\simeq0$, we have
$\Na\simeq\proolim[k] \Na/(\Na\cap A_{\ge k}\Ma)$.
On the other hand, \eqref{item: Mmn} implies that $\Na/A_{\ge m}\Na$ belongs to $\shc$.
Therefore, there exists $n\in\Z_{>0}$ such that
the morphism $\Na\simeq\proolim[k] \Na/(\Na\cap A_{\ge k}\Ma)\to \Na/A_{\ge m}\Na$
factors through $\Na/(\Na\cap A_{\ge n}\Ma)$.
Hence we obtain $\Na\cap A_{\ge n}\Ma\subset A_{\ge m}\Na$.

\medskip
It remains to prove \eqref{it:Noethe}.
Let $\st{\Na_j}_{j\in\Z_{>0}}$ be an increasing sequence of subobjects
of $\Ma\in\Proc(A,\shc)$. We shall show that it is stationary.

 The family $\st{(z_1^k)^{-1}\Na_j+z_1\Ma}_{j,k\in\Z_{>0}}$ is stationary 
since $\Ma/z_1\Ma \in \Proc(B,\shc)$ is noetherian by the induction hypothesis. Hence there exist $k_0,j_0\in\Z_{>0}$ such that
$(z_1^k)^{-1}\Na_j+z_1\Ma=(z_1^{k_0})^{-1}\Na_{j_0}+z_1\Ma$ for $j\ge j_0$ and $k\ge k_0$.
Hence we have
$$\Na_j\cap z_1^k\Ma\subset \Na_{j_0}+z_1^{k+1}\Ma\qt{for $k\ge k_0$ and $j\ge j_0$.}$$
It implies that
$$ \Na_{j_0}+\Na_j\cap z_1^k\Ma\subset \Na_{j_0}+\Na_j\cap z_1^{k+1}\Ma
\qt{for any $k\ge k_0$. }$$
Hence, by induction, we obtain
$$ \Na_{j_0}+\Na_j\cap z_1^{k_0}\Ma\subset \Na_{j_0}+\Na_j\cap z_1^{k+1}\Ma
\qt{for any $k\ge k_0$.}$$
Taking $\proolim[k]$, we obtain
$$ \Na_{j_0}+\Na_j\cap z_1^{k_0}\Ma\subset \Na_{j_0}.$$
Hence $\Na_j\cap z_1^{k_0}\Ma\subset \Na_{j_0}$ for $j\ge j_0$.
On the other hand,
$\st{\Na_j+z_1^{k_0}\Ma}_{j\in\Z_{>0}}$ is stationary since
$\Ma/z_1^{k_0}\Ma\in\Proc(B,\shc)$ is noetherian.
Hence, there exists $j_1\ge j_0$ such that
$\Na_j+z_1^{k_0}\Ma=\Na_{j_1}+z_1^{k_0}\Ma$ for $j\ge j_1$.
Hence we have
$\Na_j\subset \Na_{j_1}+\Na_j\cap z_1^{k_0}\Ma\subset \Na_{j_1}+\Na_{j_0}=\Na_{j_1}$ for any $j\ge j_1$.
\end{proof}

\Lemma\label{lem:fin}
Let $X\in\Modgc(A)$ and $\Ma\in\Proc(A,\shc)$. Then we have:
\bnum
\item $X\tens[A]\Ma\in \Proc(A,\shc)$,
\item if $\dim_\cor X<\infty$, then $X\tens[A]\Ma\in\shc$.
\ee
\enlemma
\Proof
(i)\ Taking an epimorphism $A^{\oplus r}\epito X$, we conclude (i) since
$X\tens[A]\Ma$ is a quotient of $\Ma^{\oplus r}$.

\snoi
(ii)\ Since $\dim_\cor \END_A(X)<\infty$, there exists $m\in\Z$ such that $\END_A(X)_{\ge m}=0$.
Then $A_{\ge m}X=0$, which implies that
$A_{\ge m}\bl X\tens[A]\Ma\br=0$.
Hence $X\tens[A]\Ma\in\shc$ by Proposition \ref{prop:noether} \eqref{item: Mmn}.

\QED

\Lemma\label{lem:red}
Let $f\cl \Ma\to\Na$ be a morphism in $\Proc(A,\shc)$.
\bnum 
\item If $\tf\seteq(A/A_{>0})\tens[A]f\cl\Ma/A_{>0}\Ma\to\Na/A_{>0}\Na$ 
is an epimorphism, then $f$ is an epimorphism.
\item If $\Na$ is $A$-flat and $\tf$ is an isomorphism, then
$f$ is an isomorphism.
\ee
\enlemma
\Proof
(i)\ If $\tf$ is an epimorphism, then 
$(A/A_{>0})\tens[A]\Coker(f)\simeq0$, and hence
$\Coker(f)\simeq0$ by Proposition \ref{prop:noether}.

\snoi
(ii) Since $\Na$ is $A$-flat, 
we have an exact sequence 
$$0\to (A/A_{>0})\tens[A]\Ker(f)\to \Ma/A_{>0}\Ma\to\Na/A_{>0}\Na \to0.$$
Hence $\Ker(f)\simeq0$.
\QED

\Prop\label{prop:fin}
Let $\Ma,\Na\in\Proc(A,\shc)$
and set $M\seteq\Ma/A_{>0}\Ma\in\shc$ and $N\seteq\Na/A_{>0}\Na\in\shc$.
{\em Assume that $\Na$ is $A$-flat.}
Let $X$ be a finitely generated graded $A$-module
and $\la,\mu\in\Z$.
Assume that $\HOM_\shc(M,N)$ has degree $\ge \la$ and
$X$ has degree $\ge \mu$.
Then we have 
\bnum
\item 
$\HOM_A(\Ma, X\tens[A]\Na)$ has degree $\ge \mu+\la$.
\label{it:gela}
\item
$\HOM_A(\Ma,X\tens[A]\Na)$ is a finitely generated $A$-module,
\item there exists an integer $m\in\Z_{\ge0}$ such that
$$\HOM_A\bl\Ma,X_{\ge k+m}\tens[A]\Na\br\subset A_{\ge k}\HOM_A(\Ma,X\tens[A]\Na)
\qt{for any $k\in\Z_{\ge0}$,}$$
as a subset of $\HOM_A(\Ma,X\tens[A]\Na)$.
\ee
\enprop
\Proof
Set $\Laa=X\tens[A]\Na\in\Proc(A,\shc)$ and $\Laa_{\ge m}=X_{\ge m}\tens[A]\Na\subset\Laa$.
Since $\Na$ is $A$-flat, we have
$$\Laa_{\ge m}/\Laa_{>m}\simeq
(X_{\ge m}/X_{> m})\tens[A]\Na
\simeq X_m\tens[\cor]N.$$

(i)\ 
We have an inclusion
\eq \label{eq:ST}\\
&&\hs{2ex}\xymatrix@C=3.4ex{
\dfrac{\HOM_A(\Ma,\Laa_{\ge m})}{\HOM_A(\Ma,\Laa_{>m})}\akew \ar@{>->}[r] 
&\HOM_A\bl\Ma,\Laa_{\ge m}/\Laa_{>m}\br
\ar[r]^-{\sim } &X_m\tens[\cor]\HOM_\shc(M,N).
}\nn
\eneq
Hence,
$\HOM_A(\Ma,\Laa_{\ge m}/\Laa_{>m})$
has degree $\ge m+\la$. Therefore,
$\HOM_A(\Ma,\Laa_{\ge m}/\Laa_{\ge s})$ has also degree $\ge m+\la$ for any $s\ge m$.
Hence, we obtain
$$\HOM_A(\Ma,\Laa)_k\simeq
\prolim[s]\HOM_A(\Ma,\Laa_{\ge \mu}/\Laa_{\ge s})_k\simeq 0\qt{if $k<\mu+\la$.}$$
Thus we obtain \eqref{it:gela}.
Applying \eqref{it:gela} to $X_{\ge m}$, we obtain
\eq\text{$\HOM_A(\Ma,\Laa_{\ge m})$ has degree $\ge m+\la$.}
\label{eq:La}
\eneq

\smallskip
Let $S=\soplus_{m\in\Z}S_m$ be the graded $A$-module with
$S_m=\HOM_A(\Ma,\Laa_{\ge m})/\HOM_A(\Ma,\Laa_{>m})$.
Then by \eqref{eq:ST} we have a monomorphism
$$S\monoto X\tensc  \HOM_\shc(M,N),$$
which implies that $S$ is a finitely generated graded  $A$-module.
Hence there exist $r\in\Z_{\ge\mu}$ and 
 a finite-dimensional graded $\cor$-subspace $K_j\subset \HOM_A(\Ma,\Laa_{\ge j})$ ($\mu\le j\le r$)
such that
$$\HOM_A(\Ma,\Laa_{\ge m})\subset \sum_{\mu\le j\le r}A_{\ge m-j}K_j
+\HOM_A(\Ma,\Laa_{>m})\qt{for any $m\in\Z$.}$$
Hence by  induction on $t$, we obtain
$\HOM_A(\Ma,\Laa_{\ge m})\subset \sum_{j=\mu}^rA_{\ge m-j}K_j+\HOM_A(\Ma,\Laa_{\ge t})$ 
for any $m\in\Z$ and $t\in\Z$ such that $t\ge m$. 
Hence \eqref{eq:La} implies that $\HOM_A(\Ma,\Laa_{\ge m})\subset \sum_{j=0}^rA_{\ge m-j}K_j$ from which (ii) and (iii) follows. 
\QED

\Prop\label{prop:base}
Let $\Ma,\Na\in\Proc(A,\shc)$, and 
set $M=\Ma/A_{>0}\Ma$ and $N=\Na/A_{>0}\Na$.
Assume that $\Na$ is $A$-flat.
Let $H_0$ be a graded $\cor$-subspace of $\HOM_A(\Ma,\Na)$.
\bnum
\item If the composition $H_0\to\HOM_A(\Ma,\Na) \To[(A/A_{>0}) \tensa \scbul ]  \HOM_\shc(M,N)$ is injective, then, $A\tensc H_0\to\HOM_A(\Ma,\Na)$ is injective.
\item If $H_0\to\HOM_\shc(M,N)$ is surjective, then, $A\tensc H_0\to\HOM_A(\Ma,\Na)$ is surjective.
\ee
\enprop
\Proof
Let $s\in\Z$.
Then the composition
\eqn
g_m\cl  \bl(A_{\ge m}/A_{\ge m+1})\tensc H_0\br_s&\to&
\bl(A_{\ge m}/A_{\ge m+1})\tensc \HOM_\shc(M,N)\br_s\\
&&\hs{10ex}\isoto
\HOM_A(\Ma,A_{\ge m}\Na/A_{\ge m+1}\Na)_s
\eneqn
is injective (resp.\ surjective) under the assumption in (i) (resp.\ (ii)).
We shall show that
$f_m\cl (A_{\ge m}\tensc H_0)_s\to \HOM_A (\Ma,A _{\ge m}\Na)_s$ is 
injective (resp.\ surjective)
by the descending induction on $m\in\Z_{\ge0}$.
If $m\gg0$, then it is true since
$(A_{\ge m}\tens H_0)_s\simeq \HOM_A (\Ma,A_{\ge m}\Na)_s\simeq0$
 by Proposition \ref{prop:fin}.

Now assume that
$f_{m+1}$ is injective (resp.\ surjective).
Consider a commutative diagram with exact rows:

\scb{.9}{\parbox{80ex}{
\eqn
&&\xymatrix@R=7ex{
0\ar[r]&\bl A_{\ge m+1}\tens H_0\br_s\ar[r]\ar[d]^{f_{m+1}}&\bl A_{\ge m}\tens H_0\br_s\ar[r]\ar[d]^{f_{m}}
&\bl (A_{\ge m}/A_{\ge m+1})\tens H_0\br_s\ar[d]^{g_m}\ar[r]&0\\
0\ar[r]&\HOM_A(\Ma,A_{\ge m+1}\Na)_s\ar[r]&
\HOM_A (\Ma,A_{\ge m}\Na)_s\ar[r]&\HOM_A (\Ma,A_{\ge m}\Na/A_{\ge m+1}\Na)_s\,.
}
\eneqn}}

Since $f_{m+1}$ and $g_m$
are injective (resp.\ surjective), so is $f_{m}$,
and the induction proceeds.
\QED

\Prop\label{prop:uniquehom}
Let $\Ma,\Na\in\Proc(A,\shc)$, and
set $M=\Ma/A_{>0}\Ma$ and $N=\Na/A_{>0}\Na$.
Assume that \bna
\item $A$ is an integral domain, \label{it.int}
\item 
$\Na$ is $A$-flat,\label{it:Nflat}
\item $\HOM_\shc(M,N)=\cor f$ for a non-zero $f$ 
of homogeneous degree $\la\in\Z$,
\item there exists a non-zero $F\in\HOM_A(\Ma,\Na)$.
\ee
Then, $\HOM_A(\Ma,\Na)$ is a free $A$-module generated by an element $\tF
\in\HOM_A(\Ma,\Na)_\la$ such that $(A/A_{>0})\tens_A \tF= f$. 
\enprop
\Proof
First note that $\HOM_A(\Ma,A_{\ge m}\Na)$ is of degree $\ge m+\la$
by Proposition \ref{prop:fin}.

\smallskip
Let $\psi\in \HOM_A(\Ma,\Na)$ be a non-zero homomorphism of degree $d$.
Let us take $s\in\Z_{\ge0}$ such that
 $\psi\in \HOM_A(\Ma,A_{\ge s}\Na)$ and
 $\psi\not\in \HOM_A(\Ma,A_{>s}\Na)$.
Then, its image
$\bpsi\in\HOM_A(\Ma,A_{\ge s}\Na/A_{>s}\Na)$ does not vanish.
Since $\HOM_A(\Ma,A_{\ge s}\Na/A_{>s}\Na)\simeq (A_{\ge s}/A_{>s})\tens\HOM_\shc(M,N)$
by \eqref{it:Nflat},
there exists $a\in A_s$ such that $\bpsi=af$.
In particular, $d=s+\la$.
Hence, we have
$\psi\in \HOM_A(\Ma, a\Na+A_{>s}\Na)$.
We have an exact sequence
$$0\To\HOM_{A}(\Ma, a\Na)_d\To \HOM_A(\Ma, a\Na+A_{>s}\Na)_d
\To \HOM_A(\Ma, (a\Na+A_{>s}\Na)/a\Na)_d.$$
Since $(a\Na+A_{>s}\Na)/a\Na\simeq\bl(aA+A_{>s})/aA\br\tens[A]\Na$
and $(aA+A_{>s})/aA\simeq A_{>s}/(aA\cap A_{>s})$ has degree $\ge s+1$,
Proposition~\ref{prop:fin} implies that $\HOM_A(\Ma, (a\Na+A_{>s}\Na)/a\Na)$ has degree $\ge s+1+\la=d+1$.
Hence, $\HOM_A(\Ma, (a\Na+A_{>s}\Na)/a\Na)_d\simeq0$, which implies that
$\HOM_{A}(\Ma, a\Na)_d\simeq \HOM_A(\Ma, a\Na+A_{>s}\Na)_d$.
Hence we have $\psi\in \HOM_{A}(\Ma, a\Na)_d=a\HOM_{A}(\Ma, \Na)_\la$
since $A$ is an integral domain.
Hence there exists $\psi'\in \HOM_{A}(\Ma,\Na)_\la$ such that
$\psi=a\psi'$ and $(A/A_{>0})\tens\psi'=f$.

\medskip
Applying it to  $\psi=F$, there exist $a\in A$ and
$\tF\in \HOM_{A}(\Ma,\Na)_\la$ such that
$F=a\tF$.

It remains to remark
$\HOM_A(\Ma,\Na)_\la\isoto \HOM_\shc (M,N)_\la $.
Hence $\HOM_{A}(\Ma,\Na)$ is  generated by $\tF$ as an $A$-module. Since $\Na$ is $A$-flat,  $a:\Na \to \Na$ is injective for any $a\in A$. Thus $a \tF =0$ implies $a=0$ so that  $\HOM_{A}(\Ma,\Na)$ is free over $A$.
\QED

\Rem In Proposition~\ref{prop:uniquehom}, the condition \eqref{it.int}
is necessary.
Indeed, if we take $\shc=\cor\gmod$,
$A=\cor[z]/\cor[z]z^2$,  $\Ma=Az$ and $\Na=A$.
Then $\HOM_\shc(M,N)\simeq\cor$, 
$\HOM_A(\Ma,\Na)\simeq A/Az$ and $\HOM_A(\Ma,\Na)\to\HOM_{\shc}(M,N)$
vanishes.

Here is another example where $\Ma$ and $\Na$ are $A$-flat.
Let $B=\cor[t]/\cor[t]t^2$ with $\deg(t)=1$ and $\shc=  B\gmod$ and $A=\cor[z]/\cor[z]z^2$.
For $c\in\cor^\times$, let $\Xa_c=\cor[t]/\cor[t]t^2 \in\shc$ where $z\in\END(\Xa_c)$ is given by $\Xa_c \ni x\mapsto ctx \in \Xa_c$. 
Then  $\Xa_c$ is $A$-flat. If $\Ma=\Xa_{c_1}$, $\Na=\Xa_{c_2}$
for $c_1\not=c_2$, then
$\HOM_\shc(M,N)\simeq\cor$, 
$\HOM_A(\Ma,\Na)\simeq A z $ and $\HOM_A(\Ma,\Na)\to\HOM_{\shc}(M,N)$
vanishes.
\enrem 

\Lemma\label{lem:cohcor}
We have equivalences of categories
\eqn
A\gmod&&\isoto\Mod(A,\cor\gmod),\\
\Modgc(A)&&\isoto\Proc(A,\cor\gmod).
\eneqn
\enlemma
\Proof
The first equivalence is obvious.
We define the functor $\Phi\cl \Modgc(A)\to\Proc(A,\cor\gmod)$  by
$M\to\proolim[m]M/M_{\ge m}$,
and $\Psi\cl\Proc(A,\cor\gmod) \to \Modgc(A)$ by
$\Ma\to \HOM_A(\Phi(A),\Ma)$.
It is easy to check that they are well-defined and quasi-inverse to each other.
\QED

\subsection{Flatness}

\Prop\label{prop:platitude}
Let $A$ be a commutative graded ring satisfying \eqref{cond:gring},
let $z\in A_{>0}$ be a non-zero-divisor,
and set $B=A/Az$.
Assume that $\Ma\in \Proc(A,\shc)$ satisfies
\bna
\item
$z\vert_\Ma$ is a monomorphism,
\item $\Ma/z\Ma$ is $B$-flat.
\ee
Then $\Ma$ is $A$-flat.
\enprop
\Proof
It is enough to show that
$\Tor_1^A(S,\Ma)\simeq0$ for any finitely generated $A$-module $S$.

\snoi
(i)\ Assume that $zS=0$.
Then we have
$\Tor_1^A(S,\Ma)\simeq \Tor_1^B(S,\Ma/z\Ma)\simeq0$.  

\snoi
(ii)\ Assume that $z^mS=0$ for some $m\in\Z_{>0}$.
We shall show $\Tor_1^A(S,\Ma)\simeq0$ by induction on $m$.
We have an exact sequence
$0\to\Ker(z\vert_S)\to S\to zS\to0$.
Since $z^{m-1}\vert_{zS}=0$, we have
$\Tor_1^A(zS,\Ma)\simeq0$ by the induction hypothesis, and
$\Tor_1^A(\Ker(z\vert_S),\Ma)\simeq0$ by (i).
Hence the exact sequence 
$$\Tor_1^A(\Ker(z\vert_S),\Ma)\To\Tor_1^A(S,\Ma)\To\Tor_1^A(zS,\Ma)$$
implies $\Tor_1^A(S,\Ma)\simeq0$.

\mnoi
(iii) Assume that $z\vert_{S}$ is a monomorphism.
Then we have an exact sequence
$$\Tor_1^A(S,\Ma)\To [z]\Tor_1^A(S,\Ma)\To\Tor_1^A(S/zS,\Ma).$$
We have $\Tor_1^A(S/zS,\Ma)\simeq 0$ by (i).
Hence
$\Tor_1^A(S,\Ma)/z\Tor_1^A(S,\Ma)\simeq0$. Since
$\Tor_1^A(S,\Ma)\in\Proc(A,\shc)$, 
Proposition~\ref{prop:noether}\;\eqref{item:mic} implies 
$\Tor_1^A(S,\Ma)\simeq0$.

\snoi
(iv)\ In general we have
an exact sequence $0\to S'\to S\to S''\to 0$
such that
$z^mS'=0$ for some $m\in\Z_{>0}$, and $z\vert_{S''}$ is a monomorphism. 
Hence we have $\Tor^A_1(S,\Ma)\simeq 0$ by (ii) and (iii).
\QED

\Cor \label{cor:platitude}Let $z_k$ be a homogeneous indeterminate
with positive degree \ro $k=1,\ldots,n$\rf, and let $A\seteq\cor[z_1,\ldots, z_n]$ be 
the graded polynomial ring.
Assume that $\Ma\in\Proc(A,\shc)$  satisfies
\eqn
\parbox{70ex}{
$(z_1,\ldots,z_n)$ is $\Ma$-regular, i.e.,
for any $k$ such that $1\le k\le n$, $z_k\bigm|_{\Ma/\sum_{j=1}^{k-1} z_j\Ma}$
is a monomorphism.}
\eneqn
Then $\Ma$ is $A$-flat.
\encor
\Proof
This immediately follows from Proposition~\ref{prop:platitude}
by induction on $n$.
\QED

\Lemma \label{lem:trunc}
Let $m\in\Z_{>0}$ and $A=\cor[z]/\cor[z]z^m$,
and let $M$ be a graded $A$-module in an abelian  $\cor$-linear graded category.
Then the following conditions are equivalent:
\bna
\item $M$ is $A$-flat,
\item 
$z\cl z^{k-1}M/z^kM\to z^{k}M/z^{k+1}M$ is an isomorphism for any $k$ such that 
$1\le k<m$,
\item
$z^k\cl M/zM\to z^{k}M/z^{k+1}M$ is an isomorphism for any $k$ such that 
$1\le k<m$,
\item
$\Ker(z^{k}\vert_M)\subset z^{m-k}M$ for any $k$ such that
$1\le k<m$,
\item
$\Ker(z^{k}\vert_M)\subset z^{m-k}M$ for some $k$ such that
$1\le k<m$.
\ee
\enlemma
\Proof
It follows from the fact that any graded ideal of $A$ is $z^kA$ for some $k\in\Z$
with $0\le k\le m$. 
\QED

\Prop\label{prop:flatfin}
Let $A$ be a commutative graded $\cor$-algebra
satisfying \eqref{cond:gring} and
let $\Ma\in\Proc(A,\shc)$.
\bnum
\item If the sequence
$$0\to X'\tensa \Ma \to X\tensa\Ma\to X''\tensa\Ma\to 0$$ is exact for any exact sequence
$0\to X'\to X\to X''\to0$ in $A\gmod$, 
then $\Ma$ is $A$-flat.
\item $\Ma$ is $A$-flat if and only if $\Tor_1^{A}(A/A_{>0},\Ma)\simeq0$.
\ee
\enprop
\Proof
(i)\ 
Let $0\to X'\to X\to X''\to 0$ be an exact sequence
in $\Modgc(A)$.
We shall show that $0\to X'\tensa \Ma\to X\tensa \Ma\to X''\tensa \Ma\to0$ is
exact.

Let $0\to X'_m\to X_m\to X''_m\to 0$ be the exact sequence 
of $A$-modules with finite $\cor$-dimension where
$X_m=X/A_{\ge m}X$, $X''_m=X''/A_{\ge m}X''$ and $X'_m=X'/(X'\cap A_{\ge m}X)$.
By Proposition~\ref{prop:noether}\;\eqref{it:subK}, 
there exists $n\in\Z$ such that
$A_{\ge n}X'\subset X'\cap A_{\ge m}X$, which implies that
$$\proolim[m]X'/A_{\ge m}X'\simeq \proolim[m]X'_m.$$
Hence
$$0\to\proolim[m](X'/A_{\ge m}X')\tens[A]\Ma\to
\proolim[m](X/A_{\ge m}X)\tens[A]\Ma\to
\proolim[m](X''/A_{\ge m}X'')\tens[A]\Ma\to0$$
is exact.
Since $\proolim[m](X/A_{\ge m}X)\tens[A]\Ma\simeq 
\proolim[m]X\tens[A](A/A_{\ge m})\tens[A]\Ma\simeq
X\tens[A]\Ma$, etc.,
we obtain the desired result.

\snoi
(ii)\ Assume that $\Tor_1^A(A/A_{>0},\Ma)\simeq0$.
We shall show that $\Tor_1^A(X,\Ma)\simeq0$ for any
$X\in A\gmod$. There exists $m\in\Z_{\ge0}$ such that $A_{\ge m}X=0$.
We shall argue by induction on $m$.
By the exact sequence
$\Tor_1^A(A_{>0}X,\Ma) \to\Tor_1^A(X,\Ma)\to\Tor_1^A(X/A_{>0}X,\Ma)$,
the left term vanishes by the induction hypothesis,
and the right term vanishes since 
$X/A_{>0}X$ is a direct sum  of copies of $A/A_{>0}$.

Now for any exact sequence $0\to X'\to X\to X''\to 0$ in $A\gmod$,
 we obtain the exact sequence
$$ 0 \simeq\Tor_1^A(X'',\Ma)\To X'\tensa\Ma\To X\tensa\Ma\To X''\tensa\Ma\To0,$$
as desired.\QED

\section{Duality} \label{Sec: duality}

\subsection{Complements}
In this subsection, we give two lemmas.
Note that the graded versions of these lemma still hold,
although we do not repeat them.
\Lemma\label{lem:NMX}
Let $B$ be a commutative noetherian ring and let $\sha$ be an abelian category.
Let $X$ be a $B$-flat module in $\sha$,
and let $M$ and $N$ are finitely generated $B$-modules.
$$\ihom_B(N,M\tens[B]X)\simeq\Hom_B(N,M)\tens[B]X.$$
\enlemma

\Proof
We have a morphism
$$\Hom_B(N,M)\tens[ B]X\to \ihom_B (N,M\tens[B]X).$$
There exists an exact sequence of $B$-modules
$$L\to L'\to N\to 0,$$
where $L$ and $L'$ are finitely generated free $B$-modules.
Then we have an exact sequence
$$0\To\Hom_B(N,M)\To \Hom_B(L',M)\To\Hom_B(L,M).$$

In the commutative diagram with exact rows
$$\xymatrix@C=5ex{
0\ar[r]&\Hom_B(N,M)\tens[B]X\ar[r]\ar[d]&\Hom_B(L',M)\tens[B]X\ar[r]\ar[d]&\Hom_B(L,M)\tens[B]X\ar[d]\\
0\ar[r]&\ihom_B(N,M\tens[B]X)\ar[r]&\ihom_B(L',M\tens[B]X)\ar[r]
&\ihom_B(L,M\tens[B]X),}
$$
the middle and right vertical arrows are isomorphism, and hence
the left one is also an isomorphism.
\QED

\Lemma Let $B$ be a commutative noetherian ring and let $\sha$ and $\sha'$ be abelian categories, 
If\/ $\Dual\cl\sha^\opp\to\sha' $ is an equivalence of categories,
then we have
$$\Dual(\ihom_{B}(M,X))\simeq \Dual X \tens[B]M$$
for any  $X  \in\Mod(B,\sha)$ and a finitely generated $B$-module $M $.
\enlemma
\Proof
For any $Y\in\sha'$, we have
\eqn
\Hom_{\sha'}(\Dual X\tens[B]M,Y)
&&\simeq \Hom_{B}\bl M,\Hom_{\sha'}(\Dual X,Y)\br\\
&&\simeq \Hom_{B}\bl M,\Hom_{\sha}(\Dual^{-1}Y,X)\br\\
&&\simeq \Hom_{\sha}\bl \Dual^{-1}Y,\ihom_B(M,X)\br\\
&&\simeq \Hom_{\sha'}\Bigl( \Dual\bl\ihom_B(M,X)\br,Y\Bigr).
\eneqn
\QED

\subsection{Duality} \label{subSec: duality}

Let $\shc$ and $\shc'$ be $\cor$-linear abelian categories which satisfy 
\eqref{cond:fcat}.
Let $\Dual\cl \shc^\opp\to\shc'$ be  an equivalence of categories.

Let $A$ be a commutative graded $\cor$-algebra which satisfies \eqref{cond:gring}.
For a graded $A$-module $X$, we wet $X^*\seteq\HOM_\cor(X,\cor)$ which is a
graded $A$-module.

\Lemma\label{lem:DXY}
Let $\Ma\in\Modg(A,\Pro(\shc))$ which is $A$-flat,
and let $X$ and $Y$ be finite $\cor$-dimensional graded $A$-modules.
Then, we have
$$
\Dual(X^*\tens[A]\Ma)\tens[A]Y
\simeq \Dual\bl(X\tens[A]Y)^*\tens[A]\Ma\br.
$$
\enlemma
\Proof
We have
$$
\Dual(X^*\tensa \Ma)\tens[A]Y\simeq
\Dual\bl\ihom_A(Y,X^*\tens[A]\Ma)\br
\simeq \Dual\bl\HOM_A(Y,X^*)\tens[A]\Ma\br.$$
It remains to remark that $\HOM_A(Y,X^*)\simeq(X\tens[A]Y)^*$.
\QED

\Def\label{def:duala}
For $\Ma\in\Proc(A,\shc)$, set $\DA(\Ma)\seteq\proolim[m] \Dual\bl(A/A_{\ge m})^*\tens[A]\Ma\br
\in\Modg\bl A,\Pro(\shc')\br$.
Similarly, for $\Na\in \Proc(A,\shc')$, set
$\DmA(\Na)=\proolim[m] \Dual^{-1}\bl (A/A_{\ge m})^*\tens[A]\Na\br
\in\Modg\bl A,\Pro(\shc)\br$.
\edf
\Prop\label{prop:dualflat}
Assume that $\Ma\in\Proc(A,\shc)$ is $A$-flat. Then we have:
\bnum
\item $\DA(\Ma)$ belongs to $\Proc(A,\shc')$,
\item $\DA(\Ma)$ is $A$-flat,
\item
for any finite-$\cor$-dimensional graded $A$-module $X$, we have
$$\DA(\Ma)\tens[A]X\simeq \Dual(X^*\tens[A]\Ma),$$
\label{item:MXY}
\item $\DmA\DA(\Ma)\simeq \Ma$.
\ee
\enprop
\Proof
For any finite $\cor$-dimensional $A$-module $X$,
we have
\eqn\DA(\Ma)\tensa X&&\simeq \proolim[m] \Dual\bl(A/A_{\ge m})^*\tens[A]\Ma\br\tens[A]X\\
&&\simeq
\proolim[m]\Dual\Bigl(\bl(A/A_{\ge m})\tens[A]X\br^*\tens[A]\Ma\Bigr)
\simeq
\Dual(X^*\tens[A]\Ma)
\eneqn
by Lemma~\ref{lem:DXY}.
Hence $\DA(\Ma)\tensa X\in\shc'$ and 
$$\proolim[m](A/A_{\ge m})\tens[A]\DA(\Ma)
\simeq\proolim[m] \Dual\bl(A/A_{\ge m})^*\tens[A]\Ma\br\simeq\DA(\Ma).$$
Hence we have $\DA(\Ma)\in\Proc(A,\shc')$. 
Since the functor 
$A\gmod\to \Proc(A,\shc')$ given by $X\mapsto \DA(\Ma)\tensa X\simeq\Dual(X^*\tens[A]\Ma)$ is exact,
the object $\DA(\Ma)$ is $A$-flat by Proposition~\ref{prop:flatfin}.

Finally we have
\eqn
\DmA\DA(\Ma)
&&\simeq \proolim[m]\Dual^{-1}\bl(A/A_{\ge m})^*\tens[A]\DA(\Ma)\br\\
&&\simeq \proolim[m] \Dual^{-1}\circ \Dual\bl(A/A_{\ge m})\tens[A]\Ma)\br
\simeq \Ma.
\eneqn
\QED

\Def\label{def:aff} Let us denote by $\Aff[A](\shc)$ 
the full subcategory of $\Proc(A,\shc)$ consisting of $A$-flat objects.
\edf
Then  by Proposition \ref{prop:dualflat} the functor
$$\DA\cl \Aff[A](\shc)^\opp\isoto\Aff[A]( \shc')$$
is an equivalence of categories.

\subsection{Affine objects} \label{suSec: affine obj}

\Def \label{Def: affine objects}
Let $z$ be an indeterminate of homogeneous degree $d\in\Z_{>0}$.
An object
of $\Proc(\cor[z],\shc)$ is nothing but a pair $(\Ma,z)$ such that
\bna
\item $\Ma\in\Pro(\shc)$ and $z\in\END_{\Proc(\shc) }(\Ma)_d$,
\item $\Ma/z\Ma\in\shc$,
\item $\Ma\isoto\proolim[n]\Ma/z^{n}\Ma$.
\setcounter{myc}{\value{enumi}}
\ee
If $(\Ma,z)$ satisfies further the following condition

\bna\setcounter{enumi}{\value{myc}}
\item $z\in\END_{\Proc(\shc) }(\Ma)$ is a monomorphism,
\ee
then we say that $(\Ma,z)$ is an {\em \afn} (of $\Ma/z\Ma$). 
\edf
Note that an \afn $(\Ma,z)$ is nothing but a $\cor[z]$-flat object in $\Proc(\cor[z],\shc)$.
We denote by $\Aff(\shc)$ the category of 
\afns.

\Prop\label{prop:fpz}
The full subcategory $\Proc(\cor[z],\shc)$ of $\Mod(\cor[z],\Pro(\shc))$ has
the following properties.
\bnum
\item $\Proc(\cor[z],\shc)$  is stable by taking subquotients and extensions,

\item any object of $\Proc(\cor[z],\shc)$ is noetherian,
\item for any $(\Ma,z)\in\Proc(\cor[z],\shc)$, we have
\bna
\item $z^m\Ma/z^n\Ma\in\shc$ for any $m,n\in\Z$ such that $0\le m\le n$,
\item $\Ker z^n\in\shc$ for any $n\in\Z_{\ge0}$, and the sequence $\st{\Ker(z^n)}_{n\in\Z_{\ge0}}$ is  stationary,
\item for any $\Na\in\Mod(\cor[z],\Pro(\shc))$ such that $\Na\subset\Ma$, there exists $n>0$ such that
$\Na\cap z^n\Ma\subset z\Na$,
\item if $\Ma=z\Ma$, then $\Ma=0$. 
\ee
\ee
\enprop
This is nothing but a special case of Proposition~\ref{prop:noether}.

For $\Ma\in\Proc(\cor[z],\shc)$, we set
$$\Ma_\tor\seteq\bigcup_{n\in\Z_{>0}}\Ker(z^n\vert_\Ma)\subset\Ma\qtq
\Ma_\fl=\Ma/\Ma_\tor.$$
Note that $\Ma_\tor$ is well-defined since $\Ma$ is noetherian,
and $\Ma_\fl$ is an \afn.

\Def \label{Def: truncated aff}
A \emph{truncated \afn} of $M\in \shc$ at $m\in\Z_{>0}$ (of degree $d$) is a
pair $(\Ma,z)$ of an object $\Ma$ of $\shc$
and an endomorphism $z$ of $\Ma$
such that
\bna
\item $z$ is homogeneous of degree $d\in\Z_{>0}$,
\item $z^m=0$,
\item $\Ma/z\Ma\simeq M$,
\item $z\cl z^{k-1}\Ma/z^{k}\Ma\To z^k\Ma/z^{k+1}\Ma$ is an isomorphism
if $1\le k\le m-1$.
\ee
We sometimes say that $\Ma$ is an $m$-truncated \afn.
\edf
Note that an $m$-truncated \afn is nothing but a
$(\cor[z]/\cor[z]z^m)$-flat object
of $\Proc(\cor[z]/\cor[z]z^m,\shc)$
by Lemma \ref{lem:trunc}.

If $(\Ma,z)$ is an \afn, then
$\Ma/z^m\Ma$ is an $m$-truncated \afn.

\Lemma
Let $\Ma$ be an  object of $\Proc(\cor[z],\shc)$ such that $M\seteq\Ma/z\Ma$
is a simple object of $\shc$. Then $(\Ma,z)$ is either an \afn or
a truncated \afn.
\Proof
Let us consider the epimorphism
$$f_k\cl \Ma/z\Ma\epito z^k\Ma/z^{k+1}\Ma$$
given by $z^k$.
If $f_k$ is an isomorphism for every $k\in\Z_{\ge0}$, then
$(\Ma,z)$ is an \afn.
Otherwise there exists $k$ such that $f_k$ is not an isomorphism.
Take the smallest $k$ among such $k$'s.
Then $f_k=0$ and hence $z^k\Ma=z^{k+1}\Ma$, which implies that 
$z^k\Ma=0$.
Moreover $f_j\cl \Ma/z\Ma\to z^j\Ma/z^{j+1}\Ma$ is an isomorphism for $j<k$,
since it is non-zero. 
Hence $(\Ma,z)$ is a
$k$-truncated affine object. 
\QED
\enlemma
\Lemma  \label{lem:endpreaff}
If $(\Ma,z)$ is an \afn of a simple object of $\catC$, then we have
$$\END_{\cor[z]}(\Ma)\simeq\cor[z].$$
\enlemma

\Proof
It immediately follows from Proposition~\ref{prop:uniquehom}.
\QED

\Def \label{Def: Raff}
We denote by $\Rat(\shc)$ the category 
with $\Ob\bl\Proc(\cor[z],\shc)\br$  
as the set of objects and with the morphisms defined as follows. 
For $\Ma,\Na\in\Ob\bl\Rat(\shc)\br$,
\eqn\Hom_{\,\Rat(\shc)}(\Ma,\Na)&&=\indlim_{k\in\Z_{\ge0}}
\Hom_{\,\Proc(\cor[z],\shc)}\bl(z^k\cor[z])\tens[{\cor[z]}]\Ma,\Na\br\\
&&\simeq
\indlim_{k\in\Z_{\ge0}}
\Hom_{\,\Proc(\cor[z],\shc)}(\Ma,\cor[z]z^{-k}\tens_{\cor[z]}\Na).\eneqn
\edf

Note that we have
$$\HOM_{\Rat(\shc)}(\Ma,\Na)\simeq
\cor[z,z^{-1}]\tens_{\cor[z]}\HOM_{\Proc(\cor[z],\shc)}(\Ma,\Na)
\qt{for $\Ma,\Na\in\Proc(\cor[z],\shc)$.}$$
Hence,
any object of $\Rat(\shc)$ is a $\cor[z,z^{-1}]$-module, i.e.,
$z\in\END(\Ma$) is invertible for any $\Ma\in\Rat(\shc)$.
\Lemma
For any $\Na\in \Proc(\cor[z],\shc)$,
$\Na\to\Na_\fl$ is an isomorphism in $\Rat(\shc)$.
\enlemma
\Proof
Since $z^m\Na_{\tor}\simeq0$ for some $m\in\Z_{>0}$, the canonical morphism
$$\indlim_{k\in\Z_{\ge0}}
\Hom_{\Proc(\cor[z],\shc)}(\Ma,\cor[z]z^{-k}\tens_{\cor[z]}\Na)
\to
\indlim_{k\in\Z_{\ge0}}
\Hom_{\Proc(\cor[z],\shc)}(\Ma,\cor[z]z^{-k}\tens_{\cor[z]}\Na_\fl)$$
is an isomorphism
for any $\Ma\in \Proc(\cor[z],\shc)$.
\QED

Hence a canonical functor
$$\Aff(\shc)\to\Rat(\shc)$$
is faithful and essentially surjective (i.e., any object of
$\Rat(\shc)$ is isomorphic to the image of an object of
$\Aff(\shc)$).

Let us denote by $K(\shc)$ the Grothendieck group of $\shc$.
It is a $\Z[q^{\pm1}]$-module, where $q$ acts on $K(\shc)$ by
$[M]\mapsto [qM]$.
We write $K(\shc)\vert_{q=1}$ for $K(\shc)/(q-1)K(\shc)$.

Since the following lemma is classical in other contexts,  we omit
its proof  (cf.\ \cite[Section 2.3]{CG97} ). 

\Lemma\label{lem:ratcl}
Let $\Ma\in \Aff(\shc)$.
\bnum\item 
Then, $[\Ma/z\Ma]\in K(\shc)\vert_{q=1}$
depends only on the isomorphism class of $\Ma$ in $\Rat(\shc)$.
\item Let $\Ma', \Ma,\Ma'' \in \Aff(\shc)$.
If there exists an exact sequence
$0\to \Ma'\to\Ma\to \Ma''\to 0$ in $\Proc(\cor[z],\shc)$, then
we have
$$[\Ma/z\Ma]=[\Ma'/z\Ma']+[\Ma''/z\Ma'']\qt{in $K(\shc)\vert_{q=1}$.}$$ 
\ee
\enlemma

\Lemma
The category $\Rat(\shc)$ satisfies the following properties.
\bnum
\item $\Rat(\shc)$ is abelian,
\item every object of $\Rat(\shc)$ has finite length.
\ee
\enlemma
\Proof
Since (i) is elementary, we omit its proof.

Let us show (ii). Since every object of $\Proc(\cor[z],\shc)$ is noetherian,
every object $\Ma$ of $\Rat(\shc)$ is noetherian.
Let us show that any decreasing sequence $\st{\Na_j}_{j\in\Z_{\ge0}}$ of  subobjects  of $\Ma$
is stationary. It is represented by a decreasing sequence $\st{\Na_j}_{j\in\Z_{\ge0}}$ of subobjects of $\Ma$ in $\Aff(\shc)$.
Hence the decreasing sequence $\ell(\Na_j/z\Na_j)$ is stationary. 
Then our assertion follows from
\eq\parbox{70ex}{Let $\Laa'$, $\Laa\in\Aff(\shc)$
satisfies $\Laa'\subset \Laa$ in $\Proc(\cor[z],\shc)$ and
$\ell(\Laa'/z\Laa')=\ell(\Laa/z\Laa)$,
then $\Laa'\to \Laa$ is an isomorphism in
$\Rat(\shc)$,}
\eneq
which is a consequence of Lemma~\ref{lem:ratcl}.
\QED

\Def
Let $\Ma$ be an \afn.
\bnum
\item
We say that $\Laa\subset\Ma$ is a strict \subafn of an \afn
 if $\Ma/\Laa$ is an \afn.
In other words, $\Laa\cap z\Ma=z\Laa$.
\item We say that $\Ma$ is  a {\em rationally simple} \afn
if 
$\Ma$ is  simple as an object of $\Rat(\shc)$.
\item
Let $\Laa$ be an \afn and let $\Ma\to \Laa$ be an epimorphism in $\Proc(\cor[z],\shc)$.
If $\Laa$ is a head of $\Ma$ in $\Rat(\shc)$,
we say that $\Laa$ is a {\em rational head} of $\Ma$.
\ee
\edf

\Lemma\label{lem:ratsimple}
Let $(\Ma,z)$ be an \afn.
If $\Ma/z\Ma$ is a simple object, then $\Ma$ is 
a rationally simple \afn.
\enlemma
\Proof
Let $\Na$ be a non-zero \afn and let $\Ma\epito\Na$
be an epimorphism in $\Rat(\shc)$. 
We may assume that it is an epimorphism in $\Proc(\cor[z],\shc)$, since any epimorphism in $\Rat(\shc)$ is represented by an epimorphism in $\Proc(\cor[z],\shc)$. 
Then $\Ma/z\Ma\epito\Na/z\Na$ is an epimorphism, and
$\Na/z\Na$
is non-zero by Proposition~\ref{prop:noether}\;\eqref{it:Nakayama}, and hence
it is an isomorphism.
Hence $\Ma\to\Na$ is an isomorphism 
by Lemma~\ref{lem:red}.
\QED
\Rem
The converse of Lemma~\ref{lem:ratsimple} is not true
(see Proposition~\ref{prop:defactor}). 
\enrem

\section{Coherent objects in Monoidal categories} \label{Sec: coh in mon}

\subsection{Graded monoidal categories }

In the sequel, 
let $\shc$ be an abelian $\cor$-linear graded monoidal category which satisfies
\eq&&\left\{\parbox{70ex}{
\bna
\item $\shc$ satisfies \eqref{cond:fcat},
\item $\tens$ is $\cor$-bilinear and bi-exact,
\item the unit object $\one$ is simple; in particular, 
$\END_{\shc}(\one)\simeq\cor$, 
\item $\tens$ commutes with the grading shift functor $q$, \\
i.e., $q(X\tens Y)\simeq (qX)\tens Y\simeq X\tens(qY)$,
\setcounter{mycc}{\value{enumi}}
\ee}\right.
\label{cond:exactmono}
\eneq
For generalities on monoidal categories, we refer the reader 
to \cite[Chapter 4]{KS06}, \cite{EGNO15}. 

By identifying $q$ with the invertible central object $q\one\in\shc$, we have
$qX\simeq q\tens X$.

The category $\Pro(\shc)$ has also a structure of monoidal category in which the tensor product $\tens$ is bi-exact.

Let $A$ be a commutative graded $\cor$-algebra satisfying \eqref{cond:gring}.
The functor $\cor\gmod\ni V\mapsto V\tensc\one\in\shc$ 
extends to a fully faithful monoidal functor
$\Pro(\cor\gmod)\to \Pro(\shc)$.
In the sequel, we regard $\Pro(\cor\gmod)$, as well as $\cor\gmod$, as a full
subcategory of $\Pro(\shc)$.
For example, $A$ will be identified with
$\proolim(A/A_{\ge m})\tensc\one\in\Pro(\shc)$ (cf.\ Lemma~\ref{lem:cohcor}).

For $\Ma,\Na\in \Modg(A, \Pro(\shc))$, define
$\Ma\tensa\Na\in\Modg(A, \Pro(\shc))$ by the universal property:
\eq 
\hs{2ex}&&\ba{l}
\Hom_{\Modg(A, \Pro(\shc))}(\Ma\tensa\Na,\Laa)\\
\hs{4ex}\simeq \{f\in\Hom_{\Pro(\shc)}(\Ma\tens\Na,\Laa)\mid\\
\hs{11ex}(a\id_{\Laa})\circ f=f\circ(a\id_{\Ma}\tens\id_\Na)
=f\circ(\id_\Ma\tens a\id_\Na)
\hs{2ex}\text{for any $a\in A$}\}.
\ea\label{Eq: ten_A}
\eneq
It is well-defined since $A$ is a finitely generated $\cor$-algebra.
Note that 
$\Modg(A, \Pro(\shc))$ is a monoidal category with $\tensa$ 
as its tensor product and $A\in \Modg(A, \Pro(\shc))$ as its unit object.
Note also that there is an epimorphism
$\Ma\tens\Na\epito\Ma\tensa\Na$ in $\Pro(\shc)$.

\Lemma  \label{lem:monproc}
The category $\Proc(A,\shc)$ has a structure of a monoidal category
by $\tensa\;$. The unit object is $A=\proolim[m](A/A_{\ge m})\tensc\one\in \Proc(A,\shc)$.
\enlemma
\begin{proof}
Let $\Ma$, $\Na\in\Proc(A,\shc)$.
For each $m\in \Z_{\ge0}$ we have
\eqn
\dfrac{\Ma \tensa\Na }{A_{\ge m} (\Ma\tensa\Na)} 
\simeq
 \dfrac{\Ma}{A_{\ge m}\Ma} \tensa\dfrac{\Na}{A_{\ge m}\Na} \in \Modg(A,\shc).
\eneqn
In particular,  $(\Ma\tensa \Na)/A_{>0}(\Ma\tensa\Na)\in\shc$.

Since $\proolim[m]$ is exact, we have
$$ \proolim[m] \left (\dfrac{\Ma \tensa\Na }{A_{\ge m}(\Ma\tensa\Na)} \right) \simeq 
\proolim[m] \left (\dfrac{\Ma}{A_{\ge m}\Ma} \tensa\dfrac{\Na}{A_{\ge m}\Na}\right)
\simeq \Ma\tensa\Na.
  $$
  Hence $\Ma\tensa\Na$ belongs to $\Proc(A,\shc)$.
Now it is immediate that $\bl\Proc(A,\shc), \tensa\,\br$
is a monoidal category. 
\end{proof}

Note that $\tensa$ is a right exact bi-functor.
\Lemma \label{lem:finflat}
Assume that $\dim_\cor A<\infty$.
Let $\eps$ denote the projection $A \epito A/A_{\ge0} \simeq \cor$.
\bnum
\item 
\eq A\tens A\tens A_{>0}\To[g_1] A\tens A_{>0}\To[g_0] A\label{eq:AAA}
\eneq
is an exact sequence in $\Modg(A)$. Here 
$$\text{$g_1(a\tens b\tens c)=ab\tens c-a\tens bc$ and
$g_0(a\tens b)=ab$.}$$
\item
For $M,N\in \Modg(A,\shc)$, the sequence 
\eq 
&&A\tens A_{>0}\tens M\tens N\To[f_1(M,N)]
A_{>0}\tens M\tens N\To[f_0(M,N)] M\tens N 
\label{eq:tencom}
\eneq
is a complex, i.e., the composition vanishes. Here $f_1(M,N)$
is given by 
\eqn
&&A\tens A_{>0}\ni a\tens b\mapsto
b\tens a\id_{M}\tens\id_{N}
+(a-\eps(a) ) \tens\id_{M}\tens b\id_{N}-ab\tens\id_{M}\tens \id_{N}\\
&&\hs{45ex}\in \Hom_\shc(M\tens N, A\tens M\tens N),
\eneqn
and $f_0(M,N)$ is given by
$$A\ni a\mapsto a\id_{M}\tens \id_{N}-\id_{M}\tens a\id_{N}
\in\Hom_\shc(M\tens N,M\tens N).$$
\item
If $M$ is $A$-flat, then
\eq A\tens A_{>0}\tens M\To[h_1(M)]
A_{>0}\tens M\To[h_0(M)]M\label{eq:resM}
\eneq
is exact.
Here $h_1(M)$ is given by
$A\tens A_0\ni a\tens b\mapsto b\tens a\id_{M}-ab\tens\id_{M}
\in\HOM_\shc (M,A_{>0}\tens M) $, and $h_0(M)$ is the multiplication.
\item $A_{>0}\tens M\tens N\To[f_0(M,N)] M\tens N\To M\tensa N\To0$ is exact.
\item If $M$ is $A$-flat, then \eqref{eq:tencom} is exact.
\item If $M$ is $A$-flat, then
$M\tensa\scbul\cl \Modg(A,\shc)\to\Modg(A,\shc)$ is
exact.
\ee
\enlemma
\Proof
(i) follows from
$a\tens b=g_1(1\tens a\tens b)+1\tens g_0(a\tens b)$.

\snoi
(ii) is straightforward.

\snoi
(iii) follows form (i) and \eqref{eq:resM}$\simeq$\eqref{eq:AAA}$\tens[A]M$.

\snoi
(iv) follows from the fact that $a\id_M\tens \id_N-\id_M\tens a\id_N=0$ if $a\in A_0$.

\snoi
(v) 
For $M, N\in\Modg(A,\shc)$, set
$S_2(M,N)=A\tens A_{>0}\tens M\tens N$ and
$S_1(M,N)=A_{>0}\tens M\tens N$.

Assume that $M$ is $A$-flat. For any $N\in\Modg(A,\shc)$, there exists $m\in\Z_{>0}$
such that $A_{\ge m}N=0$.
We shall show that \eqref{eq:tencom} is exact by induction on $m$.
Assume that $m=1$. Then we have
$f_k(M,N)=h_k(M)\tens\id_N$ ($k=0,1$).
Hence, \eqref{eq:tencom}$\simeq$\eqref{eq:resM}$\tens N$ is exact.

Now assume that $m>1$.

For an exact sequence
$0\to N'\to N\to N''\to 0$ in $\Modg(A,\shc)$, we have a commutative diagram with exact rows:
\eq
\xymatrix{
0\ar[r]&S_2(M,N') \ar[r] \ar[d]_{f_1(M,N')}&S_2(M,N)
\ar[r] \ar[d]_{f_1(M,N)}&S_2(M,N'')
\ar[d]_{f_1(M,N'')}\ar[r]&0\\
0\ar[r]&S_1(M,N') \ar[r] \ar[d]_{f_0(M,N')}&S_1(M,N)
\ar[r] \ar[d]_{f_0(M,N)}&S_1(M,N'')\ar[r]\ar[d]_{f_0(M,N'')}&0\\
0\ar[r]&M \tens N' \ar[r] \ar[d]& M\tens N\ar[r] \ar[d]& M\tens N''
\ar[r]\ar[d]&0\\
&M\tensa N' \ar[r] \ar[d]& M\tensa N\ar[d]\ar[r] &M\tensa N''\ar[d]\ar[r]  & 0\\
&0 & 0 &0 &
}\label{eq:bigd}
\eneq
Set $N'=A_{>0}N$ and $N''=N/N'$.
Then the left and the right column are exact by the induction hypothesis,
and hence the middle arrow is exact. Thus the induction proceeds.

\mnoi
(vi) Assume that $M$ is $A$-flat.
It is enough to show that $M\tensa N'\to M\tensa N$ is a monomorphism for any exact sequence $0\to N'\to N\to N''\to 0$ in $\Modg(A,\shc)$.
It follows from the fact that
the commutative diagram \eqref{eq:bigd}
has exact rows and columns.
\QED

\Prop\label{prop:tenszexact}
Assume that $\Ma \in \Proc(A,\shc)$ is $A$-flat.
Then the functor
$\Ma\tensa\scbul\cl  \Proc(A,\shc)\to \Proc(A,\shc)$ is exact.
\enprop
\Proof
Let us show that
\eq
0\to \Ma\tensa \Na'\to \Ma\tensa \Na\to \Ma\tensa \Na''\to0
\label{eq:exactMN}
\eneq
is exact if $0\to \Na'\to \Na\to \Na''\to 0$ is an exact sequence
in $\Proc(A,\shc)$.
For $m\in\Z_{>0}$, set
$A(m)=A/A_{\ge m}$, $M_m=\Ma/A_{\ge m}\Ma$,
$N_m=\Na/A_{\ge m}\Na$, $N''_m =\Na''/A_{\ge m}\Na''$
and $N'_m=\Na'/(\Na'\cap A_{\ge m}\Na)$.
Then $\dim A(m)<\infty$ and $M_m\in\Modg(A(m),\shc)$ 
is $A(m)$-flat.
Since $0\to N_m'\to N_m\to N_m''\to 0$ is an exact sequence in
$\Modg(A(m),\shc)$,
the sequence
$$0\to M_m\tens[A(m)] N_m'\to  M_m\tens[A(m)]N_m\to M_m\tens[A(m)] N_m''\to 0$$
is exact by Lemma~\ref{lem:finflat}.
Then taking $\proolim[m]$, 
we conclude that \eqref{eq:exactMN} is exact. 
\QED

\Lemma\label{lem:Affmon}
For $\Ma,\Na\in\Aff[A](\shc)$  {\rm (see Definition~\ref{def:aff})},
we have $\Ma\tensa \Na\in\Aff[A](\shc)$.
Namely $\Aff[A](\shc)$ is a monoidal category.
\enlemma
\Proof
Let $X'\to X\to X''$ be an exact sequence in
$\Modc(A)$.
Then $X'\tensa \Ma\to X\tensa \Ma\to X''\tensa \Ma$ is exact since
$\Ma$ is $A$-flat.
Now, Proposition~\ref{prop:tenszexact} implies that
$X'\tensa \Ma\tensa\Na\to X\tensa \Ma\tensa\Na\to X''\tensa \Ma\tensa\Na$ is exact since $\Na$ is $A$-flat.
\QED

\Prop\label{prop:AAflat}
Let $A'$ be another commutative  graded ring satisfying \eqref{cond:gring},
and let $\Ma\in\Proc(A,\shc)$, $\Ma'\in\Proc(A',\shc)$.
\bnum
\item $A\tensc A'$ satisfies \eqref{cond:gring}.
\item $\Ma\tens\Ma'\in \Proc(A\tensc A',\shc)$. 
\item For $X\in\Modc(A)$, $X'\in\Modc(A')$ and $m\in\Z_{\ge0}$, we have
$$\Tor_m^{A\tensc A'}(X\tensc X',\Ma\tens\Ma')\simeq
\soplus_{a,b\in\Z_{\ge0},\;a+b=m}\Tor_a^A(X,\Ma)\tens
\Tor_b^{A'}(X',\Ma').$$
\item
If $\Ma$ is $A$-flat and  $\Ma'$ is $A'$-flat, then $\Ma\tens \Ma'$ is $(A\tensc A')$-flat.
\ee
\enprop
\Proof
(i), (ii) and (iii) are immediate. 
(iv) follows from Proposition~\ref{prop:flatfin} and
\eqn
&&\Tor_1^{A\tensc A'}\bl (A\tensc A')/(A\tensc A')_{>0},\Ma\tens\Ma'\br\\
&&\hs{10ex}\simeq
\Tor_1^A( A/A_{> 0} ,\Ma)\tens \bl(A'/A'_{>0})\tensa\Ma'\br\\
&&\hs{20ex}\oplus\bl(A/A_{>0})\tensa\Ma\br\tens \Tor_1^{A'}(A'/A'_{>0},\Ma').
\eneqn
\QED

\subsection{Rigid case}

In this subsection, we assume that $\shc$  satisfies \eqref{cond:exactmono}
and also that $\shc$ is a rigid monoidal category, i.e.,
every object of $\shc$ has a right dual and a left dual.

We denote by $\D$ the right duality functor.
Hence its quasi-inverse $\D^{-1}$ is a left duality functor.
Let $$\ev_M\cl M\tens\D M\to \one\qtq
\coev_M\cl \one\to \D M\tens M$$ be the evaluation morphism and the coevaluation morphism, respectively.
Recall that
for $\Ma\in\Proc(A,\shc)$, we define in Definition~\ref{def:duala}
\eq
\ba{rl}
\D_A(\Ma)&\seteq\proolim[m]\D\bl(A/A_{\ge m})^*\tensa\Ma\br,\\
\D^{-1}_A(\Ma)&\seteq\proolim[m]\D^{-1}\bl(A/A_{\ge m})^*\tensa\Ma\br.
\ea
\eneq
If $\Ma\in\Aff[A](\shc)$, then $\D_A(\Ma)$ and
$\D_A^{-1}(\Ma)$ also belong to $\Aff[A](\shc)$.

They are a quasi-inverse to each other.

\Prop \label{Prop:rigid}
If $\shc$ is rigid, then
the monoidal category $\Aff[A](\shc)$ 
is also rigid, 
and $\D_A$ and $\D_A^{-1}$ are a right dual and a left dual, respectively.
\enprop

\Proof
We shall show that $\D_A(\Ma)$ is a right dual of $\Ma\in\Aff[A](\shc)$. 

Let us first construct $\ev\cl\Ma\tensa\D_A(\Ma)\to A\tensc\one$.
For $m\in\Z_{>0}$, we have 
$$\bl(A/A_{\ge m})^*\tensa\Ma\br\tens\D\bl(A/A_{\ge m})^*\tensa\Ma\br
\to\one.$$
Set $L_m=A/A_{\ge m}$.
Then $\st{L_m}_{m\in\Z_{\ge0}}$ is a projective system.
By using $\cor \to L_m\tensc L_m^*$, we obtain 
\eqn
\Ma\tens \D\bl L_m^*\tensa\Ma\br
&&\to (L_m\tensc L_m^*)\tensc\Ma\tens
\D(L_m^*\tensa\Ma)\\
&&\to L_m\tensc\bl(L_m^*\tensa\Ma\br\tens
\D\bl L_m^*\tensa\Ma\br\\
&&\to
 L_m\tensc\one.
\eneqn
Since it is $A$-bilinear, we obtain
$$\Ma\tensa\;\D_A(\Ma)\To \Ma\tensa\; \D\bl L_m^*\tensa\Ma\br\to L_m\tensc\one.$$
Taking $\proolim[m]$, we obtain
$$\ev\cl \Ma\tensa\;\D_A(\Ma)\to A\tensc\one.$$

Now let us construct the coevaluation morphism
$$\coev\cl A\tensc \one\to \D_A(\Ma)\tensa\Ma.$$
For any $m\in\Z_{\ge0}$, we have
$$f_m\cl \one \to \D\bl L_m^*\tensa\Ma\br\tens\bl L_m^*\tensa\Ma\br.$$
It satisfies the relation : $(a\tens 1)\cdot f_m=(1\tens a)\cdot f_m$ for any
$a\in A$.
Hence $f_m$ gives
$$\one \to
\ihom_{A\tensc A}\Bigl(A,\D\bl L_m^*\tensa\Ma\br\tens\bl L_m^*\tensa\Ma\br\Bigr).
$$
Here, we regard $A$ as the $A\tensc A$-module
$(A\tensc A)/\sum_{a\in A}(A\tensc A)(a\tensc 1-1\tensc a)$, and
$\D\bl L_m^*\tensa\Ma\br\tens\bl L_m^*\tensa\Ma\br\in\Modg(A\tensc A,\shc)$.

On the other hand,  we have
\eqn
\D\bl L_m^*\tensa\Ma\br\tens\bl L_m^*\tensa\Ma\br
&&\simeq
\bl L_m\tensa \D_A(\Ma)\br\tens (L_m^*\tensa\Ma)\\
&&\simeq
(L_m\tensc L_m^*)\tens_{A\tensc A}\bl \D_A(\Ma)\tens \Ma \br.
\eneqn
Since $ \D_A(\Ma) \tens \Ma$ is ($A\tensc A$)-flat by Proposition~\ref{prop:AAflat},
Lemma~\ref{lem:NMX} implies that
\eqn
&&\ihom_{A\tensc A}\Bigl(A,\D\bl L_m^*\tensa\Ma\br\tens\bl L_m^*\tensa\Ma\br\Bigr)
\\
&&\hs{10ex}\simeq
\Hom_{A\tensc A}(A,L_m\tensc L_m^*)\tens_{A\tensc A}\bl \D_A(\Ma)\tens\Ma\br.
\eneqn
Since
$\Hom_{A\tensc A}(A,L_m\tensc L_m^*)\simeq\Hom_A(L_m ,L_m)\simeq L_m$, 
we obtain
\eqn
&&\Hom_{A\tensc A}(A,L_m\tensc L_m^*)\tens_{A\tensc A}\bl\D_A(\Ma)\tens\Ma\br\\
&&\hs{10ex}\simeq
L_m\tens_{A\tensc A}\bl\D_A(\Ma)\tens\Ma\br\simeq
 L_m\tensa(\D_A(\Ma)\tensa \Ma).
\eneqn
Thus we obtain
$\one \to  L_m\tensa(\D_A(\Ma)\tensa \Ma)$.
Taking the projective limit $\proolim[m]$, we obtain
$\one\to \D_A(\Ma)\tensa \Ma$ in $\Pro(\shc)$,
which induces a morphism
$$\coev\cl A\tensc\one\to \D_A(\Ma)\tensa \Ma$$ in $\Modg(A,\Pro(\shc))$.

Note that
$\Ma\To[\coev]\Ma\tensa\D_A(\Ma)\tensa \Ma\To[\ev]\Ma$
and $\D_A(\Ma)\To[\coev]\D_A(\Ma)\tensa \Ma\tensa\D_A(\Ma)\To[\ev]\D_A(\Ma)$
are isomorphisms 
 by Lemma~\ref{lem:red}
since
they remain isomorphisms after the operation of
$(A/A_{>0})\tensa\scbul$.
Hence $\D_A(\Ma)$ is a right dual of $\Ma$ (see \cite[Lemma A.2]{K^3}).
\QED

\Lemma\label{lem:tensdual}
For $X\in\shc$ and $\Ma\in\Aff[A](\shc)$, we have
$$
 \D_A^{ \pm1 } (\Ma\tens X)\simeq\bl  \D^{\pm 1} (X)\br\tens\bl \D_A^{\pm 1}(\Ma)\br\qtq
\D_A^{\pm1}(X\tens\Ma)\simeq\bl\D_A^{\pm1}(\Ma)\br\tens\bl\D^{\pm1}(X)\br.$$
\enlemma
\Proof
Let us only show the first isomorphism since the proof of the other is similar.
We have
\eqn\D_A^{\pm}(\Ma\tens X)&&
\simeq\proolim[m]\D^{\pm}\bl(A/A_{\ge m})^*\tensa(\Ma\tens X)\br\\
&&\simeq\proolim[m]\D^{\pm}\Bigl(\bl(A/A_{\ge m})^*\tensa\Ma\br\tens X\Bigr)\\
&&\simeq\proolim[m](\D^{\pm} X)\tens \D^{\pm}\bl(A/A_{\ge m})^*\tensa\Ma\br\\
&&\simeq(\D^{\pm} X)\tens (\D_A^{\pm}\Ma).
\eneqn
\QED

\section{R-matrices} \label{Sec: R-matrices}

\subsection{\Afns}
Let $\shc$ be an abelian graded $\cor$-linear monoidal category which satisfies
\eqref{cond:exactmono}.

Recall that  $\Aff(\shc)$ is a monoidal category with $\tensz$ as a tensor product, where we write
 $\tensz$ for $\tens[{\cor[z]}]$.

\Lemma
Let $(\Ma,\z)$ be an \afn in $\shc$
and let $(\Na,\zN)\in\Proc(\cor[\zN],\shc)$.
Then for any homogeneous polynomial $f(\z,\zN)\in\cor[\z,\zN]$
which is monic in $\z$,
the morphism $f(\z,\zN)\vert_{\Ma\tens\Na}$ is a monomorphism.
\enlemma
\Proof
It is enough to show that
$f(\z,\zN)\vert_{\Ma\tens(\Na/\zN^m\Na)}$ is a monomorphism for any $m\in\Z_{\ge 0}$.
We argue by induction on $m$.
If $m=0$, then it is obvious.
Assume that $m>0$.
Then, in the commutative diagram with exact rows
$$\xymatrix{
0\ar[r]&\Ma\tens (\zN^{m-1}\Na/\zN^{m}\Na)\ar[r]\ar[d]^{f(\z,0)}&\Ma\tens
(\Na/\zN^{m}\Na)\ar[r]\ar[d]^{f(\z,\zN)}
&\Ma\tens(\Na/\zN^{m-1}\Na)\ar[r]\ar[d]^{f(\z,\zN)}&0\\
0\ar[r]&\Ma\tens (\zN^{m-1}\Na/\zN^{m}\Na)\ar[r]&\Ma\tens(\Na/\zN^{m}\Na)\ar[r]
&\Ma\tens(\Na/\zN^{m-1}\Na)\ar[r]&0,}
$$
the left and the right arrows are monomorphisms and so is
the middle.
\QED

\Prop \label{pro:onedimhom}
Let $(\Ma,\z)$ and
$(\Na,\zN)$ be \afns.
Set $M=\Ma/\z\Ma$ and $N=\Na/\zN\Na$.
\bnum 
\item Assume that 
$\HOM_\shc(M\tens N, M\tens N)=\cor\,\id_{M\tens N}$. Then we have
$$\HOM_{\cor[\z,\zN]}(\Ma\tens\Na,\Ma\tens\Na)=\cor[\z,\zN]\id_{\Ma\tens \Na}.$$
\item Assume that $\dim_\cor\HOM_\shc(M\tens N, N\tens M)=1$ and 
$\HOM_{\cor[\z,\zN]}(\Ma\tens\Na,\Na\tens\Ma)\not=0$.
Then there exists $\Rre_{\Ma,\Na}\in\HOM_{\cor[\z,\zN]}(\Ma\tens\Na,\Na\tens \Ma)$
such that
\eq\HOM_{\cor[\z,\zN]}(\Ma\tens\Na,\Na\tens \Ma)
=\cor[\z,\zN]\Rre_{\Ma,\Na}
\eneq
and $\Rre_{\Ma,\Na}\vert_{\z=\zN=0}\in\HOM_\shc(M\tens N, N\tens M)$
does not vanish.
{\em We call $\Rre_{\Ma,\Na}$ the renormalized R-matrix.} 
\ee
\enprop

\Proof 
Note that $\Ma\tens\Na \in \Proc(\cor[\z,\zN],\shc)$. Hence it
 immediately follows from Proposition \ref{prop:uniquehom}.
\QED

\Def \label{def: rmat}
If $M,N\in\shc$ satisfy
$\dim\HOM(M\tens N,N\tens M)=1$, then a non-zero
morphism $\rmat{}\in \HOM(M\tens N,N\tens M)$ is called the R-matrix between $M$ and $N$ and denoted by $\rmat{M,N}$. It is well-defined up to a constant multiple.
\edf

If $\shc$ is a rigid monoidal category, we write
$\Da^{\pm1}$ for $\D^{\pm1}_{\cor[z]}$, where $\D$ is the right dual functor of $\shc$
(see Definition~\ref{def:duala}).
Hence,
Proposition~\ref{Prop:rigid} implies the following proposition.
\Prop
Assume that $\shc$ is rigid.
Then $\Aff(\shc)$ is a rigid monoidal category with
$\Da^{\pm1}$ as their right and left duality functors.
\enprop

\subsection{Rational centers and affinizations}
Let $\shc$ be a $\cor$-linear graded monoidal category which satisfies
the following conditions:
\eq
&&\hs{1ex}\left\{\parbox{70ex}{
$\shc$ satisfies \eqref{cond:exactmono} and the following additional condition:
\bna
\setcounter{enumi}{\value{mycc}}
\item $\shc$ has a decomposition $\shc=\soplus_{\la\in\Idx}\shc_\la$
compatible with a monoidal structure where $\Idx$ is an abelian monoid,
and $\one\in \shc_\la$ with $\la=0$. 
\ee
}\right.\label{cond:g}
\eneq

Let $z$ be a homogeneous indeterminate with degree $d\in\Z_{>0}$.

We have also bi-exact bifunctors:
$$\tens\;\cl \shc\times\Rat(\shc)\to \Rat(\shc)
\qtq \tens\;\cl \Rat(\shc)\times\shc\to \Rat(\shc).$$

\Def
A {\em rational center} in $\shc$ is a triple $(\Ma,\phi,\Rmat_\Ma)$
of $\Ma\in\Aff(\shc)$, an additive map $\phi\cl\Idx\to\Z$ and an isomorphism 
$$\Rmat_\Ma(X)\cl q^{\tens\,\phi(\la)}\tens\Ma\tens X\isoto X\tens\Ma$$
in $\Rat(\shc)$ functorial in $X\in\shc_\la$ such that
$$
\xymatrix@C=15ex
{q^{\tens\,\phi(\la+\mu)}\tens\Ma\tens X\tens Y\ar[r]^{\Rmat_\Ma(X)\tens Y}\ar[dr]_{\Rmat_\Ma(X\tens Y)}&q^{\tens\,\phi(\mu)}\tens X\tens\Ma\tens Y\ar[d]^{X\tens\Rmat_\Ma(Y)}\\
&X\tens Y\tens\Ma}$$
and $$\xymatrix@C=12ex{\Ma\tens\one\ar[r]^{\Rmat_\Ma(\one)}\ar[dr]^-{\sim}&\one\tens \Ma\ar[d]^\bwr\\
&\Ma}
$$
commute in $\Rat(\shc)$ for any $X\in\shc_\la$ and $Y\in\shc_\mu$ ($\la,\mu\in\Idx$).
\edf
Note that the commutativity of the bottom diagram is a consequence of the one of the top.
{\em In the sequel, we neglect grading shifts.}
\Lemma \label{lem:dualAff}
Assume that $\shc$ is a rigid monoidal category.
Let $(\Ma,\Rmat_{\Ma})$ be a rational center.
Then $(\Da^{\pm1}(\Ma),\Rmat_{\Da^{\pm1}(\Ma)})$  is a rational center.
Here,
$\Rmat_{\Da^{\pm1}(\Ma)}(X)\seteq\Da^{\pm1}\bl\Rmat_{\Ma}(\D^{\mp1}(X))\br$. 

\enlemma

Note that
$$\Rmat_{\Ma}(\D^{\mp1}(X))\cl\Ma\tens\D^{\mp1}(X)\isoto\D^{\mp1}(X)\tens\Ma$$
and
$$
\Da^{\pm1}\bl\Rmat_{\Ma}(\D^{\mp1}(X))\br:\xymatrix{\Da^{\pm1}\bl\D^{\mp1}(X)\tens\Ma\br\ar[r]^\sim
\ar[d]^\sim&\Da^{\pm1}\bl\Ma\tens\D^{\mp1}(X)\br\ar[d]^\sim\\
\Da^{\pm1}(\Ma)\tens X\ar[r]^\sim&X\tens\Da^{\pm1}(\Ma),}$$
where the vertical arrows follow from Lemma~\ref{lem:tensdual}. 

\Def\label{def:affinization}
An {\em affinization} $\Ma$ of $M\in\shc$ is an \afn
$(\Ma,\z)$ with a rational center $(\Ma,\Rmat_\Ma)$ and an isomorphism
$\Ma/\z\Ma\simeq M$.
\edf

We sometimes simply write $\Ma$ for affinization if no confusion arises.  
The following lemma is an immediate consequence of Lemma~\ref{lem:ratcl}. 
\Lemma
Let $M$ and $N$ be objects of $\shc$.
Assume that $M$ admits an affinization.
Then we have
$$[M]\cdot[N]\equiv [N]\cdot[M]\quad\bmod (q-1)K(\shc).$$
\enlemma

\Prop 
Let $(\Ma,\Rmat_{\Ma})$ be a rational center in $\shc$, and let $L\in\shc$.
Assume that $\Ma$ and $L$ do not vanish.
Then there exist $m\in\Z$ and  a morphism
$\Rre_{\Ma,L}\cl \Ma\tens L\to L\tens \Ma$ in $\Proc(\cor[z],\shc)$ such that
\bna
\item $\Rre_{\Ma,L}$ is equal to $z^m\Rmat_{\Ma}(L)\cl \Ma\tens L\to L\tens \Ma$ in $\Rat(\shc)$,
\setcounter{myc}{\value{enumi}}
\item $\Rre_{\Ma,L}\vert_{z=0}\cl (\Ma/z\Ma)\tens L\to L\tens (\Ma/z\Ma)$
does not vanish.
\ee
\addtocounter{myc}{1}
Moreover such an integer $m$ and an $\Rre_{\Ma,L}$ are unique.

Similarly,  there exist $m\in\Z$ and  a morphism $\Rre_{L,\Ma}\cl L\tens\Ma\to \Ma\tens L$ in $\Proc(\cor[z],\shc)$ such that  
\bna
\setcounter{enumi}{\value{myc}}
\item $\Rre_{L,\Ma}$ gives $z^m\Rmat_{\Ma}(L)^{-1}\cl L\tens \Ma\to \Ma\tens L$ in $\Rat(\shc)$,
\item $\Rre_{L,\Ma}\vert_{z=0}\cl L\tens (\Ma/z\Ma)\to (\Ma/z\Ma)\tens L$
does not vanish.
\ee

\enprop
\Proof
The assertions  (a) and (b) follow  by taking  $m$  the smallest integer such that
$\Rmat_{\Ma}(L)$ is represented by an $f \in \HOM_{\Aff(\shc)}\bl\Ma\tens L, \cor[z]z^{-m}\tens_{\cor[z]} (L\tens \Ma)\br$.
The proof for (c) and (d) are similar.
\QED

\Th \label{th:ren_r_matrix} 
Let $(\Ma,\Rmat_{\Ma})$ be an affinization of $M\in\shc$, and
let $(\Na,\zN)$ be an \afn of $N\in\shc$.
Assume that
$\dim \HOM_{\shc}(M\tens N,N\tens M)=1$.
Then there exist 
a homogeneous $f(\z,\zN)\in \bl\cor[\z,\z^{-1}]\br[[\zN]]$ and a morphism
$\Rre_{\Ma,\Na}\cl \Ma\tens \Na\to \Na\tens \Ma$ in $\Proc(\cor[\z,\zN],\shc)$
such that 
\bna
\item$\HOM_{\cor[\z,\zN]}(\Ma\tens\Na,\Na\tens \Ma)
=\cor[\z,\zN]\Rre_{\Ma,\Na}$,
\item
as an element of 
$\HOM_{\Rat(\shc)}\bl \Ma\tens(\Na/\zN^k\Na),\;(\Na/\zN^k\Na)\tens\Ma\br$,
we have $\Rre_{\Ma,\Na}\vert_{\Na/\zN^k\Na}=f(\z,\zN)\Rmat_\Ma(\Na/\zN^k\Na)$
 for any $k\in\Z_{>0}$, 
\item $\Rre_{\Ma,\Na}\bigm|_{\z=\zN=0}$ does not vanish,
\item $f(\z,\zN)\vert_{\zN=0}$ is a monomial of $\z$.
\ee
Moreover such $\Rre_{\Ma,\Na}$ and $f(\z,\zN)$ are unique.
\enth
\Proof Take $\rmat{}$ such that $\HOM_{\shc}(M\tens N,N\tens M)=\cor\rmat{}$.
Set $\la=\deg(\rmat{})$.
By replacing $\Rmat_\Ma$ with $c z_\Ma^m \Rmat_\Ma$ for some $m\in\Z$ and $c\in\cor^\times$,
we may assume from the beginning that $\Rmat_\Ma(N)$ is
in $\HOM_{\cor[\z]} (\Ma\tens N,N\tens\Ma)$ and
$\Rmat_\Ma(N)\vert_{\z=0}=\rmat{}$. In particular, $\Rmat_{\Ma}(N)$
has degree $\la$.

Set $d=\deg(\z)$ and $d_\Na=\deg(\zN)$.
Let $C$ be the ring of homogeneous functions 
in $\cor[\z^{-1},\zN]$ with degree $0$.
For $k\in\Z_{\ge0}$, let $C(k)=\cor\cdot (\zN^k/\z^{kd_\Na/d})$ or $0$ according that
$kd_\Na/d$ is an integer or not.
Then we have
$C=\soplus_{k\ge0}C(k)$. Set $C(\le k)=\soplus_{0\le j\le k}C(j)$.

For $m\in\Z_{\ge0}$, set $N_m=\Na/\zN^m\Na\in\Mod(\cor[\zN],\shc)$.
Let us show the following statement by induction on $m$.
\eq
&&\hs{4ex}\parbox{72ex}{
for any $m\in\Z_{\ge1}$, there exists $f_m(\z,\zN)\in C({\le m-1})$
such that $f_m(\z,0)=1$ and
$R_m\seteq f_m(\z,\zN)\Rmat_{\Ma}(N_m)$ 
is a morphism $R_m\cl \Ma\tens N_m\to N_m\tens\Ma$ 
in $\Proc(\cor[\z,\zN],\shc)$.}\label{cond:indh}
\eneq
Since it is trivial for $m=1$, assume that $m>1$.
Then by the induction hypothesis,
we have $f_{m-1}(\z,\zN)$ and $R_{m-1}$ as in \eqref{cond:indh}.
Then $f_{m-1}(\z,\zN)\Rmat_\Ma(N_m)$ is a morphism in $\Rat(\shc)$.
Take the smallest integer $s\ge 0$ which satisfies the following condition:
there exists $a\in C(m-1)$ such that
$g(\z,\zN)\Rmat_\Ma(N_m)\cl \Ma\tens N_m\to N_m\tens \z^{-s}\Ma$ is a morphism in
$\Proc(\cor[\z],\shc)$, where $g(\z,\zN)=f_{m-1}(\z,\zN)+a$.

Then we have a commutative diagram in $\Proc(\cor[\z],\shc)$
$$\xymatrix@C=2.2ex{0\ar[d]&&0\ar[d]&0\ar[d]\\
\Ma\tens\zN^{m-1}N_m\ar[rr]\ar[d]&&\zN^{m-1}N_m\tens \z^{-s}\Ma\ar[d]\ar@{.>}[r]
&\zN^{m-1}N_m\tens (\z^{-s}\Ma/\z^{1-s}\Ma)\ar[d]\\
\Ma\tens N_m\ar[rr]_{g(\z,\zN)\Rmat_\Ma(N_m)}\ar[d]
\ar@{.>}[rrru]|-{\raisebox{-2ex}{$ $}}|(.55)\hole&&N_m\tens \z^{-s}\Ma\ar[d]
\ar@{.>}[r]
&N_m\tens (\z^{-s}\Ma/\z^{1-s}\Ma)\ar[d]\\
\Ma\tens N_{m-1}\ar[rr]_{R_{m-1}}\ar[d]\ar[dr]&&N_{m-1}\tens \z^{-s}\Ma\ar[d]\ar@{.>}[r]
&N_{m-1}\tens (\z^{-s}\Ma/\z^{1-s}\Ma)\ar[d]\\
0&N_{m-1}\tens \Ma\ar[ur]&0&0
}$$\cmtMH{$c\rmat{}$ is erased in the diagram}

Assume that $s>0$.
Then the composition 
$$\Ma\tens N_m\To[{g\,\Rmat_\Ma(N_m)}]
N_m\tens \z^{-s}\Ma\to N_m\tens (\z^{-s}\Ma/\z^{1-s}\Ma)$$
factors through $\zN^{m-1}N_m\tens (\z^{-s}\Ma/\z^{1-s}\Ma)$.
Hence the resulting morphism
$\Ma\tens N_m\to \zN^{m-1}N_m\tens (\z^{-s}\Ma/\z^{1-s}\Ma)$ induces
a non-zero morphism $\vphi\cl M\tens N\to q^{(m-1)d_\Na-ds}N\tens M$ 
of degree $\la$.
Since $\HOM(M\tens N,N\tens M)$ is concentrated in degree $\la$, we have
$s=(m-1)d_\Na/d$.
Take $c\in\cor$ such that $\vphi=c\rmat{}$,
Then $R_m\seteq (g-c\zN^{m-1}/\z^{s})\Rmat_{\Ma}(N_m)$
sends $\Ma\tens N_m$ to $N_m\tens \z^{1-s} \Ma$.
It contradicts the choice of $s$.

Hence we have $s=0$, and the condition \eqref{cond:indh}
is satisfied by setting $f_m=g$.

Since $f_m$ in \eqref{cond:indh}
is unique by Proposition~\ref{prop:base}, we have $f_m\equiv f_{m-1}\bmod \zN^{m-1}$,
and we obtain the desired result.
\QED

\Rem
In Theorem~\ref{th:ren_r_matrix},
 $f(\z,\zN)\in \bl\cor[\z,\z^{-1}]\br[[\zN]]$ cannot be weakened by
 $f(\z,\zN)\in \bl\cor[\z,\z^{-1}]\br[\zN]$.
We shall give two  examples.
\bnum
\item
Let $\g=\cor t$ be the one-dimensional graded Lie algebra with $\deg(t)=1$.
Let $\shc$ be the monoidal category of finite-dimensional graded $\g$-modules.
Take $\Ma=\cor[z]$ with the trivial action of $t$, namely $t\vert_{\Ma}=0$. 
For $X\in\shc$, we define $\Rmat_\Ma(X)\cl\cor[z]\tens X\to X\tens\cor[z]$
by $a\tens b\mapsto \e^{t/z}(b\tens a)$. Then $(\Ma,\Rmat_\Ma)$ is a rational center.  Let $\Na=\cor[t]$ with $\zN=t$.
Then, $f(\z,\zN)=\e^{-t/z}$ and $\Rre_{\Ma,\Na}(a\tens b)=b\tens a$. 
\item
Let $R$ be a quiver Hecke algebra (see \S\,\ref{subsec:QHA}
for the definition and notations).
Let $\shc=R\gmod$, $i\in I$, and $z$ a homogeneous
 indeterminate with degree $(\al_i,\al_i)$,
and $\vphi(w)\in\cor[[w]]$ with $\vphi(0)\in\cor^\times$.
Let $\Ma=\cor[z]\in\Aff(\shc)$.
For $\beta\in\prtl$ with $n=\height{\beta}$ and
$X\in R(\beta)\gmod$, let $\Rmat_\Ma(X)\cl \Ma\tens X\to X\tens\Ma$ be the 
morphism 
$\Ma\tens X\simeq\cor[z]\tensc X\To X\tens\Ma\simeq\cor[z]\tensc X$
in $\Rat(\shc)$ given by
$$\sum_{\nu\in I^\beta}\hs{2ex}\prod_{1\le k\le n\;;\;\nu_k=i}\hs{-2ex}\vphi\bl z^{-1}x_ke(\nu)\br\,.$$
Note that $\Rmat_\Ma(X)\in Z\bl R(\beta)\br[[z^{-1}]]$.
Then $(\Ma,\Rmat_\Ma)$ is a rational center and $\Rre_{\Ma,\Na}=\id_{\cor[z]\tens \Na}$.
\ee
\enrem

When $\deg(\z)=\deg(\zN)$, since $\Rre_{\Ma,\Na}$ commutes with $\zM$ and $\zN$, it induces a morphism in $\Aff(\shc)$
\eq \label{eq: induced_r}
 \Ma \tens_z \Na \To[\bRre_{\Ma,\Na}] \Na\tens_z \Ma\,,
\eneq
which is  denoted  by $\bRre_{\Ma,\Na}$. Here, $z$ acts on $\Ma$ and $\Na$ by $\z$ and $\z[\Na]$, respectively.

\smallskip
Recall that a simple object $M$ in a monoidal abelian category is called \emph{real} if $M\tens M$ is simple. 
\Def 
We say that a simple object $M\in \shc$ is \emph{\afr} if $M$ is real and there is an affinization $(\Ma,z)$ of $M$.
If $\deg(z)=d$, we say that $M$ is \emph{\afr of degree $d$}.
\edf

\Prop \label{prop:nuk=nuk+1} 
Let $(\Ma,z)$ be an affinization
of a real simple $M\in\shc$.
Then, there exists $c\in\cor^\times$ and
$T\in\HOM_{\Pro(\shc)}(\Ma\tens\Ma,\Ma \tens\Ma)$ such that
$$c\Rre_{\Ma,\Ma}-\id_{\Ma\tens \Ma }=(z \tens\id_{\Ma}-\id_{\Ma} \tens z)T.$$
\enprop
\Proof
By Lemma~\ref{lem:Affmon},  $\Ma\tensz \Ma$ is an \afn such that $(\Ma\tensz \Ma)/z(\Ma\tensz \Ma) \simeq M\tens M$.
Since $M\tens M$ is simple,  $\END_{\cor[z]}(\Ma\tensz \Ma) = \cor[z] \id_{\Ma\tensz \Ma}$ by Lemma \ref{lem:endpreaff}.
Note that $\bRre_{\Ma,\Ma}\in \END_{\cor[z]}(\Ma\tensz \Ma)$  so that 
$\bRre_{\Ma,\Ma}=cz^a$ for some $a\in \Z_{\ge 0}$ and $c\in \cor^\times$. By Theorem \ref{th:ren_r_matrix} (c), we have $\bRre_{\Ma,\Ma}\vert_{z=0}  \neq 0$ and hence $a=0$.
Thus $\bRre_{\Ma,\Ma} - c\,\id_{\Ma\tensz \Ma}=0$,
from which the assertion follows.
\QED

\begin{definition}  
Let $M$ and $N$ be simple objects in $\shc$. 
\begin{enumerate}[(i)]
\item
  Assume that $\dim\HOM_\shc(M\tens N,N\tens M)=\dim\HOM_\shc (N\tens M,M\tens N)=1$. 
Then define (see Definition \ref{def: rmat})
\eq
&&\La(M,N)\seteq\deg(\rmat{M,N})\in\Z,\\
&&\de(M,N) \seteq  \dfrac{1}{2} \bl\La(M,N)+\La(N,M)\br\in \dfrac{1}{2}\Z_{\ge 0}. \label{eq:de}
\eneq
\item 
Let $(\Ma, \Rmat_{\Ma})$ and  $(\Na, \Rmat_{\Na})$
be affinization of $M$ and $N$, respectively, and assume that $\dim \END_\shc(M\tens N)=1$.
Then define $\Daf(\Ma,\Na) \in \cor[\z,z_{\Na}]  $ by
\eq
 \Rre_{\Na,\Ma} \circ \Rre_{\Ma,\Na}  = \Daf(\Ma,\Na) \id_{\Ma\tens \Na}. \label{eq:Daf}
\eneq
\ee
Note that  $\Daf(\Ma,\Na)$ exists by Proposition~\ref{pro:onedimhom},
and it is well-defined up to a constant multiple in $\cor^\times$.
\end{definition}

If moreover $\dim\END_\shc(N\tens M)=1$,
since
$ \Rre_{\Ma,\Na} \circ \Rre_{\Na,\Ma} \circ \Rre_{\Ma,\Na}  = \Daf(\Ma,\Na) \Rre_{\Ma,\Na}  = \Daf(\Na,\Ma) \Rre_{\Ma,\Na} $,
we have $ \Daf(\Ma,\Na) = \Daf(\Na,\Ma)$.

 Note that $\de(M,N)$ is an integer in the quiver Hecke algebra case
(cf.\ \cite[Proposition 2.5]{KP18}).
As for criterions for the dimension $1$ of the hom spaces, see 
Proposition~\ref{prop:simplehd} and Corollary~\ref{cor:dimEND} below.

\subsection{Quasi-rigid Axiom}

\Def
Let $\sha$ be a monoidal category.
We say that $\sha$ is a {\em \KO} category if it satisfies:
\bna
\item $\sha$ is abelian and $\tens$ is bi-exact,
\item for any $L,M,N\in\sha$, $X\subset L\tens M$ and
$Y\subset M\tens N$ such that $X\tens N\subset L\tens Y\subset L\tens M\tens N$, there exists $K\subset M$ such that
$X\subset L\tens K$ and $K\tens N\subset Y$,
\item for any $L,M,N\in\sha$, $X\subset M\tens N$ and
$Y\subset L\tens M$ such that $L\tens X\subset Y\tens N$, there exists $K\subset M$ such that
$X\subset K\tens N$ and $L\tens K\subset Y$.
\ee
\edf

\Lemma
Let $\shc$ be a \KO monoidal category such that $\one$ is simple.
Let $M$ and $N$ be objects of $\shc$.
If $M\tens N\simeq0$, then $M\simeq0$ or $N\simeq0$.
\enlemma
\Proof
Assume that $M\tens N\simeq0$.
Set $X\seteq M\tens\one\subset M\tens\one$ and $Y\seteq0\subset\one\tens N$.
Then we have $X\tens N\subset M\tens Y\subset M\tens \one \tens N$.
Hence there exists
$K\subset \one$ such that
$X\subset M\tens K$ and $K\tens N\subset Y$.
Since $\one$ is simple, $K\simeq0$ or $K\simeq\one$.
If $K\simeq0$, then $X\subset M\tens K$ implies $M\simeq0$.
If $K\simeq \one$, then $K\tens N\subset Y$ implies $N\simeq0$.
\QED

\Lemma
An abelian rigid monoidal category with bi-exact tensor product is \KO.
\enlemma
\Proof
Since it is a well-known result and the proof is similar to the one in Lemma \ref{lem:K3OAff} below, we omit its proof.
\QED

\Conj \label{Conj: KO}
A \KO monoidal category which satisfies \eqref{cond:exactmono}
is embedded into a rigid monoidal category.
\enconj

Recall that a simple object $M$ in an abelian monoidal category is called \emph{real} if $M\tens M$ is simple.
In this subsection, we keep the assumption that
$\shc$ is a monoidal abelian category satisfying
\eqref{cond:g}.

Recall that, for $M,N\in\shc$, we denote by
$M\hconv N$ the head of $M\tens N$ and
$M\sconv N$ the socle of $M\tens N$.

\Prop \label{prop:simplehd}
Assume that $\shc$ is \KO.
Let $M$ be an \afr object of $\shc$ and
$N$ a simple object of $\shc$.
Then,
$M\tens N$ and $N\tens M$ have simple heads and simple socles.
Moreover,
$\dim\HOM_\shc(M\tens N,N\tens M)=\dim\HOM_\shc(N\tens M,M\tens N)=1$ and
$$M\hconv N\simeq\Im(\rmat{M,N})\simeq N\sconv M\qtq
N\hconv M\simeq\Im(\rmat{N,M})\simeq M\sconv N$$
up to grading shifts. 
\enprop
Recall that $\rmat{M,N}$ is an R-matrix
between $M$ and $N$, i.e., a generator of $\HOM_\shc(M\tens N,N\tens M)$. 

\Proof Since the same proofs in
\cite[Theorem 3.2]{KKKO15} and 
\cite[Proposition 3.2.9]{KKKO18} work, we omit the proof.
\QED

\Prop\label{prop:La}
Let $M$ be an \afr object of $\shc$ and
$N$ be a simple object of $\shc$.
Then we have
\bnum
\item
$\La(M,M\htens N)=\La(M,N)$ and $\La(N\htens M,M)=\La(N,M)$,
\item for any simple subquotient $S$ of 
the radical $\Ker(M\tens N\to M\htens N)$,
we have $\La(M,S)<\La(M,N)$,\label{item:srad}
\item for any simple subquotient $S$ of $(M\tens N)/(M\stens N)$,
we have $\La(S,M)<\La(N,M)$,
\item for any simple subquotient $S$ of 
the radical $\Ker(N\tens M\to N\htens M)$,
we have $\La(S,M)<\La(N,M)$,
\item for any simple subquotient $S$ of $(N\tens M)/(N\stens M)$,
we have $\La(M,S)<\La(M,N)$.
\ee
In particular,
$M\htens N$ appears in the composition series of $M\tens N$
only once \ro up to a grading\rf.
\enprop
\Proof
The proof is similar to \cite[Theorem 4.1]{KKKO18}.
\QED

\Cor\label{cor:dimEND}
Assume that $M\in\shc$ is \afr and $N\in\shc$ is simple.
Then we have $\END_\shc(M\tens N)=\cor\id_{M\tens N}$
and $\END_\shc(N\tens M)=\cor\id_{N\tens M}$.
\encor
\Proof
We shall only prove $\END_\shc (M\tens N)=\cor\id_{M\tens N}$, since the other assertion can be similarly proved.
Let $f\in \END_\shc(M\tens N)$.
Since $M\tens N$ has a simple head, the composition $M\tens N\To[f] M\tens N\to M\htens N$
factors through $M\htens N$.
Then the resulting endomorphism of $M\htens N$ should be $c\id_{M\htens N}$ for some $c\in\cor$.
Replacing $f$ with $f-c\id_{M\tens N}$, we may assume that
$f(M\tens N)\subset\Ker(M\tens N\to M\htens N)$.
If $f(M\tens N)$ is non-zero, then a simple quotient $S$ of
$f(M\tens N)$ should  be isomorphic to $M\htens N$.
Hence $\La(M,S)=\La(M,M\htens N)=\La(M,N)$, which contradicts Proposition~\ref{prop:La}\;\eqref{item:srad}.
\QED
\Lemma
Let $\shc$ be a \KO monoidal category.
Let $(\Ma,\Rmat_{\Ma})$ be an affinization of a real simple $M\in\shc$, and
let $(\Na,\zN)$ be an \afn of a simple $N\in\shc$.
Then there exist 
a homogeneous $f(\z,\zN), g(\z,\zN)\in \bl\cor[\z,\z^{-1}]\br[[\zN]]$ 
and morphisms
$\Rre_{\Ma,\Na}\cl \Ma\tens \Na\to \Na\tens \Ma$ in $\Proc(\cor[\z,\zN],\shc)$
and
$\tRre_{\Na,\Ma}\cl \Na\tens \Ma\to \Ma\tens \Na$ in $\Proc(\cor[\z,\zN],\shc)$
such that 
\bna
\item $f(\z,\zN)\vert_{\zN=0}$ and $g(\z,\zN)\vert_{\zN=0}$ are monomials of $\z$,
\item for any $k\in\Z_{>0}$, we have
$$\Rre_{\Ma,\Na}\vert_{\Na/\zN^k\Na}=f(\z,\zN)\Rmat_\Ma(\Na/\zN^k\Na)
\ \text{and}\ 
\tRre_{\Na,\Ma}\vert_{\Na/\zN^k\Na}=g(\z,\zN)\Rmat_\Ma(\Na/\zN^k\Na)^{-1},$$
\item $\Rre_{\Ma,\Na}\vert_{\z=\z[\Na]=0}\in \HOM_\shc (M\tens N,N\tens M)$
and $\tRre_{\Na,\Ma}\vert_{\z=\z[\Na]=0}\in \HOM_\shc(N\tens M,M\tens N)$
do not vanish.
\ee
\enlemma
\begin{proof}
By Proposition \ref{prop:simplehd} we have
$\dim \HOM_\shc(M\tens N, N\tens M) =  1$, 
Hence
by Theorem \ref{th:ren_r_matrix}  there exist $\Rre_{\Ma,\Na}$  and $f(z_\Ma,z_\Na)$ satisfying the properties in (a) and (b). 
The proof of the statement on
$\tRre_{\Na,\Ma}$ is similar.
\end{proof}

\Lemma
\label{lem:YB}
  Assume that $\shc$ is a \KO monoidal category.
Let $(\Laa,\Rmat_\Laa)$, $(\Ma,\Rmat_\Ma)$ and $(\Na,\Rmat_\Na)$
be affinizations of real simple objects in $\shc$.
Then the Yang-Baxter equation holds 
for $(\Rre_{\Laa,\Ma},\Rre_{\Ma,\Na},\Rre_{\Laa,\Na})$, namely,
the following diagram in
$\Proc(\cor[\z[\Laa],\z,\z[\Na]],\shc)$ commutes {\rm:}
$$\xymatrix@R=3ex{&\Laa\tens\Ma\tens\Na\ar[ld]_{\Rre_{\Laa,\Ma}}
\ar[rd]^{\Rre_{\Ma,\Na}}\\
\Ma\tens\Laa\tens \Na\ar[d]_{\Rre_{\Laa,\Na}}&&\Laa\tens\Na\tens\Ma
\ar[d]^{\Rre_{\Laa,\Na}}\\
\Ma\tens\Na\tens\Laa\ar[rd]_{\Rre_{\Ma,\Na}}&&\Na\tens\Laa\tens\Ma
\ar[ld]^{\Rre_{\Laa,\Ma}}\\
&\Na\tens\Ma\tens\Laa\;.
}$$
\enlemma
\Proof
By Proposition \ref{prop:simplehd} we have
$$\dim \HOM_\shc(L\tens M, M\tens L) =  \dim \HOM_\shc(L\tens N, N\tens L) =\dim \HOM_\shc(M\tens N, N\tens M) =  1.$$
Then by Theorem \ref{th:ren_r_matrix}, we have non-zero morphisms 
$\Rre_{\Laa,\Ma}$, $\Rre_{\Ma,\Na}$, and $\Rre_{\Laa,\Na}$. 
Since the Yang-Baxter equation holds for the triple
$(\Rmat_{\Laa}(\Ma),\Rmat_{\Ma}(\Na),\Rmat_{\Laa}(\Na))$ and they are proportional to $\Rre_{\Laa,\Ma}$, $\Rre_{\Ma,\Na}$, $\Rre_{\Laa,\Na}$,  the assertion follows.
\QED

The following two statements are an analogue of results of \cite{KKKO15}
in the affinization case.
\Lemma \label{lem:K3OAff}
Assume that  $\shc$ is an abelian rigid monoidal category 
with \eqref{cond:exactmono}.
Let $(\Laa,z),(\Ma,z),(\Na,z)$ be \afns in $\shc$.
Let $X\subset \Laa \tensz \Ma$ and $Y\subset\Ma \tensz \Na$
be \afns.
If $X\tensz \Na\subset\Laa \tensz Y$, then there exists
an \afn $Z\subset\Ma$ such that
$$X\subset \Laa\tensz  Z\qtq Z\tensz \Na\subset Y.$$
If we further assume that $Y$ is a strict \subafn of $\Ma \tensz  \Na$, then we may assume that $Z$ is a strict \subafn of $\Ma$.
\enlemma
\Proof
Let $Z$ be an object in $\Proc(\cor[z],\shc)$ such that the diagram
\eqn
\xymatrix{
Z\akete \ar[r] \ar@{>->}[d]& Y \tensz  \lDa(\Na)\akete[-1.2ex] \ar@{>->}[d]
\ar@{}[ld]|(.5)\square  \\
\Ma \ar[r] & \Ma \tensz  \Na \tensz  \lDa(\Na)
}
\eneqn
is a cartesian square in $\Proc(\cor[z],\shc)$.
By Proposition~\ref{prop:dualflat} and Proposition~\ref{prop:tenszexact}, $\bullet \tensz   \lDa(\Na)$ is exact and hence  the right vertical arrow is a monomorphism,
which implies that the the left vertical arrow is also a monomorphism.
Hence $Z$ is an \afn.

Since $Z \tensz  \Na \monoto \Ma\tensz \Na$ is decomposed into $Z\tensz  \Na \to Y \to \Ma\tensz \Na$, we get $Z\tensz  \Na \subset Y $.

By applying $\Laa\tensz  \bullet$, which is exact  by Proposition~\ref{prop:tenszexact},  to the above square, we obtain the following commutative diagram in which the bottom square is cartesian:
\eqn
\xymatrix{
X\akete[-3ex]\ar[r] \ar@{ >->}[ddr] \ar@{>-->}[dr] &X\tensz \Na \tensz  \lDa(\Na)\ar[dr]  & \\
&\Laa\tensz  Z \ar[r] \ar[d] & \Laa\tensz  Y \tensz  \lDa(\Na)\akete[-2ex] \ar@{>->}[d]\ar@{}[dl]_(.6){\square} \\
&\Laa\tensz  \Ma \ar[r] & \Laa\tensz   \Ma \tensz  \Na \tensz  \lDa(\Na).
}
\eneqn
Hence we have $X\monoto\Laa\tensz  Z$ in $\Aff(\shc)$, as desired.

Assume further that $Y$ is a strict \subafn of $\Ma \tensz  \Na$.
Take $Z'\seteq\Ker(\Ma \to (\Ma/Z)_\fl)$. Then $Z \subset Z' \subset \Ma$ and $\Ma/Z'\simeq (\Ma/Z)_\fl \in \Aff(\shc)$. Hence $X\subset \Laa  \tensz   Z'$ and $Z'$ is a strict \subafn of $\Ma$.
We have a morphism
$$\bl ( \Ma\tensz \Na)/(Z\tensz\Na)\br_\fl\to \bl(\Ma\tensz\Na)/Y\br_\fl.$$
Since the left hand side is
$(\Ma\tensz\Na  )/(Z'\tensz\Na)$ and the right hand side is $(\Ma\tensz\Na)/Y$,
we obtain $Z'\tensz\Na\subset Y$.
\QED

\Prop\label{prop:simplehead}
Assume that  $\shc$ is an abelian rigid monoidal category
with \eqref{cond:g}.
Let $(\Ma,z)$ be an affinization and let $(\Na,z)$ be an \afn.
Assume that $\Ma/z\Ma$ is real simple and $\Na/z\Na$ is simple.
Let $\bRre_{\Ma,\Na}\cl \Ma\tensz \Na\to \Na \tensz \Ma$ be
the renormalized $R$-matrix in \eqref{eq: induced_r}.
Then $\Im\bl\bRre_{\Ma,\Na}\br$ is a unique rationally simple quotient \afn of 
$\Ma\tensz \Na$.
\enprop

\Proof
Let $K\subset \Ma \tensz \Na$ be a strict \subafn
such that $K\not=\Ma \tensz  \Na$. It is enough to show that
$\bRre_{\Ma,\Na}(K)=0$.
Indeed, then $\Ker\bl\bRre_{\Ma,\Na}\br$ is a unique maximal strict \subafn of $\Ma \tensz \Na$, and it remains to remark that the quotient of an affine object by a maximal strict  \subafn is rationally simple.

Consider the following commutative diagram
\eqn
\xymatrix@C=7ex{
\Ma\tensz  K\akete[-1.5ex]\ar[rr]\ar@{>->}[d]&&K\tensz \Ma\akete[-1.5ex]\ar@{>->}[d]\\
\Ma\tensz \Ma\tensz \Na\ar[r]_{\bRre_{\Ma,\Ma}}
&\Ma\tensz \Ma\tensz \Na\ar[r]_{\bRre_{\Ma,\Na}}
&\Ma\tensz \Na\tensz \Ma.}
\eneqn
Since $\bRre_{\Ma,\Ma}=\id_{\Ma\tensz\Ma}$ up to a constant multiple by 
Proposition~\ref{prop:uniquehom}, we have
$\Ma\tensz  \bRre_{\Ma,\Na}(K)\subset K\tensz \Ma$.
Hence, Lemma~\ref{lem:K3OAff} implies that there exists a strict \subafn $Z$ of $\Na$ such that
$\bRre_{\Ma,\Na}(K)\subset Z\tensz \Ma$
and $\Ma\tensz Z\subset K$.
Since $K\not=\Ma\tensz \Na$, we have
$Z\not=\Na$. Since $\Na$ is rationally simple, we have $Z=0$.
Hence $\bRre_{\Ma,\Na}(K)=0$.
\QED

\Def
Assume that $(\Ma,z)$ and $(\Na,z)$ be affinizations of
real simple objects. If $\Im(\bRre_{\Ma,\Na})$ is an affinization of 
$(\Ma/z\Ma)\hconv(\Na/z\Na)$, we denote $\Im(\bRre_{\Ma,\Na})$ by
$\Ma\hconv_z\Na$.
\edf

\Lemma \label{lem:CVB}
Assume that  $\shc$ is an abelian rigid monoidal category 
with \eqref{cond:g}.
Let $(\Ma,z)$ and $(\Na,z)$ be affinizations of real simple
modules $M$ and $N$, respectively, and let
$(\Laa,z)$ be an affinization of $M\hconv N$.
Assume that
there is an epimorphism
$$\Ma\tensz  \Na\epito\Laa$$
in $\Proc(\cor[z],\shc)$.
Then, we have $\Im\bl\bRre_{\Ma,\Na}\br\simeq\Laa$.
\enlemma
\Proof
It follows from Proposition~\ref{prop:simplehead}.
\QED

\Prop
Assume that $\shc$ is an abelian rigid monoidal category 
with \eqref{cond:g}.
Let $(\Ma,z)$ be an affinization of a real simple object $M$.
Then $\Ma\tensz\Da(\Ma)\to\cor[z]$ is an epimorphism 
in $\Proc(\cor[z],\shc)$. \enprop
\Proof
It follows from Proposition~\ref{prop:noether}\;\eqref{it:Nakayama} and 
the fact that
$M\tens\D M\to\one$ is an epimorphism.
\QED

\Prop\label{prop:defactor}
Assume that $\shc$ is an abelian rigid monoidal category 
with \eqref{cond:g}.
Let $(\Ma,z)$, $(\Na,z)$ and $(\Laa,z)$ be an affinization with $\deg(z)=d$ of 
a real simple $M$, $N$  and $L$ in $\shc$, respectively.
Assume that 
\bna
\item $\de(M,N)>0$,
\item there exists an epimorphism 
$\Ma\tensz \Na\epito\Laa$  in $\Aff(\shc)$.\label{it:reg}
\ee

Then we have
\bnum
\item
$\Im(\bRre_{\Ma,\Na})\simeq\Laa$,
\item
$\Daf(\Ma,\Na)\in\cor[\z,z_{\Na}](\z-z_{\Na})$. 
Here $\z=z\vert_\Ma$ and $\zN=z\vert_{\Na}$.
\ee
\enprop
\Proof
\ (i) follows from Lemma~\ref{lem:CVB}.

\snoi
(ii)\ Set $z=\z$, $w=\zN$ and $f(z,w)=\Daf(\Ma,\Na)$. 
Assume that  $f(z,w)$ is a homogeneous function of degree $r$ in $z,w$
(counting the degrees of $z,w$ as one),
i.e., $\de(M,N)=dr/2$.
Then we have
$\Rre_{\Ma,\Na}\cl\Ma\tens \Na\to\Na\tens \Ma$
and
$\Rre_{\Na,\Ma}\cl\Na\tens \Ma\to\Ma\tens\Na$
such that
$$\Rre_{\Na,\Ma}\circ\Rre_{\Ma,\Na}=f(z,w)\id_{\Ma\tens\Na}.$$
Hence, the composition
$$\Ma\tensz\Na\To[\Rre_{\Ma,\Na}\vert_{z=w}]\Na\tensz\Ma
\To[\Rre_{\Na,\Ma}\vert_{z=w}]\Ma\tensz \Na$$
is $z^rf(1,1)$.
Hence if $f(1,1)\not=0$, then
$\Ma\tensz\Na\to\Na\tensz\Ma$
is a monomorphism, and hence $\Ma\tensz\Na\to\Laa$ is an isomorphism,
which implies that $M\tens N\simeq L$ is a simple object.
Hence $M\tens N\simeq N\tens M$ and
it contradicts $\de(M,N)>0$. 
Thus we obtain $f(1,1)=0$, which implies that $z-w$ is a factor of $f(z,w)$.
\QED

\Lemma\label{lem:LaMDN}
Assume that $\shc$ is an abelian rigid monoidal category 
with \eqref{cond:g}.
Let $M,N$ be simple objects in $\shc$ such that one of them is \afr. Then
$$\La(M,\D N) = \La(N,M).$$
\enlemma
\begin{proof}
This immediately follows from
$$\HOM(N\tens M,M\tens N)\simeq\HOM(M\tens\D N,\D N\tens M).$$
\end{proof}

\section{Quiver Hecke algebras and Schur-Weyl duality functors}
\label{Sec:QHSW}

\subsection{Quiver Hecke algebras}\label{subsec:QHA}

Let $\bR$ be a field and let $\cmA$ be a  symmetrizable generalized Cartan matrix. We fix a Cartan datum $ \bl\cmA,\wlP,\Pi,\Pi^\vee,(\cdot,\cdot) \br $ consisting of $\cartan$ called
a generalized Cartan matrix, $\wlP$ a free abelian group called the weight lattice, $\Pi = \{ \alpha_i \mid i\in I \} \subset \wlP$ called the set of simple roots, $\Pi^{\vee} = \{ h_i \mid i\in I \} \subset \wlP^{\vee}\seteq\Hom( \wlP, \Z )$ 
called the set of simple coroots, and $(\cdot,\cdot)$ a $\Q$-valued 
symmetric bilinear form on $\wlP$ satisfying the following conditions:
\begin{enumerate} [{\rm (a)}]
\item $\cmA = (\langle h_i,\alpha_j\rangle)_{i,j\in I}$,
\item  $(\alpha_i,\alpha_i)\in 2\Z_{>0}$ for $i\in I$,
\item $\langle h_i, \lambda \rangle =\dfrac{2(\alpha_i,\lambda)}{(\alpha_i,\alpha_i)}$ for $i\in I$ and $\lambda \in \Po$,
\item for each $i\in I$, there exists $\Lambda_i \in \wlP$
such that $\langle h_j, \Lambda_i \rangle = \delta_{ij}$ for any $j\in I$.
\end{enumerate}

We denote by $\g\seteq\g(\cmA)$ the corresponding symmetrizable Kac-Moody algebra and 
set $\pwtl\seteq\st{\la\in\wlP\mid 
\text{$\ang{h_i,\la}\ge0$ for any $i\in I$}}$ 
the set of dominant integral weights.

 For $i,j\in I$, we choose polynomials
$\qQ_{i,j}(u,v) \in \bR[u,v]$ of the form
\begin{align}
\qQ_{i,j}(u,v) =\bc
                   \sum\limits
_{p(\alpha_i , \alpha_i) + q(\alpha_j , \alpha_j) = -2(\alpha_i , \alpha_j) } t_{i,j;p,q} u^pv^q &
\text{if $i \ne j$,}\\[3ex]
0 & \text{if $i=j$,}
\ec\label{eq:Q}
\end{align}
such that $t_{i,j; -\lan h_i,\al_j \ran,0} \in  \bR^{\times}$ and
$$\qQ_{i,j}(u,v)= \qQ_{j,i}(v,u) \quad \text{for all} \ i,j\in I.$$

Let $\rlQ\seteq\soplus_{i\in I} \Z\al_i$
and $\rlQ_+\seteq\soplus_{i\in I} \Z_{\ge0} \al_i$ be the root lattice of $\g$
and the positive root lattice, respectively.
For $\beta\in \rlQ_+$, set
$$
I^\beta\seteq  \Bigl\{\nu=(\nu_1, \ldots, \nu_n ) \in I^n \bigm| \sum_{k=1}^n\alpha_{\nu_k} = \beta \Bigr\}.
$$
The symmetric group $\mathfrak{S}_n = \langle s_k \mid k=1, \ldots, n-1 \rangle$ acts  by place permutations on $I^\beta$.

For $\beta=\sum_{i\in I} b_i \al_i\in\rtl$, set $\height{\beta}=\sum_{i\in I} |b_i|$.
\begin{df}
\ For $\beta\in\rlQ_+$ with $\height{\beta}=n$,
the {\em quiver Hecke algebra} $R(\beta)$ associated with $ \bl\cmA,\Pi,\wlP,\Pi^\vee,(\cdot,\cdot) \br $ and $(\qQ_{i,j}(u,v))_{i,j\in I}$
is the $\bR$-algebra generated by
$$
\{e(\nu) \mid \nu \in I^\beta \}, \; \{x_k \mid 1 \le k \le n \},
 \; \{\tau_l \mid 1 \le l \le n-1 \}
$$
satisfying the following defining relations:
\eqn
&& e(\nu) e(\nu') = \delta_{\nu,\nu'} e(\nu),\ \sum_{\nu \in I^{\beta}} e(\nu)=1,\
x_k e(\nu) =  e(\nu) x_k, \  x_k x_l = x_l x_k,\\
&& \tau_l e(\nu) = e(s_l(\nu)) \tau_l,\  \tau_k \tau_l = \tau_l \tau_k \text{ if } |k - l| > 1, \\[1ex]
&&  \tau_k^2 = \sum_{\nu\in I^\beta}\qQ_{\nu_k, \nu_{k+1}}(x_k, x_{k+1})e(\nu), \\[5pt]
&& \tau_k x_l - x_{s_k(l)} \tau_k =
\bl\delta(l=k+1)-\delta(l=k)\br
\sum_{\nu\in I^\beta,\ \nu_k=\nu_{k+1}}e(\nu),\\
&&\tau_{k+1} \tau_{k} \tau_{k+1} - \tau_{k} \tau_{k+1} \tau_{k}
=\sum_{\nu\in I^\beta,\ \nu_k=\nu_{k+2}}
\bQ_{\,\nu_k,\nu_{k+1}}(x_k,x_{k+1},x_{k+2}) e(\nu),
\eneqn
\end{df}
where
\begin{align*}
\bQ_{i,j}(u,v,w)\seteq\dfrac{ \qQ_{i,j}(u,v)- \qQ_{i,j}(w,v)}{u-w}\in \bR[u,v,w].
\end{align*}

Now, let $\dg\cl \rtl\times\rtl\to\Z$ be a bilinear form
such that 
\eq
\dg(\al,\beta)+\dg(\beta,\al)=-2(\al,\beta)\qt{for any $\al$, $\beta\in\rtl$. }
\eneq
Then, the algebra $R(\beta)$ is $\Z$-graded by
\eq
&&\deg(e(\nu))=0, \quad \deg(x_k e(\nu))= ( \alpha_{\nu_k} ,\alpha_{\nu_k}), \quad  \deg(\tau_l e(\nu))= \dg(\alpha_{\nu_{l}} , \alpha_{\nu_{l+1}}).
\label{eq:grading}
\eneq
We denote by $R_\dg(\beta)$ the graded algebra $R(\beta)$ with the grading 
\eqref{eq:grading}. 
We also regard $R(\beta)$ as a graded algebra by taking
$\dg=-(\;\cdot\;,\;\cdot\;)$\,.

We denote by $\Modg(R_\dg(\beta))$ the graded  abelian category of
graded $R_\dg(\beta)$-modules
and we set $\Modg(R_\dg)=\soplus_{\beta\in\prtl}\Modg(R_\dg(\beta))$.

For $M \in \Modg(R_\dg(\beta))$, we set $\wt(M)\seteq-\beta$. 

We denote by $\Modgc(R_\dg(\beta))$ the full subcategory of $\Modg(R_\dg(\beta))$
consisting of finitely generated graded $R_\dg(\beta)$-modules and
we set $\Modgc(R_\dg)=\soplus_{\beta\in\prtl}\Modgc(R_\dg(\beta))$.

We denote by $R_\dg(\beta)\gmod$ the  the full subcategory of $\Modg(R_\dg(\beta))$ consisting of $\Z$-graded $R_\dg(\beta)$-modules
with finite $\cor$-dimension. Set $R_\dg\gmod \seteq \soplus_{\beta\in\prtl} R_\dg(\beta)\gmod$.
 
Note that $R_\dg(\beta)\gmod$ is independent from $\dg$ as a graded category 
after adding $q^{1/2}$ (see \cite[Lemma 1.5]{KP18}). 
We write $\Modg(R)$, $R\gmod$, etc.\ for $\Modg(R_\dg)$, $R_\dg\gmod$, 
etc.\ with $\dg=-(\;\cdot\;,\;\cdot\;)$\,.

For $M\in R\gmod$, set $M^\star \seteq  \HOM_{\bR}(M, \bR)$. Then $M^\star$ becomes an  $R(\beta)$-module via  the grade-preserving antiautomorphism of $R(\beta)$ which fixes the generators
 $e(\nu)$, $x_k$, and $\tau_k$'s. 
We say that $M$ is \emph{self-dual} if $M \simeq M^\star$ in $R\gmod$. 
 For each simple module $M$ in $R\gmod$, there exists $m\in \Z$ such that $q^mM$
is self-dual.

For each $i\in I$ and $n\ge 1$, $R(n\al_i)$ has a unique self-dual simple module which is denoted by $L(i^n)$. Sometimes we will denote it by $\ang{i^n}$. 
We have $\dim(\ang{i^n})=n\ms{1mu}!$.

\subsection{Universal R-matrices} \label{subSec: UR}
The category $\Modg(R_\dg)$ is endowed with a monoidal category structure with
the convolution product $\conv$ (e.g.\ see 
\cite{KL09, KL11, Rouquier08} for details).
Moreover it is \KO by \cite[Lemma 3.1]{KKKO15}.
The category $R_\dg\gmod$ is a graded monoidal category satisfying 
\eqref{cond:exactmono}.
We have a fully faithful monoidal functor
$$\Modgc(R_\dg)\to \Pro(R_\dg\gmod)$$
by
$$M\longmapsto \proolim[m]M/\bl R_\dg(R_\dg)_{\ge m} M\br.$$
Hereafter, we identify $\Modgc(R_\la)$ as a full subcategory of
$\Pro(R_\dg\gmod)$.

Note that an \afn in $R_\dg\gmod$ is equivalent to a pair
$(\Ma,z)$ of a graded $R_\dg$-module $\Ma$ and an
injective  endomorphism $z$ of
$\Ma$ of a positive degree such that
$\Ma/z\Ma\in R_\dg\gmod$ and $\Ma_k=0$ for $k\ll0$.

Let $\beta \in \rlQ_+$ with $m =  \Ht(\beta)$. For  $k=1, \ldots, m-1$ and $\nu \in I^\beta$, the \emph{intertwiner} $\varphi_k \in R(\beta) $ is defined by 
\eq
\varphi_k e(\nu) =
\bc
 \bl\tau_k(x_k-x_{k+1})+1\br e(\nu) 
& \text{ if } \nu_k = \nu_{k+1}, 
 \\
 \tau_k e(\nu) & \text{ otherwise.}
\ec\label{def:intertwiner}
\eneq
\begin{lem} [{\cite[Lemma 1.5]{K^3}}] \label{Lem: intertwiners} 
The intertwiners have the following properties.
\begin{enumerate}[\rm (i)]
\item $\varphi_k^2 e(\nu) = \bl Q_{\nu_k, \nu_{k+1}} (x_k, x_{k+1} )+ \delta_{\nu_k, \nu_{k+1}} \br\, e(\nu)$.
\item $\{  \varphi_k \}_{k=1, \ldots, m-1}$ satisfies the braid relation.
\item For a reduced expression $w = s_{i_1} \cdots s_{i_t} \in \sg_m$, we set $\varphi_w \seteq  \varphi_{i_1} \cdots \varphi_{i_t} $. Then
$\varphi_w$ does not depend on the choice of reduced expression of $w$.
\item For $w \in \sg_m$ and $1 \le k \le m$, we have $\varphi_w x_k = x_{w(k)} \varphi_w$.
\item For $w \in \sg_m$ and $1 \le k < m$, if $w(k+1)=w(k)+1$, then $\varphi_w \tau_k = \tau_{w(k)} \varphi_w$.
\end{enumerate}
\end{lem}

For $m,n \in \Z_{\ge 0}$, we set $w[m,n]$ to be the element of $\sg_{m+n}$ such that
$$
w[m,n](k) \seteq
\left\{
\begin{array}{ll}
 k+n & \text{ if } 1 \le k \le m,  \\
 k-m & \text{ if } m < k \le m+n.
\end{array}
\right.
$$

For an $R(\beta)$-module $M$ and an $R(\gamma)$-module $N$, the $R(\beta)\otimes R(\gamma)$-linear map $M \otimes N \rightarrow N \conv M$ defined by $$u \otimes v \mapsto \varphi_{w[\height{\gamma},\height{\beta}]}(v \tens u)$$
extends to an $R(\beta+\gamma)$-module homomorphism (neglecting a grading shift)
$$
\Runi_{M,N}\cl  M\conv N \longrightarrow N \conv M.
$$
We call $\Runi_{M,N}$ the {\em universal $R$-matrix} between $M$ and $N$.
Since the intertwiners satisfies the braid relations,  the universal R-matrices $\Runi_{M,N}$ satisfy the Yang-Baxter equation (\cite[(1.9)]{K^3}).

Let $(\Ma,z)$ be an \afn in $R_\dg\gmod$.
We set $\Runi_\Ma(X)=\Runi_{\Ma,X}\cl \Ma\conv X\to X\conv\Ma$
(neglecting the grading shift).
If $(\Ma,\Runi_\Ma)$ is a rational center in $R_\dg\gmod$, then we say that
$(\Ma,\Runi_\Ma)$ is a {\em canonical affinization} of $\Ma/z\Ma\in R_\dg\gmod$.

\subsection{Canonical affinizations in $R_\dg\gmod$}
We recall the notion of affinizations for $R$-modules introduced  in \cite{KP18}. 
For $\beta \in \rlQ_+$ and $i\in I$, let
\begin{align} \label{Eq: def of p}
\mathfrak{p}_{i, \beta}  \seteq \sum_{\nu \in I^\beta} \Bigl(\hs{1ex}  \prod_{a \in \{1, \ldots, \Ht(\beta) \},\ \nu_a=i} x_a \Bigr) e(\nu)\in R(\beta).
\end{align}
Then $\mathfrak{p}_{i, \beta} $ belongs to the center of $R(\beta)$.

  \label{Def: aff}
Let $M$ be a simple module in $R_\dg(\beta)\gmod$.  
 Assume that there exists   a  graded  $R_\dg(\beta)$-module $\Ma$ with an endomorphism $z_{\Ma}$ of $\Ma$
with degree $d_{\Ma} \in \Z_{>0}$ such that
\eq \label{eq:oldaff}
&&\hs{2ex}\parbox{75ex}{
\begin{enumerate}[\rm (i)]
\item $\Ma / z_{\Ma} \Ma \simeq M$,
\item $\Ma$ is a finitely generated free module over the polynomial ring $\bR[z_{\Ma}]$,
\item $\mathfrak{p}_{i,\beta} \Ma \ne 0$ for all $i\in I$.\label{it:nonzeroP}
\end{enumerate}
}\eneq
Note that the conditions (i) and (ii) are equivalent to the statement $(\Ma,\z)$ is an \afn.

For an indeterminate $z$ and a non-zero homogeneous $w\in\cor[z]$, we denote by
$\Ma\vert_{\z=w} \seteq \cor[z] \tens_{\cor[\z]} \Ma$ the object in $\Proc(\cor[z],R_\dg\gmod)$. Here, $\cor[z]$ is a $\cor[\z]$-algebra by
the algebra homomorphism $\z  \mapsto w \in \cor[z]$.
If $w=cz^a$ for some $c\in \cor^\times , a\in \Z_{>0}$ such that $\deg(\z)=a\deg(z)$, then $\Ma\vert_{\z=w} $ is an affine object (\cite[Example 2.5 (ii)]{KP18}). Sometimes  we will write $\Ma\vert_{\z=w} =\Ma_w$  for short.

\Prop[{\cite[Lemma 2.9]{KP18}}] \label{pro:canoaff} 
Let $M$ be a simple module in $R_\dg(\beta)\gmod$. Assume that 
an $R_\dg(\beta)$-module  $(\Ma,z)$  satisfies the conditions \eqref{eq:oldaff}. 
Then $(\Ma,\Runi_\Ma)$ is a rational center 
in $R_\dg\gmod$.
\enprop
\Proof

It is enough to show that the R-matrix
$\Runi_{\Ma,X}\cl\Ma\conv X\to X\conv \Ma$ is an isomorphism in $\Rat(R_\dg\gmod)$
for any $X\in R_\dg\gmod$.

The composition
$\Runi_{\Ma,X}\circ \Runi_{X,\Ma}\cl X\conv \Ma\to \Ma\conv X\to X\conv\Ma$ is given by
$$f(z)\seteq\prod_{i\not=j}\chi_i(\Ma)(t_i)\res[t_i]Q_{i,j}(t_i,t_j)\res[t_j]
\chi_j(X)$$ 
(see \eqref{eq:RNMRMN} below, and for the notation $\chi_i$ and $ \res[t]$,  see \S\,\ref{sec:resultant}, \S\,\ref{subsec:Inv}).
Since $\END_{\cor[z]}(\Ma)\simeq\cor[z]$ by Lemma~\ref{lem:endpreaff},
the homogeneous polynomial $\chi_i(\Ma)(t_i,z)\in\cor[t_i,z]$ 
is monic in $t_i$ and quasi-monic in $z$ 
by the assumption \eqref{it:nonzeroP}\;\eqref{eq:oldaff}. 

Hence $f(z)$ has the form (up to a constant multiple)
$f(z)=z^m-a$ for some $m\in\Z_{\ge0}$ and $a\in\END(X)_{>0}[z]$. 
Since $a^s=0$ for some $s\in\Z_{>0}$,
if we set $g(z)=(z^m)^{s-1}+az(z^m)^{s-2}+\cdots+a^{s-1}$,
then we have
$$f(z)g(z)=g(z)f(z)=z^{ms}-a^s=z^{ms}.$$
Hence 
$\Runi_{\Ma,X}\circ \bl \Runi_{X,\Ma}g(z)\br=z^{ms}\id_{X\tens\Ma}$
is invertible in $\Rat(R_\dg\gmod)$, and hence $\Runi_{\Ma,X}$ has a right inverse.

Similarly,  $\Runi_{\Ma,X}$ has a left inverse.
\QED

\Rem 
Let $(\Ma,z)$ be an \afn of a simple $M\in \Rd$. 
Even if $(\Ma,\Runi_\Ma)$ is a rational center in $\Rd$,
$(\Ma,z)$ may not satisfy \eqref{eq:oldaff}.
For example, $(\Ma\conv C,\Runi_{\Ma}\circ R_C)$ is a rational center
if $(\Ma,\Runi_\Ma)$ is a rational center and 
$(C,R_C)$ is a central object in $\Rd$. \enrem

{\em In the sequel, we use the terminology ``affinization''
in the sense of Definition~\ref{def:affinization}. }

\subsection{Schur-Weyl duality} \label{subSection: SW}

Let $\shc$ be a graded $\cor$-linear monoidal category   
satisfying \eqref{cond:g}. {\em We assume further that $\shc$ is \KO.}
Let $R$ be the quiver Hecke algebra 
associated with  a Cartan matrix $\cartan$ and a set of polynomials  $\st{Q_{i,j}(u,v)}$ as in \eqref{eq:Q}.

Let $\st{(\hSW_i,z_i)}_{i\in I}$ be a family of affinizations in $\shc$
such that
\eq
&&\left\{\parbox{70ex}{
\bna
\item $\deg(z_i)=2(\al_i,\al_i)$,
\item $\SW_i\seteq\hSW_i/z_i\hSW_i$ is real simple in $\shc$ for any $i\in I$,
\item $\Daf(\hSW_i,\hSW_j)=Q_{i,j}(z_i,z_j)$ for $i\not=j$.
\ee}\right.\label{eq:SW}
\eneq
We define the bilinear form  $\dg\cl\rtl\times\rtl\to\Z$ by
$$\dg(\al_i,\al_j)=
\bc\La(\SW_i,\SW_j)&\text{for $i,j\in I$ such that $i\not=j$.}\\
-(\al_i,\al_j)&\text{if $i=j\in I$.}\ec
$$
Then we have
$\dg(\al,\beta)+\dg(\beta,\al)=-2(\al,\beta)$ for any $\al,\beta\in\rtl$.

Let $\Rre_{\hSW_i,\hSW_j}\cl\hSW_i\tens \hSW_j\to \hSW_j\tens \hSW_i$
be a renormalized R-matrix. It is determined up to a constant multiple.
We normalize them such that
\eq
&&\left\{
\parbox{70ex}{
\bna
\item
$\Rre_{\hSW_j,\hSW_i}\circ \Rre_{\hSW_i,\hSW_j}
=Q_{i,j}(z_i,z_j)\id_{\hSW_i\tens \hSW_j}$ for $i\not=j$,
\item for any $i\in I$,
there exists $T_{i}\in\END_{\Pro(\shc)}(\hSW_i\tens \hSW_i)$
such that
$$\Rre_{\hSW_i,\hSW_i}-\id_{\hSW_i\tens \hSW_i}=(z_i\tens 1-1\tens z_i)\circ
T_{i}.$$
\ee}\right.
\label{eq:normR}
\eneq
Such normalizations are possible by Proposition~\ref{prop:nuk=nuk+1}.

We call $\bl\st{(\hSW_i,z_i)}_{i\in I},\st{\Rre_{\hSW_i,\hSW_j}}_{i,j\in I}\br$ a {\em duality datum}.

\Rem
Once $\st{(\hSW_i,z_i)}_{i\in I}$ is given, 
a duality datum $\bl\st{(\hSW_i,z_i)}_{i\in I},\st{\Rre_{\hSW_i,\hSW_j}}_{i,j\in I}\br
$ exists
and it is unique up to constant multiples. 

More precisely,
if $\bl\st{(\hSW_i,z_i)}_{i\in I},\st{\widetilde{\Rre}_{\hSW_i,\hSW_j}}_{i,j\in I}\br$ is another duality datum,
then there exist $c_{i,j}\in\cor^\times$ ($i,j\in I$) such that
$\widetilde{\Rre}_{\hSW_i,\hSW_j}=c_{i,j}\, \Rre_{\hSW_i,\hSW_j}$ and
$c_{i,j}c_{j,i}=1$, $c_{i,i}=1$.
\enrem

\Prop[{cf.\ \cite[\S\,3.1, \S\, 3.2]{K^3}}] 
Let $\bl\st{(\hSW_i,z_i)}_{i\in I},\st{\Rre_{\hSW_i,\hSW_j}}_{i,j\in I}\br$ be a duality datum.
\bnum
\item
There exists a  right exact monoidal functor
$$\hF\cl\Modgc(R_\dg)\to\Pro(\shc)$$
such that
\eq&&\hF(\tL(i)_{z_i})\simeq\hSW_i\qt{and}\quad  \hF(L(i))\simeq \SW_i,  \label{eq:SW1}\\
&&\hF(e(i,j)\vphi_1)=\Rre_{\hSW_i,\hSW_j}\in \HOM_{\cor[z_i,z_j]}(\hSW_i\tens \hSW_j,\hSW_j\tens \hSW_i). \label{eq:SW2}
\eneq
Here $\tL(i)_{z_i}$ is the affinization $\bl R(\al_i),z_i\br=(\cor[x_1],x_1)$
of $L(i)=R(\al_i)/R(\al_i)x_1$, 
$\vphi_1$ is given in \eqref{def:intertwiner}, and
$e(i,j)\vphi_1\cl R_\dg(\al_i+\al_j)e(i,j)
\to R_\dg(\al_i+\al_j)e(j,i)$ is the morphism by the right multiplication.

Moreover, such a right exact monoidal functor $\hF$ is unique up to an isomorphism.
\item
The functor $\hF$  can be restricted to a monoidal functor
$$\F\cl R_\dg\gmod \to\ \shc.$$
\item
Assume further that $\cmA$ is of finite type. 
Then the functor $\hF$ is exact and we have the following properties.
\bna
\item $\F$ sends a simple module in $R_\dg\gmod$  to a  simple objects in $\shc$ or $0$.
\item If $(\Ma,\z)$ is an \afn of $M$ in $R_\dg\gmod$, then $(\hF(\Ma),\hF(\z))$ is an \afn of $\F(M)\in\shc$.

\end{enumerate}
\ee
\enprop
\begin{proof}
First note that $\st{\Rre_{\hSW_i,\hSW_j}}_{i,j\in I}$ satisfies the Yang-Baxter equation
by Lemma \ref{lem:YB}.

For each $\beta\in \prtl$ and $\nu=(\nu_1,\ldots,\nu_m) \in I^\beta$, define objects in  $\Pro(\shc)$ by
$$\hSW(\nu)\seteq \hSW_{\nu_1}\tens \cdots  \tens \hSW_{\nu_m} \qt{and} \quad
\hSW(\beta)\seteq\soplus_{\nu \in I^\beta}   \hSW(\nu).$$

Set 
$T_{i,j}=\Rre_{\hSW_i,\hSW_j}$ if $i\not=j$ and
$T_{i,j}=T_i$ if $i=j$.
Here, $T_i$ is given in \eqref{eq:normR}.

Then we can endow $\hSW(\beta)\in\Pro(\shc)$ with a graded $R_\dg(\beta)^{\opp}$-module structure in $\Pro(\shc)$
as follows (cf.\ \cite[\S\;3]{K^3}):
\begin{enumerate}
\item $e(\nu)$ is the projection $\hSW(\beta) \to \hSW(\nu)\subset \hSW(\beta)$,
\item $x_k$ acts on $\hSW(\nu)$ by $$\hSW_{\nu_1}\tens \cdots \tens\hSW_{\nu_{k-1}} \tens z_{\nu_k}  \tens \hSW_{\nu_{k+1}} \tens \cdots \tens \hSW_{\nu_m},$$
\item  $\tau_k\cl \hSW(\nu)\to\hSW\bl s_k\nu\br$ is given by 
\eqn 
\hSW_{\nu_1}\tens \cdots\tens \hSW_{\nu_{k-1}} \tens T_{\nu_k,\,\nu_{k+1}}  \tens \hSW_{\nu_{k+2}} \tens \cdots \tens \hSW_{\nu_m}.  
\eneqn
\end{enumerate}

Since $R(\beta)$ is left noetherian (\cite[Corollary 2.11]{KL09}), there exists a right exact functor 
$$\hF_\beta\cl \Modgc(R_\dg(\beta)) \to \Pro(\shc)$$
given by
$$M \mapsto \hSW(\beta) \tens_{R_\dg(\beta)} M.$$
Note that
$\hF_{\al_i}(\tL(i)_{z_i}) \simeq  \hF_{\al_i}(R_\dg(\al_i)) \simeq \hSW_i \qt{for each} \ i\in I$.

Since $\hF_\beta$ is right exact and $\hF_{\al_i}(z_i) = z_i$, we have $\hF_{\al_i}(L(i))\simeq \SW_i$.

Set
$$\hF\seteq \soplus_{\beta\in \prtl} \hF_\beta \cl \Modgc(R_\dg) \to \Pro(\shc).$$
Then, $\hF$ is a monoidal functor.  
Indeed, for any $M\in \Modc(R_\dg(\beta))$ and $N\in \Modc(R_\dg(\gamma))$, we have
\eqn
\hF(M\conv N) &&\seteq\hSW(\beta+\gamma)  \tens_{R_\dg(\beta+\gamma)} \left(R_\dg(\beta+\gamma) e(\beta+\gamma) \tens_{R_\dg(\beta)\tens R_\dg(\gamma)} (M\tens N)\right) \\
&&\simeq 
\hSW(\beta+\gamma)  e(\beta+\gamma)  \tens_{R_\dg(\beta)\tens R_\dg(\gamma)} (M\tens N) \\
&&\simeq \left(\hSW(\beta)\tens \hSW(\gamma) \right)  \tens_{R_\dg(\beta)\tens R_\dg(\gamma)}( M\tens N) \\
&&\simeq \left(\hSW(\beta)\tens_{R_\dg(\beta)}M \right) \tens \left(\hSW(\gamma) \tens_{ R_\dg(\gamma)} N \right) \\
&&\simeq \hF(M)\tens \hF(N).
\eneqn

For the uniqueness of $\hF$, see Lemma \ref{lem:uniqueness} below.

Set $\Pbb(\beta)=\soplus_{\nu\in I^\beta}\cor[x_1,\ldots,x_{\height{\beta}}]e(\nu)$.
Then $\Pbb(\beta)$ is a commutative subalgebra of $R(\beta)$.
Since  $\hSW(\beta)\in\Proc(\Pbb(\beta),\shc)$, Lemma~\ref{lem:fin}
implies that
$$\hSW(\beta)\tens_{\Pbb(\beta)} M \in \shc\qt{for any $M\in R_\dg\gmod$.}$$
Since $\hSW(\beta)\tens_{R_\dg(\beta)} M$ is a quotient of $\hSW(\beta)\tens_{\Pbb(\beta)} M$, we have
$\hF(M) \in \shc$ by Lemma \ref{lem:art}. 
Because $\shc \to \Pro(\shc)$ is fully faithful, we get the restriction
$$\F\seteq\hF\vert_{R_\dg\gmod} \cl R_\dg\gmod \to \shc.$$

\bigskip
Let us show (iii).
Assume that $\cartan$ is of finite type.  
Then $R(\beta)$ has a finite global dimension by \cite{Kato14, McNamara15}. 
Hence, \cite[Proposition 3.7]{K^3} implies that
$R(\beta)^\opp$-modules is flat over  $R(\beta)^\opp$ as soon as
it is flat over $\Pbb(\beta)$.

Since $\hSW_i$'s are \afns, $\hSW(\beta)$ is 
flat over $\Pbb(\beta)$
by  Proposition  \ref{prop:AAflat} (iv). 
Hence  $\hSW(\beta)$ is $R^{\opp}(\beta)$-flat,
which implies that $\hF_\beta$ is exact.

Assume that $M$ is simple in $R_\dg(\beta)\gmod$.
Let us show that $\F(M)$ is simple or zero
by induction on $\height{\beta}$.
We may assume that $\beta\not=0$.
Then there exist $i\in I$ and 
a simple $R_\dg(\beta-\al_i)$-module $N$
such that  $M\simeq N\hconv L(i)$.
Then $M\simeq \Im(\rmat{})$, where $\rmat{}$ is 
an R-matrix $\rmat{}\cl N\conv L(i)\to L(i)\conv N$.
Since $\F$ is exact, we have
$\F(M)\simeq \F(\Im(\rmat{})) \simeq \Im \F(\rmat{})$.
We may assume that $\F(\rmat{})$ is non-zero.
Hence $\F(N) \neq 0$. By the induction hypothesis,
$\F(N)$ is simple.
Then Proposition \ref{prop:simplehd} implies that
$\Im \F(\rmat{})$ is isomorphic to the simple head $\F(N)\hconv \F(L(i))$, 
since $\shc$ is assumed to be \KO and $ \F(L(i)) \simeq \SW_i$ is \afr. Hence we get  (a). 

\snoi
(b) is a direct consequence of the exactness of $\hF$.
\end{proof}

\begin{lem} \label{lem:uniqueness}
Let $\bl\st{(\hSW_i,z_i)}_{i\in I},\st{\Rre_{\hSW_i,\hSW_j}}_{i,j\in I}\br$ be a duality datum and let
$$\hF\cl \Modgc(R_\la) \to \Pro(\shc)$$
be a right exact monoidal functor satisfying \eqref{eq:SW1} and \eqref{eq:SW2}. 
Then it is  unique up to an isomorphism.
\end{lem}
\begin{proof}
Since $R(\beta)e(\nu) \simeq L(\nu_1)_{z_{\nu_1}} \conv \cdots \conv L(\nu_r)_{z_{\nu_r}} $  in $\Modgc(R_\la)$,  by \eqref{eq:SW1} we have
$$\hF(R(\beta)e(\nu)) \simeq \hSW(\nu) \qt{and} \quad \hF(R(\beta))\simeq \hSW(\beta)$$
for any $\beta\in \rtl_+$ and $\nu \in I^\beta$.
Note that 
$\hF(R(\beta))\simeq \hSW(\beta)$ has an $R(\beta)^\opp$-module structure,  which is uniquely determined by  \eqref{eq:SW1} and \eqref{eq:SW2} since the action of $x_k-x_{k+1}$ on $\hSW_\nu$ with $\nu_k=\nu_{k+1}$ is a monomorphism.

Hence for each $M\in \Modgc(R_\la(\beta))$ there exists a functorial morphism
\eq \hSW(\beta)\tens_{R(\beta)} M  \To \hF(M) \label{eq:functorial}
\eneq
induced from
\eqn
&&\Hom_{\Pro(\shc)}\bl  \hSW_\beta\tens_{R(\beta)} M, \hF(M) \br  
\simeq \Hom_{R(\beta)} \bl M, \Hom_{\Pro(\shc)}(\hSW_\beta,\hF(M)) \br \\
&&\simeq \Hom_{R(\beta)} \bl \Hom_{R(\beta)}(R(\beta),M), \Hom_{\Pro(\shc)}(\hF(R(\beta)),\hF(M)) \br. 
\eneqn
Note that \eqref{eq:functorial} is an isomorphism  if $M$ is a free $R(\beta)$-module.
Since $\hF$ is right exact, one can take a free presentation  of $M$ to  conclude that \eqref{eq:functorial} is an isomorphism for a general $M$.
\end{proof}

\section{Localizations of $R\gmod$} \label{Sec: localizations}

\subsection{Categories  $\catC_w$   
and  $\catC^*_w$}
Let $R$ be a quiver Hecke algebra as in section \ref{Sec:QHSW}.
For $\al,\beta\in\prtl$, we set
$$e(\al,\beta)=\sum_{\substack{\nu\in I^{\al+\beta}\\
\sum_{k=1}^{\height{\al}} \al_{ \nu_k} =\al,\ \sum_{k=1}^{\height{\beta}} \al_{ \nu_{k+\height{\al}}}
 =\beta}}\hs{-3ex}e(\nu)\hs{1ex}\in R(\al+\beta).$$
For $M\in R(\beta)\gmod$ we define
\begin{align*}
\gW(M) \seteq  \{  \gamma \in  \rlQ_+ \cap (\beta - \rlQ_+)  \mid  e(\gamma, \beta-\gamma) M \ne 0  \}, \\
\gW^*(M) \seteq  \{  \gamma \in  \rlQ_+ \cap (\beta - \rlQ_+)  \mid  e(\beta-\gamma, \gamma) M \ne 0  \}.
\end{align*}

Let $\weyl$ be the Weyl group, that is, the subgroup of
$\Aut(\wtl)$ generated by $\st{s_i}_{i\in I}$ where
$s_i(\la)=\la-\ang{h_i,\la}\al_i$.
 Let us denote by $>$ the Bruhat order on  $\weyl$. 

For $w\in\weyl$, let us define the full monoidal subcategories of $R\gmod$
by
\eq
&&\ba{l}
\catC_w\seteq\st{M\in R\gmod\Mid \W\,(M)\subset\prtl\cap\, w\nrtl},\\
\catC^*_w\seteq\st{M\in R\gmod\Mid \W^*(M)\subset\prtl\cap\, w\nrtl}.
\ea\label{def:Cw}
\eneq

Recall that there is an involutive $\cor$-algebra automorphism $\psi$  of $R(\beta)$   
given by
\eq e(\nu_1,\ldots, \nu_n)\mapsto e(\nu_n,\ldots,\nu_1),\quad x_k\mapsto x_{n+1-k} 
,\quad \tau_k\mapsto -\tau_{n-k},
\label{def:psi}\eneq
where $n=\height{\beta}$.
Then $\psi$ induces an equivalence of  monoidal categories:
$$\psi_*\cl R\gmod \simeq (R\gmod)^\rev.$$
Here, for a monoidal category $\sht$, $\sht^\rev$ denotes the monoidal category 
endowed with the reversed tensor product $\tens_\rev$ defined by $ M\tens_\rev N\seteq N\tens M$.

Since we have 
$\psi(e(\gamma, \beta-\gamma)) = e (\beta-\gamma, \gamma)$ for $\beta, \gamma \in \rtl_+$ with $\beta-\gamma \in \rtl_+ $, the automorphism $\psi$ induces
 an equivalence of monoidal categories 
$$\catC^*_w \simeq (\catC_w)^\rev.$$

For $\lambda,\mu\in\wtl$, we define a partial order
$\lambda \wle \mu$ 
if there exists a sequence of
real positive roots $\beta_k$ ($1\le k\le \ell$)
such that $\la=s_{\beta_\ell}\cdots s_{\beta_1}\mu$ and
$(\beta_k,s_{\beta_{k-1}}\cdots s_{\beta_1}\mu)\ge0$ for $1\le k\le\ell$.
Here $s_{\beta}(\la)=\la-(\beta^\vee,\la)\beta$ with
$\beta^\vee=\frac{2}{(\beta,\beta)}\beta$.

Assume that $\lambda, \mu \in \weyl \Lambda$ for some $\La \in \pwtl$ and that $\lambda \wle \mu$.
Then there exists a module $\dM(\lambda, \mu)$ in $R(\mu-\la)\gmod$, called  the \emph{determinantial module}. We refer \cite[Section 3.3]{KKOP21} for the definition and properties of determinantial modules.
In particular, $\dM(\la,\mu)$ admits an affinization (\cite[Theorem 3.26]{KKOP21}).

\subsection{Localizations} \label{subSec: Loc}
In this subsection we recall the localization $\tcatC_w$ of the monoidal category $\catC_w$ introduced in \cite{KKOP21}. 
Throughout this subsection, $w\in \weyl$ is an element of the Weyl group and
we assume that 
 $I_w=I$, where $I_w\seteq\st{i_1,\ldots,i_r}$ for a reduced expression $w=s_{i_1}\cdots s_{i_r}$.

For each $i\in I$, set $\Ctr_i\seteq\dM(w{\La_i},{\La_i})$. 
Then for each $X\in R(\beta)\gmod$, there exists a morphism $R_{\dC_i}(X) \cl \dC_i \conv X \to q^{\dphi_i(\beta)} X \conv \dC_i$, where
$$
\dphi_i(\beta)  = - (w \La_i + \La_i , \beta) \qquad \text{for any $\beta\in \rtl$.}
$$
Moreover, the family $(\dC_i, \coR_{\dC_i}, \dphi_i)_{ i\in I}$
forms a \emph{real commuting family of  non-degenerate  graded \ro left\rf braiders} in $R\gmod$ (\cite[Proposition 5.1]{KKOP21}).
Note that $(\dC_i, \coR_{\dC_i}\vert_{\catC_w})$
is a central object of $\catC_w$ (\cite[Theorem 5.2]{KKOP21}),  which means that $\coR_{\dC_i}(X)$ is an isomorphism for all $X\in\catC_w$.

Hence by \cite[Theorem 2.12]{KKOP21}, there exists a localization of $\catC_w$
by the family $(\dC_i, \coR_{\dC_i}\vert_{\catC_w})$ and a canonical functor
$$\Phi_w \cl \catC_w \to \tcatC_w \seteq \catC_w [ \dC_i^{\circ -1 } \mid i \in I ].$$

We have the following properties (\cite{KKOP21, KKOP22}):
\bnum
\item  the category $\tcatC_w$ is a $\cor$-linear abelian rigid monoidal category where the tensor product is a natural extension of $\conv$ on $\catC_w$,
\item  the grading shift functor $q$ and the contravariant functor $M \mapsto M^\star$ on $\catC_w$   are extended to $\lRg$,
\item 
for any simple module $M\in\lRg$, there exists a unique $n\in\Z$ such that
$q^nM$ is self-dual,
\item the objects $\Phi_w(\dC_i)$ are invertible in $\lRg$, that is, the functors $\Phi_w(\dC_i) \conv \scbul $ and $\scbul\conv \Phi_w(\dC_i) $ are equivalences,
\item for each $\La \in \wtl$, there is an invertible object $\dC_\La \in \tcatC_w$ such that $\dC_\La \conv \dC_{\La'} \simeq q^{\gH(\La,\La')}  \dC_{\La+\La'}$ for $\La,\La'\in \wtl$, and 
$\dC_{\La} =\Phi_w(\dM(w\La,\La))$ for $\La\in \pwtl$, where $\gH(\scbul,\scbul)$ is a $\Z$-bilinear map on $\lG$ determined by $\gH( \La_i, \La_j ) = (\La_i, w \La_j - \La_j)$,
\item for any simple object $S$ of $\Cw$,
the object $\Phi_w(S)$ is simple in $\tcatC_w$, 
 \item
   every simple object $M$ of $\lRg$ is isomorphic to $\dC_{ \La } \circ \Phi_w(S) $ for some simple object $S$ of $\catC_w$ and $ \La \in  \wlP $
   and hence $\HOM_{\tcatC_w}(M,M) =\cor \id_{M}$,
\item every object in $\tcatC_w$ has finite length,
\item
 for two simple objects $S$ and $S'$ in $\catC_w$ and $\La,\La'\in \wlP$,
$\dC_{\La} \conv  \Phi_w(S) \simeq \dC_{\La'} \conv  \Phi_w(S') $ in $\lRg$
if and only if $q^{\gH(\La,\mu)} \dC_{\La+\mu} \conv S  \simeq  q^{\gH(\La',\mu)} \dC_{\La'+\mu} \conv  S'  $ in
$\catC_w$
for some $\mu\in \wlP$
such that $\La+\mu, \La'+\mu \in \wlP_+$.
\end{enumerate}

If one localizes the category $R\gmod$ via the same family $(\dC_i, \coR_{\dC_i}, \dphi_i)_{ i\in I}$, then the resulting category turns out to be equivalent to the category $\tcatC_w$ (\cite[Theorem 5.9]{KKOP21}). 
Denoting by $\Qt_w$ the composition of the equivalence and the canonical functor from $R\gmod$ to the localization, we have
$$\Qt_w \cl R\gmod\to \tcatC_w,$$
and 
the following diagram is quasi-commutative:
$$\xymatrix{\catC_w\ar[r]\ar[dr]_-{\Phi_w}&R\gmod\ar[d]^-{\Qt_w}\\
&\tcatC_w\,. } $$

Similarly, the family $(\psi_*(\dC_{i}) , \psi_* \circ \coR_{\dC_i}\circ  \psi_*, \dphi_i)_{ i\in I}$ forms a real commuting family of \emph{right} braiders in $R\gmod$ so that there exists a localization $\tcatC^*_w$ of $\catC^*_w$ and $R\gmod$ via the family. 
Then the localization functors also satisfy
$$\xymatrix{\catC^*_w\ar[r]\ar[dr]_-{\Phi^*_w}&R\gmod\ar[d]^-{\Qt^*_w}\\
&\tcatC^*_w\,.}$$
Note that the functor
 $\Qt^*_w$ is the composition 
 $$\Qt^*_w \cl R\gmod \To[\psi_*] (R\gmod)^\rev  \To[\Qt_w] (\tcatC_w)^\rev\To[\psi_*]\tcatC^*_w.$$

Note that the functor $\Phi_w$ is a faithful exact monoidal functor.
In particular, a morphism $f$ in $\Cw$ is 
a monomorphism (resp.\ epimorphism, isomorphism) if and only if so is $\Phi_w(f)$.

Now let us show that $\catC_w\to\tcatC_w$ is fully faithful.
In order to see this, we start by the following proposition.

\Prop\label{prop:sub}
Let $C$ be a real simple module in $R\gmod$ with an affinization $(\Ctr,z)$.
Let $X\in R\gmod$ and assume that
$\rmat{C,X}\seteq\Rre_{\Ctr,X}\vert_{z=0}\cl C\conv X\to X\conv C$ is an isomorphism up to a grading shift.
Then for any submodule $Y\subset C\conv X$, there exists 
$X'\subset X$ such that $Y=C\conv X'$.
\enprop
\Proof
Set $\rmat{}\seteq\rmat{C,X}$.

Let us consider the following commutative diagram
$$\xymatrix@C=7ex@R=4ex{
C\conv Y\ar[rr]\ar@{>->}[d]&&Y\conv C\ar[d]\ar@{>->}[d]\\
C\conv C\conv X\ar[r]^{\id}&C\conv C\conv X\ar[r]^{C\circ\rmat{}}&
C\conv X\conv C\;.}
$$
Hence we have
$C\conv\rmat{}(Y)\subset Y\conv C\subset C\conv X\conv C$.
By the \KO property, there exists $X'\subset X$ such that
$$\rmat{}(Y)\subset X'\conv C\qtq C\conv X'\subset Y.$$
Hence we have
$$\dim(Y)=\dim(\rmat{}(Y)) \le \dim(X'\conv C)= \dim(C\conv X')\le \dim(Y).$$
Since $\dim(C\conv X')=\dim(Y)$, we have
$Y=C\conv X'$, as desired.
\QED

\Cor
Let $M\in\Cw$.
For any $X\subset \Phi_w(M)$, there exists
$X'\subset M$ such that $X=\Phi_w(X')$.
\encor

\vs{3ex}

\Th\label{th:ff} 
$\Phi_w\cl \Cw\to\tCw$ is fully faithful.
\enth

\Proof

It is enough to show that
\eq\Hom_{R\gmod}(M,N)\To\Hom_{R\gmod}(\Ctr\conv M, \Ctr\conv N)\label{mor:ff}
\eneq
is bijective for any $M,N\in\Cw$ and $\Ctr=\dM(w\La,\La)$. 
The injectivity is evident. 
Let us show that any $f\in \Hom_{R\gmod}(\Ctr\conv M, \Ctr\conv N)$
is in the image of
$\Hom_{R\gmod}(M,N)$.

Let $\vphi\cl \Ctr\conv M\to(\Ctr\conv M)\oplus(\Ctr\conv N)$
be the graph of $f$, and let $Z\subset
(\Ctr\conv M)\oplus(\Ctr\conv N)\simeq\Ctr\conv (M\oplus N)$
be the image of $\vphi$.
Then, Proposition \ref{prop:sub} implies that there exists $K\subset M\oplus N$ such that
$Z=\Ctr\conv K$.
Let $p_1\cl K\to M$ and $p_2\cl K\to N$ be the projections.
Since the composition $\Ctr\conv M\isoto\Ctr\conv K\To[{p_1}]
\Ctr\conv M$ is the identity, the faithfulness of $\Phi_w$ 
implies that $p_1$ is an isomorphism.
Since the composition
$\Ctr\conv M\isoto\Ctr\conv K\To[{p_2}]
\Ctr\conv N$ is equal to $f$,
we obtain $f=\Ctr\conv(p_2\,p_1^{-1})$.
\QED

\Cor The full subcategory $\catC_w$ of $\tcatC_w$ is stable by taking subquotients.
\encor
\Proof
Let $Y\in\tcatC_w$ be a subobject of $X\in\catC_w$.
Then there exist $\La\in\pwtl$ such that
$\dC_\La\conv Y\subset \dC_\La\conv X$ belongs to $\catC_w$.
Then Proposition~\ref{prop:sub} implies that $Y\in\catC_w$.
\QED

\Rem
Since $\catC_w$ is a full subcategory of $\tcatC_w$, 
we sometimes regard an object of $\catC_w$ as an object of $\tcatC_w$,
namely we identify $\Phi_w(M)$ and $M$ for $M\in\catC_w$.
\enrem

\Rem\label{rem:ext} The full subcategory $\catC_{w}$ is not stable by taking extensions in $\tcatC_w$ in general.
Indeed taking a center $C=\dM(w\La,\La)$,
we have an extension of $\one$:
$$0\to C\conv C^{-1}\To (\widetilde{C}/z^2\widetilde{C})\conv C^{-1}\To  C\conv C^{-1}\To 0,$$
where $(\widetilde{C},z)$ is an affinization of $C$.
\enrem

Note that $\tcatC_w$ is a rigid monoidal abelian category satisfying \eqref{cond:g} with $\La=\rtl$.
Recall that a real simple object of $\tcatC_w$ is \afr
if it admits an affinization (see Definition~\ref{def:affinization}).
Then the following proposition immediately follows 
from Proposition~\ref{prop:simplehd} and Corollary~\ref{cor:dimEND}.
\Prop\label{prop:dim1}
Let $X,Y\in\tCw$ be simple objects.
Assume that one of them is \afr.
Then, we have
\eqn
\HOM_{\tCw}(X\conv Y,X\conv Y)&&=\cor\id_{X\circ Y},\\
\HOM_{\tCw}(X\conv Y,Y\conv X)&&=\cor\rmat{X,\, Y}
\eneqn
for some non-zero $\rmat{X, Y}$.
\enprop
Recall that $\rmat{X, Y}$ is the \emph{R-matrix} between $X$ and $Y$ and  denote by $\La(X,Y)$ the degree of $\rmat{X, Y}$.
If $X \simeq X'\conv \Ctr_\la$ for some $X' \in \Cw$, $\la \in \wtl$,  and $Y\in \Cw$, then we have
$$\La(X,Y) = \La(X',Y)+\La(\Ctr_\la,Y).$$
Note that $\La(\Ctr_\la,Y)=-\bl w\la+\la,\wt(Y)\br$.

\Lemma\label{lem:extcenter}
Let $\shc$ be either $R\gmod$ or $\Cw$, and let
$Q\cl \shc\to \tCw$ be the localization functor.
If $(\dM, R_\dM)$ is a rational center in $\shc$,
then there is a rational center $(Q(\dM),R_{Q(\dM)})$ in $\tCw$ such that
$R_{Q(\dM)}(Q(X))=Q\bl R_\dM(X)\br$ for any $X\in\shc$..
\enlemma

\Proof
Let $\sha$ be the category whose object is a pair
$(X,R)$ where $X\in\tCw$ and $R\cl Q(\dM)\conv X\to X\conv Q(\dM)$
is an isomorphism in $\Rat(\tCw)$.
The morphisms of $\sha$ are defined in an obvious way so that
we have a faithful functor $\sha\to\tCw$.
Then, $\sha$ is an abelian category.
It has a monoidal structure by
$(X, R)\tens (Y\tens S)=(Z,T)$,
where $Z=X \conv Y$ and $T$ is the composition
$$T\cl Q(\dM) \conv Z=Q(\dM)\conv X\conv Y\To[R]X\conv Q(\dM)\conv Y\To[S]
X\conv Y\conv Q(\dM)=Z\conv Q(\dM).$$
Now we have a monoidal functor $F\cl \shc\to\sha$ defined by
by
$$F(X)=\bl Q(X),Q(R_\dM(X))\br\qt{for $X\in\shc$.}$$
Then it is easy to see that
$F(\dC_i)$ is an invertible object of $\sha$ for any $i\in I$ and
$F(\coR_{\dC_i}(X))$ is an isomorphism for any $X\in\shc$.
Hence, by the universal property of $\shc\to\tCw$,
the functor $F\cl \shc\to \sha$
extends to
$\tF\cl\tCw\to\sha$.

Now, for any $X\in \tCw $, define
$R_{Q(\dM)}(X)$ by $R_{Q(\dM)}(X)=R$, where $\tF(X)=(X, R)\in\sha$.
Then $\bl Q(\dM),R_{Q(\dM)}\br$ is a rational center in $\tCw$.
\QED

\section{Resultant algebra and Affinization of Invariants} \label{Section: RIA}

\subsection{Resultant algebra}\label{sec:resultant}
Let $A$ be a commutative $\cor$-algebra.
Let $z$ be an indeterminate.
Recall that a polynomial $f(z)=\sum_{k=0}^{n}a_kz^{n-k}\in A[z]$
is a quasi-monic (resp.\ monic) polynomial of degree $n$ if $a_k\in A$ and $a_0\in \cor^{\times}$ (resp.\ $a_0=1$).
For a quasi-monic polynomial $f(z)$ of degree $m$
and a quasi-monic polynomial $g(z)$ of degree $n$,
we write their {\em resultant}
$f\res g$ for
$$\prod_{j=1}^m\prod_{k=1}^n(x_j-y_k)
\equiv\prod_{j=1}^m g(x_j)\equiv\prod_{k=1}^n f(y_k)
\mod \cor^{\times},$$
writing formally 
$f(z)=a\prod_{j=1}^m(z-x_j)$ and $g(z)=b\prod_{k=1}^n(z-y_k)$ with 
$a,b\in\cor^{\times}$.
Here $B\equiv C\mod \cor^{\times}$ means that $B=c\,C$ for some $c\in\cor^\times$.
It is well-defined since any symmetric polynomial of $\st{x_j}_{1\le j\le m}$ is expressed by the coefficients of $f(z)$.

We regard $f\res g$ as an element of $A/\cor^{\times}$.

The resultant product satisfies:
\eqn
&&f(z)\res g(z)\equiv g(z)\res f(z),\\
&&\bl f_1(z)\cdot f_2(z)\br\res g(z)\equiv \bl f_1(z)\res g(z)\br
\bl f_2(z)\res g(z)\br.
\eneqn

A quasi-monic rational function with coefficients in
$A$ is a quotient $f(z)/g(z)$
for quasi-monic homogeneous polynomials $f(z)$ and $g(z)$
with coefficients in $A$. We can consider them as an element of $A((z^{-1}))$,
the ring of Laurent series in $z^{-1}$ with coefficients in $A$.

We write $\res[z]$ when want to emphasize $z$. 

A quasi-monic rational homogeneous function in $z$ and $z'$ with coefficients in
$A$ is a quotient $f(z,z')/g(z,z')$
for quasi-monic homogeneous polynomials $f(z,z')$ and $g(z,z')$
in $z$ and $z'$ with coefficients in $A$.

Let $A$ be a commutative graded $\cor$-algebra such that $A_{\le 0}=\cor$.
For homogeneous indeterminates $z_k$ ($k=1,\ldots,r$) of positive degree,
we denote by $\shm(z_1,\ldots,z_r)\subset A[z_1,\ldots,z_r]/\cor^\times$
the set of homogeneous polynomials in $z_1,\ldots 
z_r$ with
coefficients in $A$ and quasi-monic in each $z_1$,\ldots $z_k$
modulo $\cor^{\times}$. 
Then we can define
$$\res[z_k]\cl \shm(z_1,z_2,\ldots, z_k)
\times\shm(z_k,\ldots,z_r)\to\shm(z_1,\ldots,z_{k-1},z_{k+1},\ldots,z_r).$$
The homogeneous degree 
of $f\res[z_k ] g$ for $f\in \shm(z_1,z_2,\ldots, z_k)$ and $g\in \shm(z_k,\ldots,z_r)$ is  given by
$\dfrac{\deg(f)\deg(g)}{\deg(z_k)}$, 
where $\deg(z_k)$, $\deg(f)$, and $\deg(g)$ denote the homogeneous degrees of $z_k$, $f$ and $g$, respectively.

In particular, if $z$ and $w$ be indeterminates with the same degree,
$\shm(z,w)$ has a structure of semiring 
(the existence of additive inverse is not assumed):
the addition in $\shm(z,w)$ is the product in $\cor[z,w]$, and the multiplication in $\shm(z,w)$ is given as follows: the  multiplication  $h(z,w)$ 
of $f(z,w)$ and $g(z,w)$ is given by 
$h(z_1,z_3)\equiv f(z_1,z_2)\res[z_2]g(z_2,z_3)$.
We call $\shm(z,w)$ the resultant algebra. The multiplication in
$\shm(z,w)$ is commutative. 
Let us denote by $$\deg\cl\shm(z,w) \to\Z $$ the map of taking the homogeneous degree.
Then $f\mapsto\deg(f)/\deg(z)$ is a semiring homomorphism.

\subsection{Affinizations of invariants for $R$ -modules}
\label{subsec:Inv}

Let $R$ be a quiver Hecke algebra as in section \ref{Sec:QHSW}.
For $i\in I $ and $n\in \Z_{>0}$,  let 
 $P(i^{n})$  be the indecomposable  projective $R(n \alpha_i)$-module
whose head is isomorphic to $L(i^n)$.
Then,  for an $R(\beta)$-module $M$ we define
\eqn&&
\ba{ll}
E_i^{(n)} M& \seteq \HOM_{R(n\alpha_i)} \bl P(i^{n}),\, e(n\alpha_i, \beta - n\alpha_i) M\br \in   \Mod(R(\beta-n\al_i)), \\
F_i^{(n)} M& \seteq  P(i^{n}) \conv M \in \Mod(R(\beta+n\al_i)). 
\ea
\eneqn
The functors $E_i^{(1)}$ and $F_i^{(1)}$ will be denoted by $E_i$ and $F_i$.

For $i \in I$ and a non-zero  $R(\beta)$-module $M$,  we define
\begin{align*}
\wt(M) = - \beta, \ \   \ep_i(M) = \max \{ k \ge 0 \mid E_i^k M \not\simeq 0 \}, \ \  \ph_i(M) = \ep_i(M) + \langle h_i, \wt(M) \rangle.
\end{align*}
We also define $E_i^*$, $F_i^*$, $\ep^*_i$, etc.\ in the same manner as above by replacing $e(n\alpha_i, \beta-n\al_i)$, $P(i^n)\conv -$,
etc.\  with 
$e(\beta-n \al_i, n\alpha_i)$, $- \conv P(i^n)$, etc.

\Def
Let $i\in I$ and let $M$ be an $R(\beta)$-module.
Set $m=\eps_i(M)$,  $m'=\eps^*_i(M)$, $n=\height{\beta}$.
We set
\eq\hs{3ex}
&&\left\{\hs{-10ex}
\ba{rcl}
\chi_i(M)\bl t_i\br&=&\Bigl(\sum_{\nu\in I^\beta}\prod_{\nu_k=i}(t_i-x_k)e(\nu)\Bigr)
\bigm|_M
\in \tEnd_R(M)[t_i],\\[.5ex]
\bchi_i(M)\bl t_i\br&=&(t_i-x_1)\cdots (t_i-x_m)\bigm|_{\E{i}^{(m)}M}
\in \tEnd_R(\E{i}^{(m)}M)[t_i],\\[.5ex]
\akew[10ex]\bchis_i(M)\bl t_i\br&=&(t_i-x_{n-m'+1})\cdots (t_i-x_n)\bigm|_{\Es{i}{}^{(m')}M}
\in \tEnd_R(\Es{i}{}^{(m')}M)[t_i].
\ea\right.
\eneq \label{Eq: lift of chiE}
\edf
Note that, with the assignment $\deg(t_i)=(\al_i,\al_i)$, the monic polynomials
$\chi_i(M)$, $\bchi_i(M)$ and $\bchis_i(M)$ are homogeneous.

For an \afn $(\Ma,\z)$ of a real simple module, we have
$$\chi_i(\Ma),\bchi_i(\Ma),\bchis_i(\Ma)\in\cor[t_i,\z],$$ 
and they satisfy
$$\chi_i(\Ma)\equiv \bchi_i(\Ma)\cdot\chi_i\bl\E{i}^{(m)}\Ma\br
\equiv\bchis_i(\Ma)\cdot\chi_i\bl\Es{i}{}^{(m')}\Ma\br,$$
where $m=\eps_i(\Ma)$ and $m'=\eps^*_i(\Ma)$ (see \cite[Lemma 2.7 (i)]{KP18} and \cite[Lemma 3.3]{KKOP21}).

Recall that for quasi-monic polynomials $f$ and $g$, ``$f\equiv g$'' means
``$f\equiv g\ \mathrm{mod}\ \cor^\times$''.

\Lemma\label{lem:Rtau} 
Let $M\in \Mod(R(\beta))$, $N\in \Mod(R(\gamma))$
with $\beta,\gamma\in\prtl$, $m=\height{\beta}$, $n=\height{\gamma}$.
Then for $u\in M$ and $v\in N$, we have
$$\Runi_{M,N}(u\etens v)
\in \tau_{w[n,m]} \prod_{i\in I}\chi_i(N)\res[t_i]\chi_i(M)\bl v\etens u\br
+\sum_{w<w[n,m]}\tau_w(N\etens M).$$
\enlemma
\Proof
By \cite[(1.17)]{K^3},  we have
$$\Runi_{M,N}(u\etens v)
\in \tau_{w[n,m]}\sum_{\nu \in I^{\beta+\gamma}}\prod\limits_{\substack{1\le a\le n<b\le n+m,\\\nu_a=\nu_b}}(x_a-x_b) e(\nu)\bl v\etens u\br
+\sum_{w<w[n,m]}\tau_w(N\etens M).$$
\QED

\Def\label{def:inv}
 Let $(\Ma,\z)$ be an affinization of a real  simple module $M$, and
let $(\Na,\zN)$ be an affinization of a real simple $N$.
\bna
\item
Recall that  $\Daf(\Ma,\Na)\in\cor[\z,\zN]$ the quasi-monic polynomial given by
$$ \Rre_{\Na,\Ma}\circ\Rre_{\Ma,\Na}=\Daf(\Ma,\Na)\id_{\Ma\,\circ\,\Na}.$$
\item Denote by $\tLaf(\Ma,\Na)\in \cor[\z,\zN]$ the quasi-monic polynomial given by
\eqn
\Rre_{\Ma,\Na}(u\etens v)
&\in&\tLaf(\Ma,\Na)\tau_{w[n,m]}(v\etens u)
+\sum_{w<w[n,m]}\tau_w(\Na\etens \Ma)
\eneqn
for $u\in\Ma$ and $v\in\Na$. \label{itemtLa} 
Note that by Proposition \ref{pro:onedimhom}, 
 $\Rre_{\Ma,\Na}$ is proportional to $\Runi_{\Ma,\Na}$ and hence such a quasi-monic polynomial $\tLaf(\Ma,\Na)$ exists by Lemma \ref{lem:Rtau}.

\item We set
\eqn
\wtaf(\Ma,\Na)&\seteq&\dfrac{\prod_{i\in I}\bl\chi_i(\Ma)\res[{t_i}]\chi_i(\Na)\br^2}%
{\prod_{i\not=j}\chi_i(\Ma)\res[{t_i}]Q_{i,j}(t_i,t_j)\res[{t_j}]\chi_j(\Na)}
\in\cor(\z,\zN)/\cor^\times,\\
\Laf(\Ma,\Na)&\seteq&
\tLaf(\Ma,\Na)^2/\wtaf(\Ma,\Na)\in\cor(\z,\zN)/\cor^\times.
\eneqn
\ee
\edf

Note that $\wtaf(\Ma,\Na) \equiv \wtaf(\Na,\Ma)$.

For simples $M,N$ in $R\gmod$ such that one of them is \afr, define
$$\tLa(M,N):=\dfrac{1}{2}\bl \La(M,N)+(\wt(M),\wt(N)) \br.$$

The homogeneous degrees of these invariants are given as follows:
\eqn
\deg\bl\chi_i(\Ma)\br&=&n_i(\al_i,\al_i)\qt{where 
$\wt(M)=-\smash{\sum\nolimits_{i\in I}}n_i\al_i$,}\\
\deg\bl\bchi_i(\Ma)\br&=&\eps_i(M)(\al_i,\al_i)=2\tLa(L(i),M),\\
\deg\bl\bchis_i(\Ma)\br&=&\eps^*_i(M)(\al_i,\al_i)=2\tLa(M,L(i)),\\
\deg\bl\Daf(\Ma,\Na)\br&=&2\de(M,N),\\
\deg\bl\tLaf(\Ma,\Na)\br&=&2\tLa(M,N),\\
\deg\bl\Laf(\Ma,\Na)\br&=&2\La(M,N),\\
\deg\bl\wtaf(\Ma,\Na)\br&=&2\bl\wt(M),\wt(N)\br.
\eneqn
Indeed,  the fifth equality is shown in \cite[Lemma 3.11]{KKOP21}.  For the last equality we have
\eqn
&&\deg(\wtaf(\Ma,\Na))\\
&& = \sum_{i\in I} \left( 2 \deg(\chi_i(\Ma)) \deg(\chi_i(\Na)) / \deg(t_i) \right)
-\sum_{i\neq j}  
\left(  \deg(\chi_i(\Ma)) \deg(\chi_j(\Na)) \dfrac{\deg(Q_{i,j}(t_i,t_j))}{ \deg(t_i)\deg(t_j)} \right) \\
&&=2\sum_{i,j\in I} \left( \dfrac{\deg(\chi_i(\Ma))}{(\al_i,\al_i)} \dfrac{\deg(\chi_j(\Na))}{(\al_j,\al_j)} (\al_i,\al_j) \right)
=2\bl\wt(\Ma),\wt(\Na)\br.
\eneqn
The others follow directly from their definitions.    See also \cite[Corollary 3.8]{KKOP18} for the second and the third equalities.
  
Therefore, these invariants are affinizations of
$\wt$, $\de$, $\tLa$, etc.

By Lemma~\ref{lem:Rtau} and Definition~\ref{def:inv} \eqref{itemtLa},
one has
\eq \label{eq:ratio_Runi_Rre}
\Runi_{\Ma,\Na}&=&\dfrac{\ \prod_{i\in I}\chi_i(\Ma)\res[{t_i}]\chi_i(\Na)\ }%
{\tLaf(\Ma,\Na)}\Rre_{\Ma,\Na}.\label{R,Rnorm}
\eneq

\Lemma\label{lem:relaff}We have
\eqn&&\tLaf(\Ma,\Na)\cdot\tLaf(\Na,\Ma)\equiv\Daf(\Ma,\Na)\cdot\wtaf(\Ma,\Na),\\
&&\Daf(\Ma,\Na)^2\equiv\Laf(\Ma,\Na)\cdot\Laf(\Na,\Ma).
\eneqn
\enlemma
\Proof
By \cite[Proposition~1.10 (iv)]{K^3}
for $u\in\Ma$ and $v\in\Na$, we have
\eq \label{eq:RNMRMN}
\Runi_{\Na,\Ma}\Runi_{\Ma,\Na}(u\etens v)
&&=\sum_{\substack{
1\le a\le m<b\le m+n,\\[.2ex]
\nu\in I^{\beta+\gamma},\,\nu_a\not=\nu_b}}
Q_{\nu_a,\nu_b}(x_a,x_b)e(\nu)\bl u\etens v\br\\ \nonumber
&&=\Bigl(\prod_{i\not=j}\chi_i(\Ma)\res[{t_i}]Q_{i,j}(t_i,t_j)
\res[{t_j}]\chi_j(\Na)\Bigr)(u\etens v).
\eneq
On the other hand, \eqref{R,Rnorm} implies
\eqn
\Runi_{\Na,\Ma}\Runi_{\Ma,\Na}(u\etens v)
&&=\dfrac{\prod_{i\in I}\bl\chi_i(\Ma)\res[{t_i}]\chi_i(\Na)\br^2}%
{\tLaf(\Ma,\Na)\tLaf(\Na,\Ma)}\Rre_{\Na,\Ma}\Rre_{\Ma,\Na}(u\etens v)\\
&&=\dfrac{\prod_{i\in I}\bl\chi_i(\Ma)\res[{t_i}]\chi_i(\Na)\br^2}%
{\tLaf(\Ma,\Na)\tLaf(\Na,\Ma)}\Daf(\Ma,\Na)(u\etens v).
\eneqn
Hence, we obtain
$$\prod_{i\not=j}\chi_i(\Ma)\res[{t_i}]Q_{i,j}(t_i,t_j)
\res[{t_j}]\chi_j(\Na)=\dfrac{\prod_{i\in I}\bl\chi_i(\Ma)\res[{t_i}]\chi_i(\Na)\br^2}%
{\tLaf(\Ma,\Na)\tLaf(\Na,\Ma)}\Daf(\Ma,\Na),$$
which implies the desired results.
\QED

\Lemma \label{lem:prodtL}
Let $(\Laa, w)$, $(\Ma,z)$, $(\Na,z)$, $(\Xa,z)$ be  affinizations of simples $L,M,N,X$  in $R\gmod$, respectively, and assume that $L$ is real. 
If there is an epimorphism 
$\Ma \conv[z] \Na \epito \Xa$ in $\Proc(\cor[z], R\gmod)$, then there exists $a \in \cor[z,w] \setminus\{0\}$  such that the diagram below is commutative.
\eqn
\xymatrix{
\Laa \conv \Ma \conv \Na  \ar_{\Rre_{\Laa,\Ma}}[r] \ar@{->>}[d] & \Ma \conv \Laa \conv \Na \ar_{\Rre_{\Laa,\Na}}[r] & \Ma \conv \Na \conv \Laa \ar@{->>}[d]\\
\Laa \conv \Xa \ar^{a \, \Rre_{\Laa,\Xa}}[rr]&&  \Xa \conv \Laa.
}
\eneqn
In particular, we have $\tLaf(\Laa, \Ma) \tLaf(\Laa,\Na) = a \tLaf(\Laa, \Xa)$. 
If we further assume that $\La(L,X)=\La(L,M)+\La(L,N)$, then $a\equiv 1$ in $\cor[z,w]$.
\enlemma
\Proof
Since $\HOM_{\cor[z,w]}(\Laa\conv\Xa, \Xa\conv\Laa) = \cor[z,w]  \Rre_{\Laa,\Xa}$ by Proposition \ref{pro:onedimhom}, such an $a \in \cor[z,w] $ exists.  The other assertions are immediate.
\QED
 
\Lemma \label{lem:tLaii}
Let $i\in I$,  $M\in R(m\al_i)\gmod$, and $N\in R(n\al_i)\gmod$. 
Assume that $(\Ma,\z)$ and $(\Na,\zN)$  are affinizations of $M$ and $N$ respectively.
Then $\Rre_{\Ma,\Na} = \Runi_{\Ma,\Na}$ and $\tLaf(\Ma,\Na)\equiv \chi_i(\Ma) \res[t_i] \chi_i(\Na)$.
\enlemma
\begin{proof}
By \eqref{eq:RNMRMN},  $\Runi_{\Ma,\Na} \circ \Runi_{\Na,\Ma} = \id$ and hence we have $\Rre_{\Ma,\Na} = \Runi_{\Ma,\Na}$. 
By \eqref{eq:ratio_Runi_Rre}, we have $\tLaf(\Ma,\Na)\equiv \chi_i(\Ma) \res[t_i] \chi_i(\Na)$.
\end{proof}

\Prop
Let $\Ma$ be an affinization of a simple module $M$ in $R\gmod$.
Then for each $i\in I$, we have 
\eq
\ba{rcl}
\bchi_i(\Ma)&&\equiv\tLaf(\tL(i)_{t_i},\Ma),\\[.5ex]
\bchis_i(\Ma)&&\equiv\tLaf(\Ma,\tL(i)_{t_i}),\\[1ex]
\wtaf(\tL(i)_{t_i},\Ma)&&\equiv\dfrac{\chi_i(\Ma)^2}
{\prod_{j\not=i}Q_{i,j}(t_i,t_j)\res[{t_j}]\chi_j(\Ma)}\;.
\ea
\eneq
Here $\tL(i)_{t_i}$ is the affinization of $L(i)$ given by 
$R(\al_i)$ with $t_i=x_1$.
\enprop
\Proof

Let $m:=\eps_i(M)$, $M_0:=E_i^{(m)}M$, and $\Ma_0:=E_i^{(m)}\Ma$ such that $\Ma_0$ is an affinization of $M_0$. 
There exists an affinization $\Laa$  of $L(i^m)$ such that $\chi_i(\Laa) = \bchi_i(\Ma)$ by \cite[Proposition 3.6]{KKOP21}. 
Then we have an epimorphism $\Laa \conv \Ma_0 \epito \Ma$ due to the presentation of $\Laa$ in \cite[Remark 3.8]{KKOP21}. 
Hence  we have
\eqn
\tLaf(\ang{i}_{t_i}, \Ma)  \equiv \tLaf(\ang{i}_{t_i}, \Laa) \, \tLaf (\ang{i}_{t_i}, \Ma_0)   \equiv  \tLaf(\ang{i}_{t_i}, \Laa) =\chi_i(\Laa) =\bchi_i(\Ma).
\eneqn
Here the first equality follows from \cite[Corollary 2.12]{KKOP22} and Lemma \ref{lem:prodtL} together with  the fact that $\ang{i}$ and $M_0$ are unmixed.
The  second equality follows from Lemma \ref{lem:unmixed} and the third equality follows from Lemma \ref{lem:tLaii}.

The second assertion can be proved in a  similar way, and
the last assertion is immediate.
\QED

\Prop
Let $i\in I$, and let  $(\Ma,\z)$ be an affinization of a simple module $M$.
Set $m=\eps_i(M)$ and $\Ma_0=E_i^{(m)}\Ma$. 
Then we have
\eqn
&&
\Daf\bl \Ma,\tL(i)_{t_i}\br\chi_i(\Ma)\chi_i(\Ma_0)
\equiv\bchis_i(\Ma)\prod_{j\not=i} Q_{ij}(t_i,t_j)\res[t_j]\chi_j(\Ma_0).
\eneqn
\enprop

\Proof
Set $\Na=\tL(i)_{t_i}$. Then
Lemma~\ref{lem:relaff} implies
\eqn
&&\tLaf(\Ma,\Na)\cdot\tLaf(\Na,\Ma)\equiv\Daf(\Ma,\Na)\cdot\wtaf(\Ma,\Na),
\eneqn
which reads as 
\eqn\bchi_i(\Ma)\bchis_i(\Ma)
&&\equiv\Daf(\Ma,\Na)\cdot\dfrac{\chi_i(\Ma)^2}
{\prod_{j\not=i}Q_{i,j}(t_i,t_j)\res[{t_j}]\chi_j(\Ma)}\\
&&\equiv\Daf(\Ma,\Na)\cdot\dfrac{\chi_i(\Ma)\bchi_i(\Ma)\chi_i(\Ma_0)}
{\prod_{j\not=i}Q_{i,j}(t_i,t_j)\res[{t_j}]\chi_j(\Ma_0)}\;.\eneqn

Hence, we obtain the desired result.
\QED

\Cor\label{cor:XX}
Let $i\in I$, and let  $(\Ma,\z)$ be an affinization of a simple module $M$.
Set $m=\eps_i(M)$ and $\Ma_0=E_i^{(m)}\Ma$.  We assume further that
$\de_i(M)=0$, i.e., $M$ and $L(i)$ commute.
\ro Note that $\de_i(M)\seteq\eps_i(M)+\eps^*_i(M)+\ang{h_i,\wt(M)}$.\rf 
Then we have
\eqn
&&\chi_i(\Ma)\chi_i(\Ma_0)\equiv\bchis_i(\Ma)\prod_{j\not=i}
 Q_{ij}(t_i,t_j)\res[t_j]\chi_j(\Ma_0).
\eneqn
\encor
\Proof
It follows from
$\de(M,L(i))=0$, which implies $\Daf(\Ma,\tL(i)_{t_i})\equiv1$.
\QED

We say that an ordered pair $(M,N)$ of $R$-modules is {\em unmixed}
if $$\sgW(M)\cap\gW(N)\subset\{0\}.$$ 

\Lemma \label{lem:unmixed}
Let $\Ma$ and $\Na$ be affinizations in $R\gmod$ of real simple modules $M$ and $N$ respectively.
If $(M,N)$ is unmixed, then
$\tLaf(\Ma,\Na)\equiv1$.
\enlemma 
\Proof
By \cite[Corollary 2.12 ]{KKOP22}, we have $\tLa(M,N) = 0$ and hence $\tLaf(\Ma,\Na)\equiv1$.
\QED

\subsection{Examples of affinizations}
The following result will be used later.
\Lemma\label{lem:ij}
 Let $c\in\Z_{>0}$.
Assume that $Q_{i,j}(u,v)=v-u^c$
with $\ang{h_i,\al_j}=-c$ and $\ang{h_j,\al_i}=-1$.

For $m\in\Z$ such that $0\le m\le c$, let
$L=\ang{i^m}$.
Let $(\dL,z)$ be an affinization of $L$ with $\deg(z)=(\al_i,\al_i)$ such that
$\chi_i(\dL)\bl t_i\br=f(t_i,z)$ where $f(t_i,z)$ is quasi-monic with degree
$m$ in $t_i$ and also in $z$ \ro see \cite[Proposition 3.6]{KKOP21}\rf.
\bnum
\item $M\seteq L\hconv\ang{j}\simeq L\tens \ang{j}$ as an $R(m\al_i)\etens R(\al_j)$-module. Moreover, $\tau_{m}$ acts by $0$ on $L\hconv\ang{j}$.
\item
If $f(t_i,z)$ divides $Q_{i,j}(t_i,z^c)$, then
$\dM\seteq\dL\tens_{\cor[z]}\ang{j}_{z^c}$ has a structure of $R(m\al_i+\al_j)$-module
with $\tau_m=0$.
Moreover $\dM$ is an affinization of $M$.
\ee
\enlemma
\Proof
(i) We can check easily the defining relation of $R(m\al_i+\al_j)$ on $M$.
In particular, the  relation $\tau_m^2=Q_{i,j}(x_m,x_{m+1})$ follows from
$x_m^m\vert_L=0$.

\snoi
(ii) We can check easily the defining relation of $R(m\al_i+\al_j)$ on $\dM$.
In particular, the relation $\tau_m^2=Q_{i,j}(x_m,x_{m+1})$ is a consequence of
$Q_{i,j}(x_m,x_{m+1})\vert_\dM=0$ which follows from
$f(x_m,z)\vert_\dL=0$ and $Q_{i,j}(x_m,x_{m+1})=Q_{i,j}(x_m,z^c)$
on $\dM$.

\QED

\section{Saito Reflection functors for $R \gmod$   via localization} \label{Sec: reflection}
\subsection{$\catC_{s_iw_0}$  and  $\catC^*_{s_iw_0}\,$ }

\emph{
In the sequel, 
we assume that $\cmA=\st{\sfc_{i,j}}_{i,j\in I} $ is a Cartan matrix of finite type such that
$\cmA$ is not of type $A_1$.}

Let $w_0$ be the longest element of the Weyl group $\weyl$.
Let $i\mapsto i^*$ be the involution on $I$ determined by $w_0(\al_i)=-\al_{i^*}$.

Throughout this section,  we fix $i\in I$.  

We have
\eqn
\catC_{s_iw_0}=\st{M\in R\gmod\bigm|\E{i}M \simeq 0},\\
\catC^*_{s_iw_0}=\st{M\in R\gmod\bigm|\Es{i}M \simeq 0},
\eneqn
where $\E{i}$ and $\Es{i}$ denote the functor  $R(\beta)\gmod \to R(\beta-\al_i)\gmod$ given by  $M \mapsto e(\al_i,\beta-\al_i) M$ 
and $M \mapsto e(\beta-\al_i, \al_i) M$, respectively.

Note that $\catC_{s_iw_0}\simeq\catC^*_{s_iw_0}\simeq0$ in case $\cmA$ is of type
$A_1$.

Then
$\st{\dM(s_iw_0\La,\La)}_{\La\in\pwtl}$ is a family of central objects in
$\catC_{s_iw_0}$, and
$\st{\dM(w_0\La,s_i\La)}_{\La\in\pwtl}$ is a family of central objects  in
$\catC^*_{s_iw_0}$.
Note that
$$\dM(s_iw_0\La,\La)=\E{i}^{(\ang{h_{i^*},\La})}\dM(w_0\La,\La)\qtq
\dM(w_0\La,s_i\La)=\Es{i}{}^{(\ang{h_i,\La})}\dM(w_0\La,\La).
$$

Set $\Ctr_\La\seteq\dM(s_iw_0\La,\La)$ and $\Ctr^*_\La\seteq\dM(w_0\La,s_i\La)$
for $\La\in\pwtl$.
We set $\Ctr_j=\Ctr_{\La_j}$ and $\Ctr^*_j=\Ctr^*_{\La_j}$.
Since we have
$$\psi_*\bl\dM(v\La,v'\La)\br\simeq\dM(-v'\La,-v\La)$$
for any $\La\in\pwtl$ and any $v,v'\in\weyl$ such that $v\ge v'$
(\cite[Lemma 2.23]{KKOP22}),
we have
$$\psi_*(\Ctr_j)\simeq\Ctr^*_{j^*}\qt{for any $j\in I$.}$$
(For $\psi$, see \eqref{def:psi}.)
Then we have (see \cite[\S\,5.1]{KKOP21} and  \cite[Corollary 3.8]{KKOP18})
\eqn&&\ba{l}
\La(\Ctr_\La,\ang{j})=(s_iw_0\La+\La,\al_j),\\
\La(\ang{j} , \Ctr^*_\La)=-(w_0\La+s_i\La,\al_j),
\ea\hs{3ex}\qt{for any $j\in I$.}
\eneqn

We have
\eqn
\Qt_{s_iw_0}\bl \ang{i}\br
&&\simeq\Ctr_{i^*}^{\circ-1}\conv\bl\Ctr_{i^*}\hconv \ang{i}\br,\\
\Qt^*_{s_iw_0}\bl \ang{i}\br
&&\simeq\bl\ang{i}\hconv \Ctr^*_{i}\br\conv \Ctr^*_{i}{}^{\circ-1}.
\eneqn

\subsection{Schur-Weyl datum}
For any $j\in I$, define an object of $\tcatC^*_{s_iw_0}$:
\eqn
\SW_j&\seteq
\bc
\D\bl\Qt_{s_iw_0}^*(\ang{i})\br\simeq(\E{i}\Ctr^*_{i^*})\conv{\Ctr^*_{i^*}}{}^{\circ-1}&\text{if $j=i$,}\\
\dM(s_is_j\La_j,\La_j)\simeq \ang{i^{-\sfc_{i,j}}}\hconv \ang{j}&\text{if $j\not=i$.}
\ec
\eneqn
Recall that $\D$ is the right dual functor.
Note that
$\SW_j\simeq \ang{i^{-\sfc_{i,j}}}\tens \ang{j}$ as a vector space for $j\not=i$ since $\Es{k}(\SW_j)\simeq0$ for any $k\in I\setminus\st{j}$.

Since $\ang{i}=L(i)=R(\al_i)/R(\al_i)x_1$ and
$\ang{i^{-\sfc_{i,j}}}\hconv \ang{j}\simeq \dM(s_is_j\La_j,\La_j)$ 
have the canonical affinizations
 by  \cite[Theorem 3.26]{KKOP21},  
we have the following affinizations of $\SW_j$'s by Lemma \ref{lem:dualAff} and Lemma~\ref{lem:extcenter}. 
\begin{align} \label{Eq: aff hK}
\hSW_j&\seteq
\bc
\Da\bl\Qt_{s_iw_0}^*(\ang{i}_{z_i})\br&\text{if $j=i$,}\\
\bl \ang{i^{-\sfc_{i,j}}}\hconv \ang{j}\br_{z_j}&\text{if $j\not=i$.}
\ec
\end{align}
Here the functor $\Qt_{s_iw_0}^* : \Pro(R\gmod) \to \Pro(\tcatC^*_{s_iw_0})$ is induced by $\Qt_{s_iw_0}^* : R\gmod \to \tcatC^*_{s_iw_0}$.
Note that $\deg z_j=(\al_j,\al_j)$ for any $j\in I$.
For $j\in I\setminus\st{i}$, we have
\eq \label{eq:chiK}
\chi_k\bl \hSW_j\br(t_k)\equiv
\bc t_j-z_j&\text{if $k=j$,}\\
Q_{i,j}(t_i,z_j)&\text{if $k=i$,}\\
1&\text{otherwise.}
\ec
\eneq
Note that $\chi_i(\hSW_j)\equiv Q_{i,j}(t_i,z_j)$ follows also from
Corollary~\ref{cor:XX}.

We have
$$\E{i}\Ctr^*_{i^*}\simeq\dM(s_iw_0\La_{i^*},s_i\La_{i^*}).$$
(Note that $s_i\le s_iw_0$ since $\g\not=A_1$.)

Note that
$$\La\bl \ang{j},\ang{k}\br=
-(\al_j,\al_k)\delta(j\not=k)\qt{for any $j,k\in I$.}$$

\Prop \label{prop:LaKK}
For any $j,k\in I$, we have
$$\La(\SW_j,\SW_k)=\La\bl \ang{j},\ang{k}\br$$
except in the case $\cartan$ is of type $A_2$ and  $j\not=k=i$.
In the last case, we have $\La\bl \SW_j,\SW_k\br=-1$.
In particular, we have
$$\de(\SW_j,\SW_k)=\de\bl \ang{j},\ang{k}\br$$
unless $\cartan$ is of type $A_2$.
\enprop
\Proof
When $\sfC$ is of type $A_2$, we can prove directly.
Hence we assume that $\sfC$ is not of type $A_2$.

\snoi
(i)\ Let $j,j'\in I\setminus\st{i}$ satisfy $j\not=j'$.
Set $a=-\ang{h_i,\al_j}$ and $a'=-\ang{h_i,\al_{j'}}$.
Then we have
\eqn
\La(\SW_j,\SW_{j'})&&=\La\bl \ang{i^a}\hconv \ang{j},\ang{i^{a'}}\hconv \ang{j'}\br\\
&&\underset{(1)}=
\La\bl \ang{i^a}\hconv \ang{j},\ang{i^{a'}}\br
+\La\bl \ang{i^a}\hconv \ang{j},\ang{j'}\br\\
&&\underset{(2)}=
-\La\bl \ang{i^{a'}},\ang{i^a}\hconv \ang{j}\br
+\La\bl \ang{i^a}\hconv \ang{j},\ang{j'}\br\\
&&\underset{(3)}=
-\La\bl \ang{i^{a'}},\ang{j}\br
+\La\bl \ang{i^a},\ang{j'}\br
+\La\bl \ang{j},\ang{j'}\br\\
&&=
a'(\al_i,\al_j)
-a(\al_i,\al_{j'})
+\La\bl\ang{j} ,\ang{j'}\br\\
&&=\La\bl\ang{j} ,\ang{j'}\br.
\eneqn
Here, $\underset{(1)}=$ and $\underset{(2)}=$ follow from the fact that
$\ang{i^a} \hconv \ang{j}$ and $\ang{i^{a'}}$ commute,
The equality $\underset{(3)}=$ follows from the fact
$\bl\ang{i^a},\ang{j'}\br$ is unmixed and \cite[Corollary 2.3, Corollary 2.12]{KKOP22}.

\mnoi
(ii)\ For $j\in I\setminus\st{i}$,
let us show that $\La(\SW_j,\SW_i)=\La(\ang{j},\ang{i})$.
Set $a=-\ang{h_i,\al_j}$.
Since $\ang{i}$ commutes with $\SW_j$, we have
$$
\La(\SW_j,\Ctr_{i^*}^*)=\La(\SW_j,\ang{i}\hconv\E{i}\Ctr_{i^*}^*)=
\La(\SW_j,\ang{i})+\La(\SW_j,\E{i}\Ctr_{i^*}^*).$$
Hence
\eqn
\La(\SW_j,\SW_i)&&=\La(\SW_j,\E{i}\Ctr_{i^*}^*)-\La(\SW_j,\Ctr_{i^*}^*)\\
&&=-\La(\SW_j,\ang{i})=\La(\ang{i},\SW_j)\\
&&=\La(\ang{i},\ang{i^a}\hconv\ang{j})
=\La(\ang{i},\ang{j})=-(\al_i,\al_j)=\La(\ang{j},\ang{i}).
\eneqn

\snoi
(iii)\ For $j\in I\setminus\st{i}$,
let us show $\La(\SW_i,\SW_j)=\La(\ang{i},\ang{j})$.

Set $C=\Ctr^*_{i^*}$.
We have
$$
\bl\wt (\E{i}C),\wt(\SW_j)\br-\bl\wt (C),\wt(\SW_j)\br
=(\al_i,-s_i\al_j)=(\al_i,\al_j).$$
Hence we have
\eqn
&&\La(\SW_i,\SW_j)-\La(\ang{i},\ang{j})\\
&&\hs{3ex}=\La(\E{i}C,\SW_j)-\La(C,\SW_j)+(\al_i,\al_j)\\
&&\hs{3ex}=\Bigl(\La(\E{i}C,\SW_j)+\bl\wt (\E{i}C),\wt(\SW_j)\br\Bigr)
-\Bigl(\La(C,\SW_j)+\bl\wt (C),\wt(\SW_j)\br\Bigr)\\
&&\hs{3ex}=2\tLa(\E{i}C,\SW_j)-2\tLa(C,\SW_j)\\
&&\hs{3ex}=2\tLa(\E{i}C,\SW_j)-2\tLa(L(i)\hconv(\E{i}C) ,\SW_j)\\
&&\hs{3ex}\le0.\eneqn
Here the last inequality follows from
\cite[Theorem 2.11]{KKOP22}.
Thus, we have obtained
$$\La(\SW_i,\SW_j)\le\La(\ang{i},\ang{j}).$$
Hence it is enough to show 
\eq
\text{if $\La(\SW_i,\SW_j)<\La(\ang{i},\ang{j})$,
then $\sfC$ is of type $A_2$.}
\eneq

Assume that
$\La(\SW_i,\SW_j)<\La(\ang{i},\ang{j})$.
Then (ii) implies that
$$\de(\SW_i,\SW_j)<\de(\ang{i},\ang{j})=\delta\bl(\al_i,\al_j)<0\br\max(\sfd_i,\sfd_j).$$
Here $\sfd_k=(\al_k,\al_k)/2$.
Hence we have $(\al_i,\al_j)<0$.
Since $\SW_i$ and $C$ are \afr of degree $2\sfd_i$, and
$\SW_j$ is \afr of degree $2\sfd_j$, we have
$\de(\SW_i,\SW_j)\in\Z\max(\sfd_i,\sfd_j)$.
Thus we obtain $\de(\SW_i,\SW_j)=0$.
Hence we have
$\de(E_iC,K_j)=0$. Hence, $\SW_j\conv(\E{i}C)$ is simple.
Set $c=-\ang{h_i,\al_j}>0$.
Then we have $\SW_j=\ang{i^c}\hconv\ang{j}$.
The composition of
$$(\ang{i^{c}}\hconv\ang{j})\conv(\E{i}C)\monoto(\ang{i^{c-1}}\hconv\ang{j})\conv\ang{i}\conv
(\E{i}C)\epito (\ang{i^{c-1}}\hconv\ang{j})\conv C$$
does not vanish by \cite[Lemma 3.1.5]{KKKO18}.
Since $(\ang{i^{c-1}}\hconv\ang{j})\conv C$ is also simple, we obtain
$$(\ang{i^{c}}\hconv\ang{j})\conv(\E{i}C)\simeq (\ang{i^{c-1}}\hconv\ang{j})\conv C.$$

Comparing their dimensions, we obtain
$$c\,!\cdot\dim(\E{i}C)\cdot\dfrac{(c+s)!}{(c+1)!\;(s-1)!}
=(c-1)!\cdot\dim(C)\cdot\dfrac{(c+s)!}{c!\;s!}\;.$$
Here, $s\seteq\height{\wt(C)}=\height{\La_{i^*}-w_0\La_{i^*}}-\delta_{i,i^*}$.
Note that $\ang{i^a}\tens \ang{j}\to\ang{i^a}\hconv\ang{j}$
is bijective if $0\le a\le-\ang{h_i,\al_j}$.

Since $\E{j}C\simeq0$ for any $j\in I\setminus\st{i}$, we obtain
$\dim(\E{i}C)=\dim C$.
Thus we obtain
$$s=\dfrac{c+1}{c}\;.$$
Hence we obtain $c=1$ and $s=2$.
Note that $\rank\, \g\le s+1=3$.
 Then we can easily check that if $ 2\le \rank(\g)\le3$ and $c=1$,  then $s\ge 3$ unless $\g=A_2$.
\QED

\Prop \label{Prop: De}
For any $j,k\in I$ such that $j\not=k$, we have
$$\Daf(\hSW_j,\hSW_k)\equiv Q_{j,k}(z_j,z_k)$$
except in the case $\cartan$ is of type $A_2$.
In the last case, we have $\Daf\bl \hSW_j,\hSW_k\br\equiv1$.
\enprop
\Proof
When $\sfC$ is of type $A_2$, we can prove directly.
Hence assume that $\sfC$ is not of type $A_2$.

\snoi
(i)\ Let $j,j'\in I\setminus\st{i}$ satisfy $j\not=j'$.
Since $(\SW_j,\SW_{j'})$ and $(\SW_{j'},\SW_j)$ are unmixed, by Lemma \ref{lem:unmixed} we have
\eqn
\tLaf(\hSW_j,\hSW_{j'})
\equiv\tLaf(\hSW_{j'},\hSW_j)
\equiv1.
\eneqn
By Lemma \ref{lem:relaff} we obtain
\eqn
\Daf(\hSW_j,\hSW_{j'})&&\equiv\tLaf(\hSW_j,\hSW_{j'})\cdot\tLaf(\hSW_{j'},\hSW_{j})
\cdot\wtaf(\hSW_j,\hSW_{j'})^{-1}\equiv\wtaf(\hSW_j,\hSW_{j'})^{-1}.
\eneqn
By \eqref{eq:chiK}
we have
$$
\wtaf(\hSW_j,\hSW_{j'})
=\dfrac{\left(\chi_i(\hSW_j)\res[t_i]\chi_i(\hSW_{j'}) \right)^{2 }}{A},$$
where
\eqn
A=\bl\chi_i(\hSW_j)\res[t_i]Q_{i,j'}(t_i,t_{j'})\res[{t_{j'}}]\chi_{j'}(\hSW_{j'})\br&&\cdot
\bl\chi_j(\hSW_j)\res[t_j]Q_{i,j}(t_i,t_j)\res[{t_{i}}]\chi_{i}(\hSW_{j'})\br\\
&&\hs{3ex}\cdot\bl\chi_j(\hSW_j)\res[t_j]Q_{j,j'}(t_j,t_{j'})\res[{t_{j'}}]\chi_{j'}(\hSW_{j'})\br.
\eneqn
On the other hand, we have
\eqn
Q_{i,j}(t_i,z_j)\res[t_i]Q_{i,j'}(t_i,z_{j'})&&=
\chi_i(\hSW_j)\res[t_i]\chi_i(\hSW_{j'})\\
&&=\chi_i(\hSW_j)\res[t_i]Q_{i,j'}(t_i,t_{j'})\res[{t_{j'}}]\chi_{j'}(\hSW_{j'})\\
&&=\chi_j(\hSW_j)\res[t_j]Q_{i,j}(t_i,t_j)\res[{t_{i}}]\chi_{i}(\hSW_{j'})
\eneqn
and
$$\chi_j(\hSW_j)\res[t_j]Q_{j,j'}(t_j,t_{j'})\res[{t_{j'}}]\chi_{j'}(\hSW_{j'})
=Q_{j,j'}(z_j,z_{j'}).$$
Hence we have
$$\wtaf(\hSW_j,\hSW_{j'})=Q_{j,j'}(z_j,z_{j'})^{-1}$$
and
$$\Daf(\hSW_j,\hSW_{j'})=Q_{j,j'}(z_j,z_{j'}).$$

\mnoi
(ii) It is obvious 
$\Daf(\hSW_i,\hSW_{j})\equiv1$ for $j\in I$ such that 
$(\al_i,\al_j)=0$, since $\de(\SW_i,\SW_{j})=0$.

\mnoi
(iii) Let us show
$\Daf(\hSW_i,\hSW_{j})\equiv Q_{i,j}(z_i,z_j)$ for $j\in I$ such that $
\ang{h_i,\al_j}=-1$.

In this case, we have $c\seteq\sfd_i/\sfd_j=-\ang{h_j,\al_i}\in\Z_{\ge1}$ and
we may assume that $Q_{i,j}(z_i,z_j)=z_i-z_j^c$.
Set $f(z_i,z_j)\seteq\Daf(\hSW_i,\hSW_{j})$.
Then we have $\deg f(z_i,z_j)=\deg Q_{i,j}(z_i,z_j)$ by Proposition \ref{prop:LaKK}.

Note that $\hSW_j \simeq \ang{i}_{z_j^c}\tens_{\cor[z_j]}\ang{j}_{z_j} $ as a $\cor[z_j]$-module on which $\tau_1$ acts by zero.  Hence it is a quotient of  $\ang{i}_{z_j^c}\conv[z_j] \ang{j}_{z_j} $, so that 
by Lemma \ref{lem:CVB},
 we have $\hSW_j\simeq \ang{i}_{z_j^c}\hconv_{z_j}\ang{j}_{z_j}
\monoto \ang{j}_{z_j}\conv[z_j]\ang{i}_{z_j^c}$.

We have
$$\hSW_j\conv[z_j]\bl\hSW_i\vert_{z_i=z_j^c}\br\monoto
\ang{j}_{z_j}\conv[z_j]\ang{i}_{z_j^c}
\conv[z_j]\bl\hSW_i\vert_{z_i=z_j^c}\br\epito\ang{j}_{z_j},$$
which implies an epimorphism
\eq \label{eq:KjKitoj}
\hSW_j\conv[z_j]\bl\hSW_i\vert_{z_i=z_j^c}\br\epito \ang{j}_{z_j},
\eneq
since it is so after operating $(\cor[z_j]/z_j\cor[z_j])\tens[{\cor[z_j]}]\scbul$. 

Hence Proposition~\ref{prop:defactor} implies that
$z_j-w$ divides $f(w^c,z_j)$, i.e., 
$f(w^c,w)=0$.  
Hence $Q_{i,j}(z_i,z_j)=z_i-z_j^c$ divides $f(z_i,z_j)$. 
Comparing their degrees, we have
$f(z_i,z_j)=Q_{i,j}(z_i,z_j)$ up to a constant multiple.

\mnoi
(iv) Let us show
$\Daf(\hSW_i,\hSW_{j})\equiv Q_{i,j}(z_i,z_j)$ for $j\in I$ such that $
-\ang{h_i,\al_j}>1$.

In this case, we have $c\seteq\sfd_j/\sfd_i=-\ang{h_i,\al_j}\in\Z_{>1}$ and
we may assume that $Q_{i,j}(z_i,z_j)=z_j-z_i^c$.
Set $f(z_i,z_j)\seteq\Daf(\hSW_i,\hSW_{j})$.
Then we have $\deg f(z_i,z_j)=\deg Q_{i,j}(z_i,z_j)$ 
by Proposition \ref{prop:LaKK}.

Set $\Ma=\hSW_j\vert_{z_j=w^c}$ with $\deg(w)=\sfd_i$.
 By replacing $\cor$ with its algebraic closure, we may assume that $\cor$ is algebraically closed. 
Let $\omega$ be the $c$-th primitive root of unity.
Then we have
$\Ma=\Laa\hconv_w\ang{j}_{w^c}$,
where 
\eq \hskip 3em \label{eq:LL'}
\Laa=\ang{i}_w\conv[w]\ang{i}_{\omega w}\conv[w]\cdots
\conv[w]\ang{i}_{\omega^{c-1} w}\simeq \ang{i}_w\conv[w]\Laa'
\qt{and} \quad
\Laa'=\ang{i}_{\omega w}\conv[w]\cdots
\conv[w]\ang{i}_{\omega^{c-1} w}. \eneq
Then we have
$\Ma=\ang{i}_{w}\hconv_{w}\bl\Laa'\hconv_w\ang{j}_{w^c}\br
\monoto \bl\Laa'\hconv_w\ang{j}_{w^c}\br\conv[w]\ang{i}_{w}$.

Hence we have
$$\Ma\conv[w]\bl\hSW_i\vert_{z_i=w}\br\monoto
\bl\Laa'\hconv_{w}\ang{j}_{w^c}\br\conv[w]\ang{i}_{w}\conv[w]
\bl\hSW_i\vert_{z_i=w}\br\epito\Laa'\hconv_{w}\ang{j}_{w^c}.$$
Then the composition is an epimorphism
\eq \label{eq:KjKiLpj}
\bl\hSW_j\vert_{z_j=w^c}\br\conv[w]\bl\hSW_i\vert_{z_i=w}\br\epito \Laa'\hconv_w\ang{j}_{w^c},\eneq
since it is so after operating $(\cor[w]/w\cor[w])\tens[{\cor[w]}]\scbul$.

 By  Lemma~\ref{lem:ij}, $\Laa'\hconv_w\ang{j}_{w^c}$ is an affinization of $\ang{i^{c-1}}\hconv\ang{j}$ which is real simple by Lemma~\ref{lem:2det} below. 
Hence, Proposition~\ref{prop:defactor} implies that 
$Q_{i,j}(z_i,z_j)=z_j-z_i^c$ divides $f(z_i,z_j)$.
Comparing their degrees, we have
$f(z_i,z_j)=Q_{i,j}(z_i,z_j)$.
\QED

\bigskip
Thus we have  proved that
the datum
$\st{\hSW_j}_{j\in I}$ in $\tcatC^*_{s_iw_0}$
satisfies condition \eqref{eq:SW}.
Choosing  a duality datum
$\bl\st{(\hSW_j,z_j)}_{j\in I},\st{\Rre_{\hSW_j,\hSW_k}}_{j,k\in I}\br$,
we can define  exact monoidal functors
$$\hF_i\cl \Modc(R)\to \Pro(\tcatC^*_{s_iw_0}) \qt{and}\quad
\F_i\cl R\gmod\to \tcatC^*_{s_iw_0}$$
sending
$\ang{j}_{z_j}$ to $\hSW_j$, except the $A_2$-cases. 
See Remark~\ref{rem:A2} for the  $A_2$-cases. 

Note that 
\eq&&\wt(\hF_i(X))=s_i\bl\wt(X)\br\qt{for any $X\in \Modc(R)$.}
\eneq

\Lemma \label{lem:Frneq0}
Let $M,N \in R\gmod$  be simple modules.
Assume that $\F_i(M)$ and $\F_i(N)$ are simples in $\tcatC^*_{s_iw_0}$.
Assume moreover that either $M$ and $\F_i(M)$ are \afr or
$N$ and $\F_i(N)$ are \afr.
\bnum
\item We have $\La(\F_i(M),\F_i(N)) \le \La(M,N)$.
\item The following conditions are equivalent.
\bna
\item $\La(\F_i(M),\F_i(N)) = \La(M,N)$,
\item $\F_i(M\hconv N)\neq 0$. \label{it:nonvan}
\end{enumerate}
\item Moreover, if condition \eqref{it:nonvan} holds, then $\F_i(\rmat{M,N}) \neq 0$ and 
$$ \F_i(M\hconv N) \simeq \F_i(M)\hconv \F_i(N).$$
\end{enumerate}
\enlemma
\begin{proof}
Assume that $M$ and $\F_i(M)$ are \afr.
Let $(\Ma,z)$ be an  affinization of $M$.
Since 
$$\Rre_{N,\Ma} \circ \Rre_{\Ma, N} =z^d \;\id_{\Ma\tens N},$$
we have a commutative diagram
\eqn
\xymatrix@C=4em{
\hF_i(\Ma) \conv \hF_i(N) \ar[r]_{\hF_i(\Rre_{\Ma, N} )} \ar@/^2em/[rr]^{{\hF_i(z)^d}} & \hF_i(N)  \conv \hF_i(\Ma) \ar[r]_{\hF_i(\Rre_{N,\Ma})} & \hF_i(\Ma) \conv \hF_i(N). 
}
\eneqn
Since $\hF_i(z)$ is a monomorphism,  
$\hF_i(\Rre_{\Ma_z, N} )$ and $\hF_i(\Rre_{N,\Ma_z})$ are non-zero.
Note also that $\dim\HOM\bl\F_i(M)\conv\F_i(N),\F_i(N)\conv\F_i(M)\br=1$
by Proposition~\ref{prop:dim1}. 
It follows that 
$$\hF_i(\Rre_{\Ma, N} ) = z^a\Rre_{\hF_i(\Ma),\hF_i( N)} \qt{and} \quad \hF_i(\Rre_{ N,\Ma} ) = z^b\Rre_{\hF_i( N), \hF_i(\Ma)} $$
for some $a,b \in \Z_{\ge0}$ up to constant multiples by Proposition \ref{pro:onedimhom} (ii). 

Note also
that $\hF_i(\Ma)$ is an \afn of $\F_i(M)$.
Hence we have
$$\La(M,N) =\deg(\Rre_{\Ma,N})= \deg(\hF_i(\Rre_{\Ma, N} )) \ge \deg(\Rre_{\hF_i(\Ma),\hF_i( N)})=\La(\hF_i(M),\hF_i(N)),$$
and the equality $\deg(\hF_i(\Rre_{\Ma, N} )) = \deg(\Rre_{\hF_i(\Ma),\hF_i( N)})$  holds if and only if $\deg(z^a)=0$; i.e.,  $a=0$.
In the case,  we have
$$\F_i(\rmat{M,N}) = \hF_i(\Rre_{\Ma, N}|_{z=0} ) =   (\hF_i(\Rre_{\Ma, N} ))|_{\hF_i(z)=0} = \rmat{\F_i(M),\F_i(N)} \neq 0$$
and hence $$M\hconv N\simeq  \Im \F_i(\rmat{M,N})  = \Im \rmat{\F_i(M),\F_i(N)} \simeq \F_i(M)\hconv \F_i(N),$$
which completes the proof.
\end{proof}

\Lemma\label{lem:2det}
Let $m,n\in\Z_{\ge0}$ and $j\in I\setminus\st{i}$,
and let $S=\ang{j^n}\hconv \ang{i^m}$. Set $\la=n\La_j-m\La_i$.
Then we have \ro neglecting the degree shifts\rf{\rm:}
\bnum
\item $s_j\la \preceq \la \preceq s_i\la$  and $S\simeq\dM(s_j\la,s_i\la)$,\label{ite:1}
\item $\eps_i(S)=\max\bl0,m+n\ang{h_i, \al_j}\br$, 
\item if $m\ge-n \ang{h_i, \al_j}$, then $\ang{i}$ commutes with
$S$ and
$\ang{j^n}\hconv \ang{i^m}\simeq \bl\ang{j^n}\hconv \ang{i^{-n\ang{h_i,\al_j}}}\br\conv\ang{i
^{m+n \ang{h_i, \al_j}}}$,
\item if $\ang{h_i,s_j\la}=-n \ang{h_i, \al_j}-m\ge0$, then
  $\F_i(S)\simeq \dM(s_is_j\la ,\la)\simeq
  \ang{i^{-m-n\ang{h_i,\al_j}}}\hconv \ang{j^n}$.
\ee
\enlemma
\Proof
(i)\ We have
$\ang{h_i,s_i\la}=m$, $\ang{h_j,\la}=n$ and
$s_j\la=(s_js_i)(s_i\la)$.
Hence we have $S\simeq\dM(s_j\la,s_i\la)$.

\mnoi
(ii) follows from 
\cite[Proposition 2.16 (iii)]{KKOP22}.

\mnoi
(iii) follows from the fact that $\psi_*\bl\dM(s_is_jn\La_j,n\La_j)\br
\simeq \ang{j^n}\hconv \ang{i^{-n\ang{h_i,\al_j}}}$ commutes with $\ang{i}$
by \cite[Proposition 2.16 (ii)]{KKOP22}.

\mnoi
(iv)\ 
We argue by induction on $m$.
If $m=0$, then
\eqn\F_i(S)=\F_i(\ang{j}^{\circ n})
&&\simeq\dM(s_is_j\La_j,\La_j)^{\circ n}\\
&&\simeq\dM(s_is_jn\La_j,n\La_j)\simeq\ang{i^{-n\ang{h_i,\al_j}}}\hconv \ang{j^n}.
\eneqn

Assume that $m>0$. Set $c=-\ang{h_i,\al_j}$. Then we have
\eqn
\F_i(S)&&\simeq \F_i\bl(\ang{j^n}\hconv\ang{i^{m-1}})\hconv\ang{i}\br\\
&&\underset{*}{\simeq}\F_i(\ang{j^n}\hconv\ang{i^{m-1}})\hconv\F_i(\ang{i})\\
&&\simeq\bl\ang{i^{1-m+nc}}\hconv\ang{j^n}\br\hconv\D\ang{i}\\
&&\simeq\bl\ang{i}\hconv(\ang{i^{ -m+nc}}\hconv\ang{j^n})\br\hconv\D\ang{i}\\
&&\simeq \ang{i^{-m+nc}}\hconv\ang{j^n}.
\eneqn
Here the second isomorphism $\underset{*}{\simeq}$ follows from Lemma \ref{lem:Frneq0} and the fact that
\eqn\La\bl\ang{j^n}\hconv\ang{i^{m-1}},\ang{i}\br
&&=\La\bl\ang{j^n},\ang{i} \br=-n(\al_j,\al_i) 
\eneqn 
is equal to
\eqn\La\bl\F_i(\ang{j^n}\hconv\ang{i^{m-1}}),\F_i(\ang{i})\br
&&=\La\bl \ang{i^{1-m+nc}}\hconv\ang{j^n},\D\ang{i}\br
\underset{(1)}{=}\La\bl\ang{i},\ang{i^{1-m+nc}}\hconv\ang{j^n}\br\\
&&=\La(\ang{i},\ang{j^n})=-n(\al_i,\al_j),
\eneqn
where $\underset{(1)}{=}$ follows from Lemma~\ref{lem:LaMDN}.
\QED

\Lemma 
For $w\in\weyl$ such that $s_iw>w$, we have
$$\F_i\bl\dM(w\La,\La)\br\simeq\dM(s_iw\La,s_i\La)\qt{for any $\La\in\pwtl$.}$$
In particular, we have
$$\F_i\bl\dC_{\La} \br\simeq \dC^*_{\La} \qt{for any $\La\in\pwtl$.}$$
\enlemma
\Proof
Let us argue by induction on $\ell(w)$.
If $\ell(w)\le1$, it is obvious.
Assume that $\ell(w)>1$.
Take $j\in I\setminus\st{i}$ such that
$s_jw<w$.

\mnoi
(i) Assume first that $s_is_jw>s_jw$.
Set $v=s_jw$ and $n=\ang{h_j,v\La}\in\Z_{\ge0}$ and
$m=\ang{h_i,v\La}\in\Z_{\ge0}$.
Then we have $w\La=v\La-n\al_j$ 
and hence
$ \ang{h_i,w\La}=m-n\ang{h_i,\al_j}\in\Z_{\ge0}$.
Then we have
$\dM(w\La,v\La)\simeq \ang{j^n}$ and  $\dM(s_iw\La,v\La)\simeq 
\ang{i^{m-n\ang{h_i,\al_j}}}\hconv\ang{j^n}$.
We have
\eqn
\dM(s_iw\La,s_iv\La)\hconv\dM(s_iv\La,v\La)&&\simeq
\dM(s_iw\La,v\La  ) \\
&&\simeq
\ang{i^{m-n\ang{h_i,\al_j}}}\hconv\ang{j^n}\underset{*}{\simeq} \bl \ang{i^{-n\ang{h_i,\al_j}}} \hconv \ang{j^n} \br\conv \ang{i^m}\\
&&\simeq \bl \ang{i^{-\lan h_i,\al_j \ran} }\hconv \ang{j} \br^{\circ n}\conv \dM(s_iv\La,v\La),
\eneqn
where $\underset{*}{\simeq}$ follows from Lemma~\ref{lem:2det}.
Hence we have 
$\dM(s_iw\La,s_iv\La)\simeq \bl \ang{i^{-\lan h_i,\al_j \ran} }\hconv \ang{j} \br^{\circ n}$, which implies that
$$\F_i\bl\dM(w\La,v\La)\br\simeq\dM(s_iw\La,s_iv\La).$$
On the other hand, \cite[Proposition 4.6]{KKOP18} implies that
both $\La\bl \dM(w\La,v\La),\dM(v\La,\La)\br$ and
$\La\bl \dM(s_iw\La,s_iv\La),\dM(s_iv\La,s_i\La)\br$ are equal to
$-(w\La-v\La, v\La-\La)$.
Therefore, by Lemma~\ref{lem:Frneq0}, we obtain
\eqn
\F_i\bl \dM(w\La,\La)\br&&\simeq \F_i\bl\dM(w\La,v\La)\hconv\dM(v\La,\La)\br\\
&&\simeq \F_i\bl\dM(w\La,v\La)\br\hconv\F_i\bl\dM(v\La,\La)\br\\
&&\simeq \dM(s_iw\La,s_iv\La)\hconv\dM(s_iv\La,s_i\La)\\
&&\simeq \dM(s_iw\La,s_i\La).
\eneqn

\mnoi
(ii) Assume that $s_is_jw<s_jw$.
Set $v=s_is_jw$.
Then we have $s_iv>v$, $ \ell(v) < \ell(w)$ and hence by induction hypothesis, 
$\F_i\bl\dM(v\La,\La)\br\simeq\dM(s_iv\La,s_i\La)$.
Note that 
$\dM(w\La,v\La)\simeq \ang{j^n}\hconv \ang{ i^m}$ where $n=\ang{h_j,s_iv\La}$,
$m=\ang{h_i,v\La}$. 

Since
$\La(\ang{j^n},\ang{ i^m})=-nm(\al_j,\al_i)$ and
$$\La\bl(\ang{i^{-\ang{h_i,\al_j}}}\hconv \ang{j})^{\circ n},\D(\ang{ i^m})\br
=\La\bl\ang{i^m},(\ang{i^{-\ang{h_i,\al_j}}}\hconv\ang{j})^{\circ n}\br
=-nm(\al_j,\al_i)$$
are equal, we have
\eqn
\F_i(\dM(w\La,v\La))&&\simeq
\bl \ang{i^{-\ang{h_i,\al_j}}}\hconv\ang{j}\br^{\circ n}
\hconv\D\ang{i^m}\\
&&\simeq \Bigl(\ang{i^m}\hconv \bl\ang{i^{-m-n\ang{h_i,\al_j}}}\hconv\ang{j^n}\br\Bigr)
\hconv\D\ang{i^m}\\
&&\simeq \ang{i^{-m-n\ang{h_i,\al_j}}}\hconv\ang{j^n}
\simeq\dM(s_iw\la,s_iv\La).\eneqn

Note that $\ang{h_i,w\La}=-m-n\ang{h_i,\al_j}\ge 0$.

Since $\La\bl \dM(w\La,v\La),\dM(v\La,\La)\br$ and $
\La\bl \dM(s_iw\La,s_iv\La),\dM(s_iv\La,s_i\la)\br$
are equal to $-(w\La-v\La, v\La-\La)$,
we have 
\eqn
\F_i\bl \dM(w\La,\La) \br&&\simeq \F_i\bl\dM(w\La,v\La)\hconv\dM(v\La,\La)\br\\
&&\simeq \F_i\bl\dM(w\La,v\La)\br\hconv\F_i\bl\dM(v\La,\La)\br\\
&&\simeq \dM(s_iw\La,s_iv\La)\hconv\dM(s_iv\La,s_i\La) \\ 
&&\simeq \dM(s_iw\La,s_i\La).
\eneqn
\QED

\Lemma \label{lem:FiM}
Let $w,v \in\weyl$ such that $v\le w$ and $s_iv>v$,  $s_iw>w$.
Then, we have
$$\F_i\bl\dM(w\La,v\La)\br\simeq\dM(s_iw\La,s_iv\La)\qt{for any $\La\in\pwtl$.}$$
\enlemma
\Proof
By the preceding lemma, we have
\eqn
\dM(s_iw\La,s_iv\La)\hconv \dM(s_i v\La,s_i\La)&&\simeq
\dM(s_iw\La,s_i\La)\\
&&\simeq\F_i\bl\dM(w\La,\La)\br
\simeq \F_i\bl\dM(w\La,v\La)\hconv \dM(v\La, \La)\br\\
&&\underset{*}{\simeq}\F_i\bl\dM(w\La,v\La)\br\hconv\F_i\bl\dM(v\La,  \La \br\\
&&\simeq \F_i\bl\dM(w\La,v\La)\br\hconv \dM(s_iv\La,s_i\La),
\eneqn
where $\underset{*}{\simeq}$ holds by Lemma \ref{lem:Frneq0}. 
Hence we have
$$\F_i\bl\dM(w\La,v\La)\br\simeq \dM(s_iw\La,s_iv\La).$$
\QED

\Prop \label{prop:simpletosimple}
If $M$ is a simple module in  $\catC_{s_iw_0}$, then $\F_i(M)$  is a simple module in $\catC^*_{s_iw_0}$. Moreover
the functor $\F_i$ induces a bijection between the set of classes of simple modules in $\catC_{s_iw_0}$ and that of $\catC^*_{s_iw_0}$.
\enprop

\Proof
Let $\underline{w_0} = s_{i_1} s_{i_2}\cdots s_{i_{l-1}} s_{i_l}$ be a reduced expression of $w_0$ with $i_1=i$.
Set
\eq \label{eq:cuspidals}
&&\ba{l}V_k\seteq \dM(s_{i_2}\cdots s_{i_k}  \La_{i_k}, s_{i_2}\cdots s_{i_{k-1}}  \La_{i_k} ) \qt{for} \quad 2 \le k \le l \qt{and} \\
S_k\seteq \dM(s_{i_1}s_{i_2}\cdots s_{i_k}  \La_{i_k}, s_{i_1}s_{i_2}\cdots s_{i_{k-1}}  \La_{i_k} ) \qt{for} \quad 1 \le k  \le l.\ea
\eneq 
Then by Lemma \ref{lem:FiM}, we have 
$$\F_i(V_k)=S_k \qt{for } \quad 2 \le k \le l.$$
Note that  $K(\catC_{s_iw_0})\vert_{q=1}$ is the polynomial ring $\Z[\;[V_2],\ldots, [V_l]\;]$ and 
$K(\catC_{w_0})\vert_{q=1}$ is the polynomial ring  $\Z[\;  [S_1], [S_2],\ldots, [S_l]\;]$,
  where $K(\catC_{w_0})$ denotes the Grothendieck ring  of
the abelian monoidal category $\catC_{w_0}$ and 
$K(\catC_{w_0})\vert_{q=1}\seteq K(\catC_{w_0})/(q-1)K(\catC_{w_0})$.

Recall that 
for any simple $M$ in $\catC_{w_0}=R\gmod$, there exists a unique $(a_1,\ldots, a_l) \in \Z^{l}_{\ge 0}$ such that 
$S_l^{\circ a_l} \conv S_{l-1}^{\circ a_{l-1}} \conv \cdots \conv  S_1^{\circ a_1}$
has a simple head isomorphic to $M$ up to a grading shift.
Since $\Es{i}(S_k)=0$ for $2\le k\le l$ and $S_1=L(i)$, the simple module $M$ belongs to $\catC^*_{s_iw_0}$ if and only if $a_1=0$. 
Hence the ring $K(\catC^*_{s_iw_0})|_{q=1}$ is the polynomial ring $\Z[\;[S_2],\ldots, [S_l]\;]$. 
 
It follows that  
the ring homomorphism $[\F_i] \cl K(\catC_{s_iw_o})\vert_{q=1}\to K(\tcatC^*_{s_iw_0})\vert_{q=1}$ is an isomorphism.
In particular, for any non-zero module $M$, $\F_i(M)$ is non-zero, and
$[\F_i]$ induces a bijection between the set of simple modules in $\catC_{s_iw_0}$ and that of $\catC^*_{s_iw_0}$. \QED

For a simple $R(\beta)$-module $M$,  set
\eqn
&& \wt(M) = - \beta, \\
&&\tF_i(M) = L(i)\hconv M, \qquad \tE_i(M) = \hd(\E{i}(M)),  \\
&& \ep_i(M) = \max \{ k \ge 0 \mid \E{i}^k M \not\simeq 0 \}, \ \  \ph_i(M) = \ep_i(M) +  \langle h_i, \wt(M)  \rangle,\\
&&\tF^*_i(M) = M \hconv L(i), \qquad \tE^*_i(M) = \hd(\Es{i}(M)),  \\
&& \ep^*_i(M) = \max \{ k \ge 0 \mid \Es{i}^k M \not\simeq 0 \}, \ \  \ph^*_i(M) = \ep^*_i(M) + \langle h_i, \wt(M)  \rangle.
\eneqn
Here we denote by $M$  its isomorphism class (as ungraded module) for simplicity.
Then the set of isomorphism classes of self-dual simple modules in $R\gmod$ together with  $(\wt, \tE_i,\tF_i,\ep_i,\ph_i)$  forms a $\g$-crystal which is isomorphic to the crystal basis $B(\infty)$ of $\Uqm$ (\cite{LV11}).

On the other hand, Yoshihisa Saito (\cite{Saito}) defined the isomorphism
$$\st{b\in B(\infty)\mid\eps_i(b)=0}
\isoto \st{b\in B(\infty)\mid\eps^*_i(b)=0}$$
given by $b\mapsto {\tf_i}^{  \ph_i^*(b)} {\te_i}^{\hskip 0.1em * \hskip 0.1em \ep_i^*(b)}b$.
Hence, to a simple $M\in R\gmod$ with $\eps_i(M)=0$,
we can associate a simple module
$$\sigma_i(M)\seteq{\tF_i}^{  \ph_i^*(M)} {\tE_i}^{\hskip 0.1em * \hskip 0.1em \ep_i^*(M)}M $$
which satisfies $\eps^*_i\bl\sigma_i(M)\br=0$.
We call it the \emph{Saito reflection of $M$ with respect to $i$}
(cf.\ \cite{Saito,Kato14,Kato20}).

\Prop \label{Prop: cat of saito refl}
Let $M$ be a simple module in $\catC_{s_iw_0}$. Then
$$\F_i(M) \simeq \sigma_i(M)$$
as an ungraded module. 
\enprop
\begin{proof}
By \cite[Corollary 2.26]{TW16}, it is enough to show that $\F_i(V_k) \simeq S_k\simeq \sigma_i(V_k)$ for $2\le k\le l$, where $V_k$ and $S_k$ are the modules in \eqref{eq:cuspidals}.

Let $M$ be a simple module in $R\gmod$ with $\eps_i(M)=0$.
Then, we have
\eqn
\sigma_i(M) ={\tF_i}^{  \ph_i^*(M)} {\tE_i}^{\hskip 0.1em * \hskip 0.1em \ep_i^*(M)}M 
\simeq {\tE_i}^{\hskip 0.1em * \hskip 0.1em s-\ang{h_i,\wt(M)} }  {\tF_i}^{ s} M 
\eneqn
for any $s\ge \vphi_i^*(M)$ by  \cite[Proposition 2.16, Theorem 2.17]{KKOP22}.
Let $2\le k\le l$ and $\la=s_{i_2}\cdots s_{i_{k-1}}\La_{i_k}$, $\mu=s_{i_2}\cdots s_{i_k}\La_{i_k}$ so that $V_k=\dM(\la,\mu)$. Then  $\ang{h_i,\la}\ge 0$, $\ang{h_i,\mu}\ge 0$ and
\eqn 
 \vphi_i^*(V_k) = \eps_i^*(V_k)+\ang{h_i,\wt(V_k)} && \le \ang{h_{i},\mu}  + \ang{h_{i},\la-\mu} =\ang{h_{i},\la},
\eneqn
where the middle inequality follows from \cite[Lemma 9.1.5]{KKKO18}.

Hence  by taking $s=\ang{h_{i},\la} $,  we have
\eqn
\sigma_i(V_k)
\simeq {\tE_i}^{\hskip 0.1em * \hskip 0.1em s-\ang{h_i,\wt(V_k)} }  {\tF_i}^{ s}   \dM(\la,\mu)&&= 
{\tE_i}^{\hskip 0.1em * \hskip 0.1em \ang{h_i,\mu} }  {\tF_i}^{ \ang{h_{i},\la}}   \dM(\la,\mu) \\
&&\simeq {\tE_i}^{\hskip 0.1em * \hskip 0.1em \ang{h_i,\mu} }\dM(s_i\la,\mu)
\simeq
\dM(s_i\la,s_i\mu) \simeq S_k,
\eneqn
as desired.
\end{proof}

\Lemma
\label{lem:LaMLj}
  For any $j\in I$, we have
$$\La(\dM(s_iw_0\La,\La),\ang{j})=
\La(\dM(w_0\La,s_i\La),\SW_j).$$
\enlemma
\Proof
We have $\La(\dM(s_iw_0\La,\La),\ang{j})=(s_iw_0\La+\La,\al_j)$  by \cite[(5.1)]{KKOP21} 
and $\La(\dM(w_0\La,s_i\La),\SW_j)=(w_0\La+s_i\La, s_i\al_j)$  by \cite[Proposition 3.27]{KKOP21}. 
\QED
\Prop \label{Prop: factor th}
The functor $\F_i\cl R\gmod\to \tcatC^*_{s_iw_0}$
factors through $\tcatC_{s_iw_0}$:
$$\xymatrix@C=8ex{
R\gmod\ar[d]_{\Qt_{s_iw_0}}\ar[dr]^{\F_i}\\
\tcatC_{s_iw_0}\ar[r]_{\refl_i}& \tcatC^*_{s_iw_0}.
}$$
\enprop
\begin{proof}
By Lemma \ref{lem:FiM}, we have $\F_i(\Ctr_\La) \simeq \Ctr^*_\La$ for any $\La \in \pwtl$. 
 By the universal property of localization (\cite[Theorem 2.7]{KKOP21}), it remains to show that $\F_i(\coR_{\Ctr_\La}(X))\cl \F_i(\Ctr_\La) \conv \F_i(X) \to \F_i(X) \conv \F_i(\Ctr_\La)$ is an isomorphism for any $X\in R\gmod$ and $\La \in \pwtl$. 

Note that 
\bna
\item if $\F_i(\coR_{\Ctr_\La}(M))$ and $\F_i(\coR_{\Ctr_\La}(N))$ are isomorphisms for $M,N\in R\gmod$, then so is $\F_i(\coR_{\Ctr_\La}(M\conv N))$.
 \item for a morphism $f\cl M\to N$, if $\F_i(\coR_{\Ctr_\La}(M))$, $\F_i(\coR_{\Ctr_\La}(N))$ are isomorphisms, then  so is
   $\F_i(\coR_{\Ctr_\La}(\Coker(f)))$,
 \item for an exact sequence $0\to L\to M\to N\to 0$, if
   $\F_i(\coR_{\Ctr_\La}(L))$ and $\F_i(\coR_{\Ctr_\La}(N))$
   are isomorphism, then so is $\F_i(\coR_{\Ctr_\La}(M))$.
   \ee 
Thus it is enough to show that $\F_i(\coR_{\Ctr_\La}(\ang{j}))$ is an isomorphism for any $j\in I$
and it follows from Lemma \ref{lem:Frneq0} and Lemma \ref{lem:LaMLj}.
\end{proof}

\Th \label{Thm: relf equi}
The monoidal functor $\refl_i\cl \tcatC_{s_iw_0}\to\tcatC^*_{s_iw_0}$
is an equivalence of categories.
\enth

\Proof
We can define 
$\F_i^*\cl R\gmod\to\tcatC_{s_iw_0}$ by the Schur-Weyl datum
$\st{\hSW^*_j}_{j\in I}$ where
\eqn
\SW_j^*=\psi_*(\SW_j)\simeq
\bc
\D^{-1}\bl \Qt_{s_iw_0}(\ang{i})\br\simeq \Ctr_i^{-1}\conv (\Es{i}\Ctr_{i})
&\text{if $j=i$,}\\
\ang{j}\hconv \ang{i^{-\ang{h_i,\al_j}}}&\text{if $j\not=i$,}
\ec
\eneqn
and
\eqn
\hSW_j^*=\psi_*(\hSW_j)\simeq
\bc
\Da^{-1}\bl\Qt_{s_iw_0}(\ang{i}_{z_i})\br
&\text{if $j=i$,}\\
\bl\ang{j}\hconv\ang{i^{-\ang{h_i,\al_j}}}\br_{z_j}&\text{if $j\not=i$.}
\ec
\eneqn
Here the functor $\Qt_{s_iw_0}\cl\Pro(R\gmod) \to \Pro(\tcatC_{s_iw_0})$ is induced by $\Qt_{s_iw_0}\cl R\gmod \to \tcatC_{s_iw_0}$.
Then,
the functor $\F_i^*\cl R\gmod\to \tcatC_{s_iw_0}$
factors through $\tcatC^*_{s_iw_0}$:
$$\xymatrix@C=8ex{
R\gmod\ar[d]_{\Qt^*_{s_iw_0}}\ar[dr]^{\F^*_i}\\
\tcatC^*_{s_iw_0}\ar[r]_{\refl^*_i}& \tcatC_{s_iw_0}.
}$$
Note that $\F_i^*$ is isomorphic to the composition
$$R\gmod\To[\psi_*]R\gmod\To[\F_i]\tcatC^*_{s_iw_0}\To[\psi_*]\tcatC_{s_iw_0}.$$
 We claim that the composition
$$\Modgc(R)\To \Pro(R\gmod)\To[\Pro(\Qt^*_{s_iw_0})]  \Pro(\tcatC_{s_iw_0})\To[\Pro(\refl^*_i)]
 \Pro(\tcatC_{s_iw_0}) \To[\Pro(\refl_i)] \Pro(\tcatC^*_{s_iw_0})
$$
is isomorphic to the functor associated with the duality datum 
 $$(\{ \Qt^*_{s_iw_0}(\lan j \ran_{z_j}) \}_{j\in I},\ \{ \Rre_{\Qt^*_{s_iw_0}(\lan j \ran_{z_j}),\Qt^*_{s_iw_0}(\lan k \ran_{z_k}) } \}_{j,k\in I})$$ 
 which is of  type $\cartan$. 
Then  by restriction to $R\gmod$,  we have
$$
\refl_i\circ \refl^*_i \circ \Qt^*_{s_iw_0} \simeq \Qt^*_{s_iw_0},
$$
and by the universal property of the functor $\Qt^*_{s_iw_0}$,  we get 
$\refl_i\circ \refl^*_i\simeq \id_{\tcatC_{s_iw_0}}$, as desired.

By Lemma \ref{lem:uniqueness},  it is enough to show that 
$$
(\refl_i\circ \refl^*_i \circ \Qt^*_{s_iw_0})(\lan j\ran_{z_j}) 
 \simeq \Qt^*_{s_iw_0}(\lan j\ran_z) \quad \text{for all} \ j\in I.
$$

Since $(\refl_i\circ \refl^*_i \circ \Qt^*_{s_iw_0})(\lan j\ran_{z_j}) \simeq (\refl_i\circ\F^*_i)(\lan j\ran_{z_j}) \simeq
\refl_i (\hSW^*_j)$,
it is enough to show

\eq\refl_i(\hSW^*_j)\simeq  \Qt^*_{s_iw_0}(\lan j\ran_{z_j}) \quad \text{for all} \ j\in I \quad  \text{up to a grading shift.}  \label{eq:ij}
\eneq

If $j=i$,  then we have
\eqn
\refl_i (\hSW^*_i) = \refl_i (\Da^{-1} (\Qt_{s_iw_0} (\ang{i}_{z_i}) ) )
\simeq \Da^{-1} ( \refl_i (\Qt_{s_iw_0} (\ang{i}_{z_i}) ) ) \\
\simeq  \Da^{-1} ( \F_i (\ang{i}_{z_i})  )
\simeq  \Da^{-1} (\hSW_i  )  
\simeq  \Qt_{s_iw_0}^*(\ang{i}_{z_i}). 
\eneqn

When $\ang{h_i,\al_j}=0$, \eqref{eq:ij} is trivial.

If $\ang{h_i,\al_j}=-1$,  then we have
\eqn
\hSW^*_j\simeq \ang{j}_{z_j} \hconv_{z_j} \ang{i}_{z_j^c},
\eneqn 
where $c= -\ang{h_j,\al_i}\in \Z_{\ge 1}$.
Thus we have an epimorphism 
\eqn
\refl_i(\ang{j}_{z_j} \conv[z_j] \ang{i}_{z_j^c} ) \epito \refl_i(\hSW^*_j).
\eneqn
On the other hand,  we have
\eqn 
\refl_i(\ang{j}_{z_j} \conv[z_j] \ang{i}_{z_j^c} )  \simeq
\refl_i(\ang{j}_{z_j}) \conv[z_j] \refl_i(\ang{i}_{z_j^c})  \simeq
\hSW_j \conv[z_j] \hSW_i\vert_{z_i=z_j^c}   \epito \ang{j}_{z_j},
\eneqn
where the last epimorphism is the one in \eqref{eq:KjKitoj}.  Hence we have
$\refl_i(\hSW^*_j)\simeq \ang{j}_{z_j}$ by Proposition \ref{prop:simplehead}.

If  $\ang{h_i,\al_j}>1$,  then 
then we have
\eqn
\hSW^*_j \vert_{z_j=w^c} \simeq \ang{j}_{w^c} \hconv_{w} \Laa \simeq  \ang{j}_{w^c} \hconv_{w} (\ang{i}_w\conv[w]  \Laa'),
\eneqn 
where $c=-\ang{h_i,\al_j}$ and $\Laa, \Laa'$ are the ones in \eqref{eq:LL'}.
Hence there is an epimorphism
$$\refl_i(\ang{j}_{w^c}) \conv[w] \refl_i(\ang{i}_w) \conv[w] \refl_i(\Laa') \epito \refl(\hSW^*_j \vert_{z_j=w^c}). $$
On the other hand,  by \eqref{eq:KjKiLpj} we have
\eqn 
\refl_i(\ang{j}_{w^c}) \conv[w] \refl_i( \ang{i}_w) \conv[w] \refl_i(\Laa') 
\simeq (\hSW_j\vert_{z_j=w^c} \conv[w] \hSW_i\vert_{z_i=w})\conv[w] \refl_i(\Laa')  \\
\epiTo[\eqref{eq:KjKiLpj}] (\Laa' \hconv_w \ang{j}_{w^c}) \conv[w] \refl_i(\Laa') 
\epito (\Laa' \hconv_w \ang{j}_{w^c}) \hconv_w \refl_i(\Laa') 
\simeq \ang{j}_{w^c},
\eneqn
where the last isomorphism follows from that $\refl_i(\Laa')\simeq \Da(\Laa')$.
Hence we have
$\refl_i(\hSW^*_j)\simeq \ang{j}_{z_j}$ by Proposition \ref{prop:simplehead}, as desired. 
\QED

\Conj \label{Conj: refl} 
 The monoidal functor $\refl_i\cl \tcatC_{s_iw_0}\to\tcatC^*_{s_iw_0}$ induces an equivalence of categories
$\catC_{s_iw_0} \to \catC^*_{s_iw_0}$. 
\enconj
We know already that $\refl_i$ sends simples in $\catC_{s_iw_0}$
to simples in $ \catC^*_{s_iw_0}$.
However, we do not know $\refl_i(\catC_{s_iw_0})\subset\catC^*_{s_iw_0}$.
Remark that $\catC^*_{s_iw_0}$ is not stable by extensions in $\tcatC^*_{s_iw_0}$
(see Remark~\ref{rem:ext}).

\Rem\label{rem:A2}
In the $A_2$-case, we have $\de(\SW_j,\SW_k)=0\not=\de(\ang{j},\ang{k})$.
We define $\F_i$ as follows.
Take $I=\st{1,2}$, $Q_{1,2}(t_1.t_2)=t_1-t_2$ and $i=1$.
Then $s_iw_0=s_2s_1$.

The renormalized R-matrix
$\Rre_{j,k}\cl \ang{j}_{z_j}\conv \ang{k}_{z_k}\to \ang{k}_{z_k}\conv \ang{j}_{z_k}$ is given
by $\ang{j}_{z_j}\etens \ang{k}_{z_k}\mapsto \tau_1\bl \ang{k}_{z_k}\etens\ang{j}_{z_k}\br$ for $j\not=k$.
 Since $\de(K_1,K_2)=0$,
we can choose $\Rmat_{j,k}\cl \hSW_j\conv \hSW_k\to \hSW_k\conv \hSW_j$  ($j\not=k$)
such that
$\Rmat_{j,k}\circ\Rmat_{k,j}=\id_{\hSW_k\conv\hSW_j}$.
Then, 
we define $\F_1\cl R\gmod\to \tcatC_{s_iw_0}^* $ by:
$$\text{$\F_1(\ang{j} _{z_j})=\hSW_j$ ($j=1,2$),
$\F_1(\Rre_{1,2})=Q_{1,2}(z_1,z_2)\Rmat_{1,2}$
and $\F_1(\Rre_{2,1})=\Rmat_{2,1}$.}$$
Then it induces an equivalence
$\refl_1\cl\tcatC_{s_iw_0}\isoto\tcatC_{s_iw_0}^*$.

Set $\ang{12} := \ang{1}\hconv \ang{2}$ and  $\ang{21} := \ang{2}\hconv \ang{1}$. 

Note that $\Qt_{s_iw_0}(\ang{1})\simeq \ang{2}\conv\ang{21}^{-1}$, $\Qt_{s_iw_0}(\ang{12})\simeq0$.
The functor $\F_1$ sends:
\eqn
&&\begin{aligned}[t]
\ang{1}&\longmapsto\ang{2}\conv\ang{12}^{-1},\\
\ang{2}&\longmapsto\ang{12},
\end{aligned}
\hs{2ex}\begin{aligned}[t]
\ang{12}&\longmapsto0,\\
\ang{21}&\longmapsto\ang{2}.
\end{aligned}
\eneqn

\enrem

\section{Examples of $\F_i$}
Let us give examples of $\F_i$.
We ignore  grading shifts.
\subsection{$\F_i\cl R\gmod\to\tcatC_{s_iw_0}^*$}
\subsubsection{$A_2$}
$i=1$
\eqn
&&\begin{aligned}[t]
\ang{1}&\mapsto\ang{2}\conv\ang{12}^{-1}\\
\ang{2}&\mapsto\ang{12}
\end{aligned}
\hs{2ex}\begin{aligned}[t]
\ang{12}&\mapsto0\\
\ang{21}&\mapsto\ang{2}
\end{aligned}
\eneqn

\subsubsection{$A_3$}
\bnum
\item
$i=1$.
\eqn
&&\begin{aligned}[t]
\ang{1}&\mapsto\ang{23}\conv\ang{123}^{-1}\\
\ang{2}&\mapsto\ang{12}\\
\ang{3}&\mapsto\ang{3}
\end{aligned}
\hs{2ex}\begin{aligned}[t]
\ang{12}&\mapsto\ang{2312}\conv\ang{123}^{-1}\\
\ang{21}&\mapsto\ang{2}\\
\ang{23}&\mapsto\ang{123}\\
\ang{32}&\mapsto\ang{132}
\end{aligned}
\hs{2ex}
\begin{aligned}[t]
\ang{132}&\mapsto\ang{3}\conv\ang{2132}\conv\ang{123}^{-1}\\
\ang{213}&\mapsto\ang{23}\\
\ang{123}&\mapsto0\\
\ang{321}&\mapsto\ang{32}\\
\ang{2132}&\mapsto\ang{2132}
\end{aligned}\\
&&
\raisebox{2.3ex}[.5ex][.5ex]{$\ang{321}^{-1}\conv\ang{32}\mapsto \ang{1}$}
\eneqn

\item
$i=2$.
\eqn
&&\begin{aligned}[t]
\ang{1}&\mapsto\ang{21}\\
\ang{2}&\mapsto\ang{13}\conv\ang{213}^{-1}\\
\ang{3}&\mapsto\ang{23}
\end{aligned}
\hs{1.5ex}\begin{aligned}[t]
\ang{12}&\mapsto\ang{1}\\
\ang{21}&\mapsto\ang{1}\conv\ang{321}\conv\ang{213}^{-1}\\
\ang{23}&\mapsto\ang{3}\conv\ang{123}\conv\ang{213}^{-1}\\
\ang{32}&\mapsto\ang{3}
\end{aligned}
\hs{1.5ex}
\begin{aligned}[t]
\ang{132}&\mapsto\ang{213}\\
\ang{213}&\mapsto\ang{123}\conv\ang{321}\conv\ang{213}^{-1}\\
\ang{123}&\mapsto\ang{123}\\
\ang{321}&\mapsto\ang{321}\\
\ang{2132}&\mapsto0
\end{aligned}
\eneqn

\ee

\subsubsection{$C_2$}
$\xymatrix@R=-1.5ex{{\conv}\ar@2{<-}[r]&{\conv}\\1&2}$
\bnum
\item
$i=1$.
\eqn
&&\begin{aligned}[t]
\ang{1}&\mapsto\ang{2}\conv\ang{12}^{-1}\\
\ang{2}&\mapsto\ang{1^22}\\
\ang{12}&\mapsto\ang{21^22}\conv\ang{12}^{-1}\\
\ang{21}&\mapsto\ang{12}\\
\end{aligned}
\hs{5ex}\begin{aligned}[t]
\ang{1^22}&\mapsto\ang{2}\conv\ang{21^22}\conv\ang{12}^{-2}\\
\ang{21^2}&\mapsto\ang{2}\\
\ang{121}&\mapsto0\\
\ang{21^22}&\mapsto\ang{21^22}
\end{aligned}
\eneqn

\item
$i=2$.

\eqn
&&\begin{aligned}[t]
\ang{1}&\mapsto\ang{21}\\
\ang{2}&\mapsto\ang{1^2}\conv\ang{21^2}^{-1}\\
\ang{12}&\mapsto\ang{1}\\
\ang{21}&\mapsto\ang{1}\conv\ang{121}\conv\ang{21^2}^{-1}
\end{aligned}
\hs{5ex}\begin{aligned}[t]
\ang{1^22}&\mapsto\ang{21^2}\\
\ang{21^2}&\mapsto\ang{121}^2\conv\ang{21^2}^{-1}\\
\ang{121}&\mapsto\ang{121}\\
\ang{21^22}&\mapsto0
\end{aligned}
\eneqn

\ee

\subsection{$\F_i\cl \tC_{w_0}\to\tC_{w_0}$}

\subsubsection{$A_2$}

$i=12$, i.e.\ $\F_1\F_2$:
\eqn
&&\begin{aligned}[t]
\ang{1}&\mapsto\D\ang{2}=\ang{1}\conv\ang{21}^{-1}\\
\ang{2}&\mapsto\D\ang{1}=\ang{2}\conv\ang{12}^{-1}
\end{aligned}
\hs{7ex}\begin{aligned}[t]
\ang{12}&\mapsto\ang{21}^{-1}\\
\ang{21}&\mapsto\ang{12}^{-1}
\end{aligned}
\eneqn

\subsubsection{$A_3$}

\bnum
\item
$i=2$.

\scalebox{.85}{
\parbox{\textwidth}{
\eqn
&&\begin{aligned}[t]
\ang{1}&\mapsto\ang{21}\\
\ang{2}&\mapsto\ang{132}\conv\ang{2132}^{-1}\\
\ang{3}&\mapsto\ang{23}
\end{aligned}
\hs{1.5ex}\begin{aligned}[t]
\ang{12}&\mapsto\ang{1}\\
\ang{21}&\mapsto\ang{12}\conv\ang{321}\conv\ang{2132}^{-1}\\
\ang{23}&\mapsto\ang{32}\conv\ang{123}\conv\ang{2132}^{-1}\\
\ang{32}&\mapsto\ang{3}
\end{aligned}
\hs{1.5ex}
\begin{aligned}[t]
\ang{132}&\mapsto\ang{213}\\
\ang{213}&\mapsto\ang{2}\conv\ang{123}\conv\ang{321}\conv\ang{2132}^{-1}\\
\ang{123}&\mapsto\ang{123}\\
\ang{321}&\mapsto\ang{321}\\
\ang{2132}&\mapsto\ang{123}\conv\ang{321}\conv\ang{2132}^{-1}
\end{aligned}
\eneqn
}}

\item
$i=13$

\scalebox{.85}{
\parbox{\textwidth}{
\eqn
&&\begin{aligned}[t]
\ang{1}&\mapsto\D{1}=\ang{23}\conv\ang{123}^{-1}\\
\ang{2}&\mapsto\ang{132}\\
\ang{3}&\mapsto\D{3}=\ang{21}\conv\ang{321}^{-1}
\end{aligned}
\hs{1.5ex}\begin{aligned}[t]
\ang{12}&\mapsto\ang{3}\conv\ang{2132}\conv\ang{123}^{-1}\\
\ang{21}&\mapsto\ang{32}\\
\ang{23}&\mapsto\ang{12}\\
\ang{32}&\mapsto\ang{1}\conv\ang{2132}\conv\ang{321}^{-1}
\end{aligned}
\hs{1.5ex}
\begin{aligned}[t]
\ang{132}&\mapsto\ang{213}\conv\ang{2132}\conv\ang{123}^{-1}\conv\ang{321}^{-1}\\
\ang{213}&\mapsto\ang{2}\\
\ang{123}&\mapsto\ang{2132}\conv\ang{123}^{-1}\\
\ang{321}&\mapsto\ang{2132}\conv\ang{321}^{-1}\\
\ang{2132}&\mapsto\ang{2132}
\end{aligned}
\eneqn
}}
\ee

\subsubsection{$C_2$}
$\xymatrix@C=4ex@R=-1.5ex{{\conv}\ar@2{<-}[r]&{\conv}\\1&2}$
\bnum
\item
$i=1$.
\eqn
&&\begin{aligned}[t]
\ang{1}&\mapsto\ang{21}\conv\ang{121}^{-1}\\
\ang{2}&\mapsto\ang{1^22}\\
\ang{12}&\mapsto\ang{1}\conv\ang{21^22}\conv\ang{121}^{-1}\\
\ang{21}&\mapsto\ang{12}\\
\end{aligned}
\hs{5ex}\begin{aligned}[t]
\ang{1^22}&\mapsto\ang{21^2}\conv\ang{21^22}\conv\ang{121}^{-2}\\
\ang{21^2}&\mapsto\ang{2}\\
\ang{121}&\mapsto\ang{21^22}\conv\ang{121}^{-1}\\
\ang{21^22}&\mapsto\ang{21^22}
\end{aligned}
\eneqn

\item
$i=2$.

\eqn
&&\begin{aligned}[t]
\ang{1}&\mapsto\ang{21}\\
\ang{2}&\mapsto\ang{1^22}\conv\ang{21^22}^{-1}\\
\ang{12}&\mapsto\ang{1}\\
\ang{21}&\mapsto\ang{12}\conv\ang{121}\conv\ang{21^22}^{-1}
\end{aligned}
\hs{5ex}\begin{aligned}[t]
\ang{1^22}&\mapsto\ang{21^2}\\
\ang{21^2}&\mapsto\ang{2}\conv\ang{121}^2\conv\ang{21^22}^{-1}\\
\ang{121}&\mapsto\ang{121}\\
\ang{21^22}&\mapsto\ang{121}^2\conv\ang{21^22}^{-1}
\end{aligned}
\eneqn

\noi
$(\refl_1\circ\refl_2)^2\simeq(\refl_2\circ\refl_1)^2$.
\ee


\begin{thebibliography}{99}


\bibitem{CP94} V.~Chari and A.~Pressley, {\em A Guide to Quantum Groups}, Cambridge University Press, Cambridge, 1994.

\bibitem{CG97} N.~ Chriss and   V. ~Ginzburg,
{\em Representation theory and complex geometry},
Birkh{\"a}user Boston, Inc., Boston, MA, 1997.

\bibitem{EGNO15}
P.  Etingof, S. Gelaki, D. Nikshych and V. Ostrik,
{\em Tensor Categories}, 
Mathematical Surveys and Monographs, 
{\bf 205}, American Mathematical Society, Providence, RI, 2015. xvi+343 pp. 


\bibitem{EM03} P.~Etingof and  A.~A.~Moura, {\it Elliptic central characters and blocks of finite dimensional representations of quantum affine algebras}, Represent. Theory \textbf{7} (2003), 346--373. 


\bibitem{Kas02} M. Kashiwara, {\it On level zero representations of quantum affine algebras}, Duke. Math. J. {\bf112} (2002), 117--175.


\bibitem{KKKO15}
S.-J. Kang, M. Kashiwara,  M. Kim  and   S.-j. Oh,
\newblock{\em Simplicity of heads and socles of tensor products},
Compos. Math. \textbf{151} (2015), no. 2, 377--396.

\bibitem{K^3}
S.-J. Kang, M. Kashiwara and M. Kim, {\em Symmetric quiver
Hecke algebras and R-matrices of quantum affine algebras},
Invent. math. {\bf 211} (2018), 591--685,
arXiv:1304.0323 v3.




\bibitem{KKKO18}
\bysame,
\newblock{\em Monoidal categorification of cluster algebras},
J. Amer. Math. Soc. \textbf{31} (2018), no. 2, 349--426.

\bibitem{KKOP18}
M.~Kashiwara, M. Kim, S.-j. Oh, and  E.~Park,
\newblock{\em Monoidal categories associated with strata of flag manifolds},
Adv. Math. \textbf{328} (2018), 959--1009.

\bibitem{KKOP20} \bysame,  \newblock{\em Monoidal categorification and quantum affine algebras}, Compos. Math. {\bf 156} (2020), no. 5, 1039--1077. 


\bibitem{KKOP21}
\bysame, \newblock{\em Localizations for quiver Hecke algebras}, Pure Appl. Math. Q. \textbf{17} (2021), no. 4, 1465--1548.

\bibitem{KKOP22}
\bysame, 
 \newblock{\em Localizations for quiver Hecke algebras II},
Proc. London Math. Soc. (4) \textbf{127} (2023) 1134--1184.


\bibitem{KP18}
M.~Kashiwara and E.~Park, \newblock{\em Affinizations and $R$-matrices for quiver Hecke algebras},
J. Eur. Math. Soc. \textbf{20}, (2018), 1161--1193.


\bibitem{KS06} M.~Kashiwara and P.~Schapira,
\emph{Categories and Sheaves},
Grundlehren der Mathematischen Wissenschaften, vol.~\textbf{332}, Springer-Verlag, Berlin (2006).

\bibitem{Kato14}S.~Kato,
{\em Poincar\'e-Birkhoff-Witt bases and Khovanov-Lauda-Rouquier algebras}
Duke Math. J. {\bf163} (2014), no. 3, 619--663. 

\bibitem{Kato20}\bysame,
{\em On the monoidality of Saito reflection functors},
Int. Math. Res. Not. 2020, no. 22, 8600--8623. 


\bibitem{KL09}
M.~Khovanov and A. Lauda, \emph{A diagrammatic approach to
categorification of quantum groups
  {I}}, Represent. Theory \textbf{13} (2009), 309--347.

\bibitem{KL11}
\bysame, \emph{A diagrammatic approach to categorification of
  quantum groups {II}}, Trans. Amer. Math. Soc. \textbf{363} (2011),
  2685--2700.

\bibitem{LV11}
A.~Lauda and M.~Vazirani, \emph{Crystals from categorified quantum groups},
  Adv. Math. \textbf{228} (2011), no.~2, 803--861.
  
  \bibitem{Lu93}
  G.~Lusztig, {\em Introduction to Quantum Groups}, Progr. Math. \textbf{110}  Birkh\"{a}user, 1993.

\bibitem{McNamara15}  P.  J.~ McNamara,  \emph{Finite dimensional representations of Khovanov-Lauda-Rouquier algebras I: Finite type}
J. reine angew. Math. {\bf707} (2015), 103--124.

\bibitem{McNamara17}  \bysame, \emph{Monoidality of Kato's reflection functors}, arXiv:1712.00173v1.


\bibitem{Rouquier08}
R. Rouquier,  \emph{2-Kac–Moody algebras},  arXiv:0812.5023v1.

\bibitem{Saito}
Y. Saito, 
{\em PBW basis of quantized universal enveloping algebras},
Publ. Res. Inst. Math. Sci. {\bf30}, (1994), no. 2, 209--232. 

\bibitem{TW16}
P. Tingley and B. Webster,
{\em Mirkovi\'c-Vilonen polytopes and Khovanov-Lauda-Rouquier algebras},
Compos. Math. {\bf 152} (2016), no. 8, 1648--1696.


















\end{thebibliography}
\end{document}